\documentclass[12pt]{amsart}

\usepackage[matrix,arrow,curve,frame]{xy}
\usepackage{amsmath,amsthm,amssymb,enumerate}
\usepackage{latexsym}
\usepackage{amscd}
\usepackage[colorlinks=false]{hyperref}
\usepackage{euscript}
\usepackage{appendix}

\newtheorem{thm}{Theorem}[section]
\newtheorem{prop}{Proposition}[section]
\newtheorem{cor}{Corollary}[section]
\newtheorem{lemma}{Lemma}[section]

\theoremstyle{definition}
\newtheorem{defn}{Definition}[section]
\newtheorem{conj}{Conjecture}[section]
\newtheorem{example}{Example}[section]

\theoremstyle{remark}
\newtheorem{remark}{Remark}[section]

\numberwithin{equation}{section}

\def\A{{\bf A}}

\def\cA{{\mathcal A}}
\def\cB{{\mathcal B}}
\def\cC{{\mathcal C}}
\def\cD{{\mathcal D}}
\def\cE{{\mathcal E}}
\def\cF{{\mathcal F}}

\def\cH{{\mathcal H}}
\def\cI{{\mathcal I}}

\def\cO{{\mathcal O}}

\def\cR{{\mathcal R}}
\def\cS{{\mathcal S}}

\def\cV{{\mathcal V}}
\def\cW{{\mathcal W}}

\def\cY{{\mathcal Y}}

\def\ga{{\mathfrak a}}
\def\gb{{\mathfrak b}}

\def\gd{{\mathfrak d}}

\def\gg{{\mathfrak g}}
\def\gh{{\mathfrak h}}

\def\gl{{\mathfrak l}}

\def\gn{{\mathfrak n}}
\def\go{{\mathfrak o}}
\def\gp{{\mathfrak p}}

\def\gs{{\mathfrak s}}

\begin{document}

\pagestyle{plain}

\title{Trialities of $\cW$-algebras}

\author{Thomas Creutzig} 
\address{University of Alberta}
\email{creutzig@ualberta.ca}

\author{Andrew R. Linshaw} 
\address{University of Denver}
\email{andrew.linshaw@du.edu}
\thanks{T. C. is supported by NSERC Discovery Grant \#RES0019997. A. L. is supported by Simons Foundation Grant \#635650 and NSF Grant DMS-2001484. We thank T. Arakawa, B. Feigin, M. Rap\v{c}\'ak, and T. Proch\'azka for many illuminating discussions on topics related to this paper.}

\begin{abstract} We prove the conjecture of Gaiotto and Rap\v{c}\'ak that the $Y$-algebras $Y_{L,M,N}[\psi]$ with one of the parameters $L,M,N$ zero, are simple one-parameter quotients of the universal two-parameter $\cW_{1+\infty}$-algebra, and satisfy a symmetry known as triality. These $Y$-algebras are defined as the cosets of certain non-principal $\cW$-algebras and $\cW$-superalgebras by their affine vertex subalgebras, and triality is an isomorphism between three such algebras. Special cases of our result provide new and unified proofs of many theorems and open conjectures in the literature on $\cW$-algebras of type $A$. This includes (1) Feigin-Frenkel duality, (2) the coset realization of principal $\cW$-algebras due to Arakawa and us, (3) Feigin and Semikhatov's conjectured triality between subregular $\cW$-algebras, principal $\cW$-superalgebras, and affine vertex superalgebras, (4) the rationality of subregular $\cW$-algebras due to Arakawa and van Ekeren, 
and (5) the identification of Heisenberg cosets of subregular $\cW$-algebras with principal rational $\cW$-algebras that was conjectured in the physics literature over 25 years ago. Finally, we prove the conjectures of Proch\'azka and Rap\v{c}\'ak on the explicit truncation curves realizing the simple $Y$-algebras as $\cW_{1+\infty}$-quotients, and on their minimal strong generating types.
\end{abstract}

\maketitle

\section{Introduction} \label{section:intro}

Let $\gg$ be a simple Lie (super)algebra with a nondegenerate invariant bilinear form, and let $f$ be a nilpotent element in the even part of $\gg$. To this data and any complex number $k$, one associates the universal affine vertex superalgebra $V^k(\gg)$ at level $k$, and a complex $V^k(\gg) \otimes C_f$ whose homology is the universal $\cW$-superalgebra $\cW^k(\gg, f)$. Historically, the best-studied cases have $f = 0$ so that $\cW^k(\gg, 0) = V^k(\gg)$, and $f$ the principal nilpotent so that $\cW^k(\gg, f)$ is the principal $\cW$-algebra $\cW^k(\gg)$. However, the recent connection of vertex algebras to geometry, topology and higher dimensional physics involves many different types of $\cW$-superalgebras. Before explaining this in more detail we begin by stating our main results.

\subsection{Main Theorem}

We first define the two families of algebras of interest. 
Let $\gg = \gs\gl_{n+m}$ and $f_{n, m}$ be the nilpotent element corresponding to the partition $(n, 1, \dots, 1)$ of $n+m$. Let $\psi = k +n+m$. For $n+m \geq 1$ and $n>0$, we define $\cW^{\psi}(n,m):=\cW^k(\gs\gl_{n+m}, f_{n, m})$. For $m\geq 2$ and $n=0$, we define $\cW^\psi(0,m) = V^{k}(\gs\gl_m) \otimes \cS(m)$ where $\cS(m)$ is the rank $m$ $\beta\gamma$-system. In the cases $n=0$ and $m=0, 1$ we define $\cW^\psi(0,1) = \cS(1)$ and $\cW^\psi(0,0) = \mathbb{C}$. The best known cases are 
\begin{enumerate}
\item the principal $\cW$-algebra $\cW^k(\gs\gl_n) \cong \cW^\psi(n,0)$,
\item the subregular $\cW$-algebra $\cW^k(\gs\gl_{n+1}, f_{\text{subreg}}) \cong \cW^\psi(n,1)$,
\item the affine vertex algebra $V^k(\gs\gl_{m+1}) \cong \cW^k(\gs\gl_{m+1}, 0) \cong \cW^\psi(1, m)$,
\item the minimal $\cW$-algebra $\cW^k(\gs\gl_{m+2}, f_{\text{min}}) \cong \cW^\psi(2,m)$.
\end{enumerate}

For $m\geq 1$, $\cW^{\psi}(n,m)$ has affine subalgebra $V^{\psi-m-1}(\gg\gl_m)$. We set
\[
\cC^\psi(n, m) :=  \begin{cases} \text{Com}\left(V^{\psi-m-1}(\gg\gl_m), \cW^{\psi}(n,m)\right) & \quad  \text{for} \ m\geq 1, \\ 
\cW^\psi(n,0) & \quad \text{for} \ m=0 .\end{cases}
\]
Next, we consider $\gg = \gs\gl_{n|m}$ and the nilpotent element $f_{n|m}$ corresponding to the super partition $(n| 1,\dots, 1)$ of $n|m$. Let $\psi = k +n-m$.
For $n+m \geq 2$ and $n\neq m$, we define $\cV^{\psi}(n,m):=\cW^k(\gs\gl_{n|m}, f_{n|m})$.  The case $n=m$ and $n\geq 2$ is special since $\gs\gl_{n|n}$ is not simple, and we use its simple quotient $\gp\gs\gl_{n|n}$ instead; that is, $\cV^{\psi}(n,n):=\cW^k(\gp\gs\gl_{n|n}, f_{n|n})$. We set $\cV^{\psi}(1,1) \cong \cA(1)$, where $\cA(1)$ is the rank one symplectic fermion algebra. For $m\geq 2$ and $n=0$, we set $\cV^{\psi}(0,m) \cong V^{-k}(\gs\gl_m) \otimes \cE(m)$ where $\cE(m)$ denotes the rank $m$ $bc$-system. Finally, $\cV^{\psi}(0,1) \cong \cE(1)$ and $\cV^{\psi}(0,0) \cong \cV^{\psi}(1,0)\cong \mathbb{C}$. 
The best known cases are 
\begin{enumerate}
\item  the principal $\cW$-algebra $\cW^k(\gs\gl_n) \cong \cV^\psi(n,0)$,
\item the principal $\cW$-superalgebra $\cW^k(\gs\gl_{n|1}, f_{n|1}) \cong \cV^\psi(n,1)$,
\item the affine vertex superalgebra $V^{k}(\gs\gl_{1|m}) \cong \cW^k(\gs\gl_{1|m}, 0) \cong \cV^\psi(1, m)$,
\item the minimal $\cW$-superalgebra $\cW^k(\gs\gl_{2|m}, f_{\text{min}}) \cong \cV^\psi(2,m)$ for $m\neq 2$.
\end{enumerate}

For $n\neq m$ and $m\geq 1$, $\cV^{\psi}(n,m)$ has affine subalgebra $V^{-\psi-m+1}(\gg\gl_m)$. For $n = m \geq 2$, $\cV^{\psi}(n,n)$ has affine subalgebra $V^{-\psi - n+1}(\gs\gl_n)$. Note that for all $n\geq 1$, $\cV^{\psi}(n,n)$ also admits an action of $\text{GL}_1$ by outer automorphisms. We set
\[
\cD^\psi(n, m) :=  \begin{cases} \text{Com}\left(V^{-\psi-m+1}(\gg\gl_m), \cV^{\psi}(n,m)\right) & \quad  \text{for} \ n\neq m\ \text{and} \ m\geq 1, \\ 
 \text{Com}\left(V^{-\psi-n+1}(\gs\gl_n), \cV^{\psi}(n,n)\right)^{\text{GL}_1} & \quad  \text{for} \ n = m \ \text{and} \ n\geq 2, 
 \\ \cA(1)^{\text{GL}_1} & \quad  \text{for} \ n = m =1,
 \\ 
\cV^\psi(n,0) & \quad \text{for} \ m=0. \end{cases}
\]
A {\it one-parameter vertex algebra} is a vertex algebra over some localization of a polynomial ring in one variable, which in our case is the level $k$ or equivalently the critically shifted level $\psi$. The main result of this paper is
\begin{thm} \label{intro:mainthm} 
Let $n\geq  m$ be non-negative integers. As one-parameter vertex algebras
\[
\cD^\psi(n, m)  \cong \cC^{\psi^{-1}}(n-m, m) \cong \cD^{\psi'}(m, n)
\]
with $\psi'$ defined by $\frac{1}{\psi} +\frac{1}{\psi'} =1$. 
\end{thm}
This theorem was first conjectured by Gaiotto and Rap\v{c}\'ak \cite{GR} in a slightly different form (see below), and has the following special cases.
\begin{enumerate}
\item $\cD^{\psi}(1,1) = \cA(1)^{\text{GL}_1}$ is known in the logarithmic conformal field theory literature as the $p=2$ singlet vertex algebra \cite{AdaI, Kau}. It is isomorphic to $\cC^{\psi^{-1}}(0, 1)$ which is the Heisenberg coset of $\cS(1)$. This isomorphism has been discussed in \cite{CRII}.
\item Feigin-Frenkel duality says that the principal $\cW$-algebra of a simple Lie algebra $\gg$ at level $\psi - h^\vee$ is isomorphic to the principal $\cW$-algebra of the dual Lie algebra ${}^L\gg$ at level $\psi' - {}^Lh^\vee$ where $\ell \psi \psi'=1$ and $\ell$ is the lacity of $\gg$ \cite{FFII}. The special case 
$\cD^\psi(n, 0) \cong \cC^{\psi^{-1}}(n, 0)$ of our result reproduces Feigin-Frenkel duality in type $A$. 
\item Principal $\cW$-algebras of simply-laced Lie algebras can be realized as cosets \cite{ACLII}. For type $A$ this is the case $\cD^{\psi}(n, 0) \cong \cD^{\psi'}(0, n)$. 
\item Feigin and Semikhatov conjectured relations between cosets of subregular $\cW$-algebras of $\gs\gl_n$, of principal $\cW$-superalgebras of $\gs\gl_{n|1}$ and cosets of affine vertex superalgebras of $\gs\gl_{n|1}$ \cite{FS}. These correspond to the case $m=1$ of our theorem, that is $\cD^\psi(n, 1)  \cong \cC^{\psi^{-1}}(n-1, 1) \cong \cD^{\psi'}(1, n)$. The first isomorphism recovers a theorem of Genra, Nakatsuka and one of us \cite{CGN}. 
\end{enumerate}

\subsection{Outline of proof}
The proof of Theorem \ref{intro:mainthm} has three parts. 
\begin{enumerate}
\item We show that $\cC^{\psi}(n,m)$ and $\cD^{\psi}(n,m)$ are simple as one-parameter vertex algebras; equivalently, they are simple for generic values of $\psi$.
\item We find minimal strong generating sets for $\cC^{\psi}(n,m)$ and $\cD^{\psi}(n,m)$. 
\item We show that aside from the extreme cases $\cC^{\psi}(0,0)$, $\cC^{\psi}(1,0)$, $\cC^{\psi}(2,0)$, and $\cD^{\psi}(0,0)$, $\cD^{\psi}(0,1)$, $\cD^{\psi}(1,0)$,  $\cD^{\psi}(2,0)$, for all other values of $n,m$, $\cC^{\psi}(n,m)$ and $\cD^{\psi}(n,m)$ are one-parameter quotients of a universal two-parameter vertex algebra. Its simple one-parameter quotients are in bijection with a family of curves in the parameter space $\mathbb{C}^2$, and we finish the proof by explicitly describing these curves. The extreme cases are easily verified separately.
\end{enumerate}

\subsection{Basic results on $\cW$-superalgebras}
In order to carry out steps (1) and (2) above, we shall establish some foundational results on the structure of $\cW$-superalgebras in Section \ref{section:walgebras}. First, we introduce a general notion of {\it free field algebra}. This is a vertex superalgebra that is strongly generated by fields whose OPEs contain no other field than the vacuum. Examples of free field algebras that have a conformal structure are free fermions, symplectic fermions, the Heisenberg and $\beta\gamma$-vertex algebra. However there are many more examples that do not have a conformal structure. 

It is well known that $\cW$-superalgebras allow for a quasi-classical limit in which they become commutative and can be endowed with a Poisson vertex superalgebra structure; see e.g. \cite{DSKII}. We modify these constructions to obtain the free field limits for vertex superalgebras that allow a quasi-classical limit; see Proposition \ref{prop:ff}. A consequence is that the free field limit is simple if and only if the corresponding quasi-classical limit has a nondegenerate pairing on the strong generators; see Corollary \ref{cor:GFFsimple}. The main result here is Theorem \ref{thm:nondegeneracyw}, which says that if $\gg$ is a Lie superalgebra with a nondegenerate invariant bilinear form, and $f\in \gg$ is a nilpotent element, the $\cW$-superalgebra $\cW^k(\gg,f)$ has a simple free field limit. As a corollary, we obtain the simplicity of $\cW^k(\gg,f)$ for generic values of $k$ in full generality. This was previously known only for principal $\cW$-algebras \cite{ArIII}, and for minimal $\cW$-algebras and $\cW$-superalgebras \cite{GK,HR}. Finally, recall that $\cW^k(\gg,f)$ has affine vertex subalgebra $V^{\ell}(\ga)$ where $\ga \subseteq \gg$ is the centralizer of the $\gs\gl_2$-triple extending $f$. The generic simplicity of $\cW^k(\gg,f)$ implies the generic simplicity of its coset $\text{Com}(V^{\ell'}(\gb), \cW^k(\gg,f))$, where $V^{\ell'}(\gb) \subseteq V^{\ell}(\ga)$ is the affine vertex algebra corresponding to any reductive Lie subalgebra $\gb \subseteq \ga$.

Minimal strong generating sets for a large class of orbifolds of vertex algebras, as well as cosets of affine vertex algebras in certain larger structures, can be studied by passing to an orbifold problem of a suitable limit \cite{CLIII}. In Section \ref{section:orbifoldsandcosets}, we adapt this picture to study orbifolds and cosets of $\cW$-superalgebras by passing to their free field limits. Theorem \ref{thm:sfgorbandcoset} says that for any simple Lie superalgebra $\gg$ and nilpotent $f$, any reductive group $G$ of automorphisms of $\cW^k(\gg,f)$, and any affine subalgebra $V^{\ell'}(\gb) \subseteq V^{\ell}(\ga) \subseteq \cW^k(\gg,f)$ where $\gb$ is reductive, the orbifold $\cW^k(\gg,f)^G$ and the coset $\text{Com}(V^{\ell'}(\gb), \cW^k(\gg,f))$ are strongly finitely generated for generic values of $k$. This result is constructive modulo a classical invariant theory problem, namely, the first and second fundamental theorems of invariant theory for some reductive group $G$ and finite-dimensional $G$-module $V$. For the cosets $\cC^{\psi}(n,m)$ and $\cD^{\psi}(n,m)$ appearing in Theorem \ref{intro:mainthm}, $G = \text{GL}_m$ and $V$ is the standard module $\mathbb{C}^m$ plus its dual. Using Weyl's first and second fundamental theorems of invariant theory in this case \cite{W}, we give explicit minimal strong generating sets for these cosets; see Lemmas \ref{lemma:stronggencnm} and \ref{lemma:stronggendnm}.

In Section \ref{section:walgebras}, we also prove some basic results on principal $\cW$-algebras of type $A$, including the weight where the first singular vector appears in the universal $\cW$-algebra for all nondegenerate admissible levels; see Corollary \ref{wprinsingular}. This is surely known to experts but we could not find it in the literature, and it is needed in our proof of Theorem \ref{intro:mainthm}.

\subsection{The $\cW_\infty$-algebra} The last step in the proof of Theorem \ref{intro:mainthm} is to identify both $\cC^{\psi}(n,m)$ and $\cD^{\psi}(n,m)$ explicitly as one-parameter quotients of the universal $\cW_{\infty}$-algebra of type $\cW(2,3,\dots)$; see Theorems \ref{c(n,m)asquotient} and \ref{d(n,m)asquotient}. The existence and uniqueness of a two-parameter vertex algebra $\cW_{\infty}[\mu]$ which interpolates between $\cW^k(\gs\gl_n)$ for all $n$ was conjectured for many years in the physics literature \cite{YW,BK,B-H,BS,GGI,GGII,ProI,ProII,PRI,PRII}, and was recently proven by one of us in \cite{LVI}. In particular, the structure constants are continuous functions of the central charge $c$ and the parameter $\mu$, and if we set $\mu = n$, there is a truncation at weight $n+1$ such that the simple quotient is isomorphic to $\cW^k(\gs\gl_n)$ as a one-parameter vertex algebra. In the quasi-classical limit, the existence of a Poisson vertex algebra of type $\cW(2,3,\dots)$ which interpolates between the classical $\cW$-algebras of $\gs\gl_n$ for all $n$, has been known for many years; see \cite{KZ,KM,DSKV}. 

We mention that $\cW_{\infty}[\mu]$ acquires better properties if it is tensored with a rank one Heisenberg algebra $\cH$ to obtain the universal $\cW_{1+\infty}$-algebra. This vertex algebra is closely related to a number of other algebraic structures that arise in very different contexts. For example, up to suitable completions its associative algebra of modes is isomorphic to the Yangian of $\widehat{\gg\gl_1}$ \cite{AS,MO,Ts}, as well as the algebra ${\bf SH}^c$ defined in \cite{SV} as a certain limit of degenerate double affine Hecke algebras of $\mathfrak{gl}_n$. This identification allowed  Schiffmann and Vasserot to define an action of the principal $\cW$-algebra of $\gg\gl_r$ on the equivariant cohomology of the moduli space of $U_r$-instantons in \cite{SV}.

In \cite{LVI} we used a different parameter $\lambda$ which is related to $\mu$ by 
\begin{equation} \lambda =  \frac{(\mu-1) (\mu+1)}{(\mu-2) (3 \mu^2  - \mu -2+ c (\mu + 2))},\end{equation} and we denoted the universal algebra by $\cW(c,\lambda)$. Instead of using either the primary strong generating fields, or the quadratic basis of \cite{ProI}, our strong generators are defined as follows. We begin with the primary weight $3$ field $W^3$, normalized so that $W^3_{(5)} W^3 = \frac{c}{3} 1$, and we define the remaining fields recursively by $W^i= W^3_{(1)} W^{i-1}$ for $i\geq 4$. With this choice, the rich connections between the representation theory of $\cW_{\infty}[\mu]$ and the combinatorics of box partitions are not apparent. However, our choice has the advantage that the recursive behavior of the OPE algebra is more transparent. 

In addition to $\cW^k(\gs\gl_n)$, $\cW(c,\lambda)$ admits many other one-parameter quotients as well. In fact, any one-parameter vertex algebra of type $\cW(2,3,\dots, N)$ for some $N$ satisfying mild hypotheses, arises as such a quotient, so $\cW(c,\lambda)$ can be viewed as a classifying object for such vertex algebras. The simple one-parameter quotients are in bijection with a family of plane curves called {\it truncation curves}, but the explicit description of all such curves is still an open problem.

The minimal strong generating sets for $\cC^{\psi}(n,m)$ and $\cD^{\psi}(n,m)$ given by Lemmas \ref{lemma:stronggencnm} and \ref{lemma:stronggendnm} imply that they are at worst extensions of one-parameter quotients $\tilde{\cC}^{\psi}(n,m)$ and $\tilde{\cD}^{\psi}(n,m)$ of $\cW(c,\lambda)$, respectively. We can therefore regard $\cW^{\psi}(n,m)$ as an extension of $V^{\psi-m-1}(\gg\gl_m) \otimes \cW$ where $\cW$ is some one-parameter quotient of $\cW(c,\lambda)$, and the extension contains $2m$ even primary fields in weight $\frac{n+1}{2}$ which transform under $\gg\gl_m$ as $\mathbb{C}^m \oplus (\mathbb{C}^m)^*$. By imposing just seven Jacobi identities, we will prove that the truncation curve for $\cW$ is uniquely and explicitly determined by the existence of this extension. Similarly, we regard $\cV^{\psi}(n,m)$ as an extension of $V^{-\psi-m+1}(\gg\gl_m) \otimes \cW$ in the case $n\neq m$, and of $V^{-\psi-n+1}(\gs\gl_n) \otimes \cW$ in the case $n=m$, where the extension contains $2m$ odd primary fields of weight $\frac{n+1}{2}$ which transform under $\gg\gl_m$ (or $\gs\gl_n$ in the case $n=m$) as $\mathbb{C}^m \oplus (\mathbb{C}^m)^*$. The same procedure shows that the truncation curve for $\cW$ is uniquely determined. These curves allow us to find isomorphisms between the simple quotients $\tilde{\cC}_{\psi}(n,m)$ and $\tilde{\cD}_{\psi}(n,m)$ and certain principal $\cW$-algebras of type $A$, at special values of $\psi$. Using the weights of singular vectors in these $\cW$-algebras given by Corollary \ref{wprinsingular}, we prove that $\cC^{\psi}(n,m) = \tilde{\cC}^{\psi}(n,m)$ and $\cD^{\psi}(n,m) = \tilde{\cD}^{\psi}(n,m)$. Our main result then follows from our explicit truncation curves, together with the generic simplicity of $\cC^{\psi}(n,m)$ and $\cD^{\psi}(n,m)$.

\subsection{$Y$-algebras and triality}
Motivated from physics, Gaiotto and Rap\v{c}\'ak introduced a family of vertex algebras $Y_{L,M,N}[\psi]$ called  $Y$-algebras \cite{GR}. They considered interfaces of GL-twisted $\mathcal N=4$ supersymmetric 
gauge theories with gauge groups $U(L), U(M), U(N)$. The shape of these interfaces is a $Y$ and local operators at the corner of these interfaces are supposed to form a vertex algebra, hence the name $Y$-algebra. Also note that GL stands for geometric Langlands. These interfaces should satisfy a permutation symmetry which then induces a corresponding symmetry on the associated vertex algebras. This led \cite{GR} to conjecture a triality of isomorphisms of $Y$-algebras. 

The $Y$-algebras were also conjectured in \cite{GR} to arise as one-parameter quotients of the universal two-parameter $\cW_{1+\infty}$-algebra, which is just the tensor product $\cH \otimes \cW(c,\lambda)$. In fact, the distinct truncation curves of $\cW(c,\lambda)$ are expected to be in bijection with the algebras $Y_{0,M,N}[\psi]$ algebras, which are the simple quotients of $\cH \otimes \cW(c,\lambda)$ along these curves, and $Y(r,M+r,N+r)[\psi]$ is expected to be a non-simple quotient of $\cH \otimes \cW(c,\lambda)$ along the same curve. In \cite{PRI}, Proch\'azka and Rap\v{c}\'ak conjectured a precise formula for these truncation curves, and in \cite{PRII} they conjectured that $Y_{L,M,N}[\psi]$ should have minimal strong generating type $\cW(1,2,3,\dots, n)$ for $n = (L+1)(M+1)(N+1) - 1$.

The $Y$-algebras with one label being zero are up to a Heisenberg algebra our coset vertex algebras. 
More precisely,
\begin{equation}
\begin{split} 
Y_{0, M, N}[\psi] &= \cC^\psi(N-M, M) \otimes \cH , \qquad M \leq N, \\
Y_{0, M, N}[\psi] &= \cC^{-\psi+1}(M-N, N)  \otimes \cH, \qquad M > N, \\
Y_{L, 0, N}[\psi] &= \cD^\psi(N, L) \otimes \cH, \\
Y_{L, M, 0}[\psi] &= \cD^{-\psi+1}(M, L) \otimes \cH.
\end{split}
\end{equation}
By definition one has $Y_{0, N, M}[\psi] \cong Y_{0, M, N}[1-\psi]$ for $N\neq M$. We also have $\cC^\psi(0, M) \cong  \cC^{1-\psi}(0, M)$ and hence this statement also holds for $N=M$. Clearly also $Y_{L, 0, N}[\psi] \cong Y_{L, N, 0}[1-\psi]$. Combining this with the isomorphisms in Theorem \ref{intro:mainthm}, we obtain
\begin{equation}
\begin{split}
Y_{0, M, N}[\psi] &\cong Y_{0, N, M}[1-\psi] \cong Y_{M, 0, N}[\psi^{-1}] \cong Y_{M, N, 0}[1-\psi^{-1}] \\
&\cong Y_{N, 0, M}[(1-\psi)^{-1}] \cong Y_{N, M, 0}[(1-\psi^{-1})^{-1}]. 
\end{split}
\end{equation}
Let $\psi$ be defined by 
\[
\psi = -\frac{\epsilon_2}{\epsilon_1}, \qquad \epsilon_1 + \epsilon_2 +\epsilon_3 =0
\]
and set 
\[
Y^{\epsilon_1, \epsilon_2, \epsilon_3}_{N_1, N_2, N_3} := Y_{N_1, N_2, N_3}[\psi].
\]
Then with this notation the triality symmetry is manifest
\[
Y^{\epsilon_{\sigma(1)}, \epsilon_{\sigma(2)}, \epsilon_{\sigma(3)}}_{N_{\sigma(1)}, N_{\sigma(2)}, N_{\sigma(3)}} \cong Y^{\epsilon_1, \epsilon_2, \epsilon_3}_{N_1, N_2, N_3} \qquad \text{for} \ \sigma \in S_3.
\]
In particular, as a corollary of Theorems \ref{intro:mainthm}, \ref{c(n,m)asquotient}, and \ref{d(n,m)asquotient}, we obtain
\begin{cor} The following conjectures are true.
\begin{enumerate} 
\item The conjecture \cite{GR} that $Y_{L,M,N}[\psi]$ is a simple quotient of $\cH\otimes \cW(c,\lambda)$, when one of the labels $L,M,N$ is zero.
\item The triality conjecture of \cite{GR} for the algebras $Y_{L,M,N}[\psi]$ when one of the labels is zero.
\item The formula \cite[Eq. 2.14]{PRI} for the truncation curve of $Y_{0,M,N}[\psi]$.
\item The conjecture of \cite{PRII} that $Y_{0,M,N}[\psi]$ is of type $\cW(1,2,3,\dots, (M-1)(N-1)-1)$.
\end{enumerate}
\end{cor}

The general case of $Y_{L,M,N}[\psi]$ where the three labels can all be nonzero, corresponds to cosets of $\cW$-superalgebras of type $A$ by affine vertex superalgebras. These cosets are also expected to be one-parameter quotients of the universal $\cW_{1+\infty}$-algebra, however they are not the simple quotients. One can study these cases by using the invariant theory of Lie superalgebras developed by Sergeev \cite{SI,SII}, to describe orbifolds of free field algebras under the corresponding supergroups. These orbifolds will then be suitable limits of $Y_{M, N, L}[\psi]$. It also seems possible to relate $Y_{L,M,N}[\psi]$ for nonzero $L,M,N$, to $Y_{L-r, N-r, M-r}[\psi]$ by a new variant of the Duflo-Serganova functor \cite{DuS, GS}, called cohomological reduction of affine superalgebras in physics \cite{CCMS}, for $\cW$-superalgebras. This is work in progress.

\subsection{Uniqueness and reconstruction of $\cW$-algebras}
For $m\geq 1$, $\cW^{\psi}(n,m)$ can be viewed as an extension of $V^{\psi-m-1}(\gg\gl_m) \otimes \cC^{\psi}(n,m)$, where the extension is generated by primary fields $\{P^{\pm, i}|\ i = 1,\dots, m\}$ of weight $\frac{n+1}{2}$, which transform as $\mathbb{C}^m \oplus (\mathbb{C}^m)^*$ under $\gg\gl_m$. These fields also satisfy the nondegeneracy condition $(P^{+,i})_{(n)} P^{-,j} = \delta_{i,j} 1$. Theorem \ref{thm:uniquenesswnm} says that $\cW^{\psi}(n,m)$ satisfies a strong {\it uniqueness property}: its full OPE algebra is completely determined by the structure of $\cC^{\psi}(n,m)$, the normalization of the Heisenberg field, the action of $\gg\gl_m$ on the fields $\{P^{\pm, i}\}$, and the above nondegeneracy condition. A similar uniqueness theorem holds for $\cV^{\psi}(n,m)$; see Theorem \ref{thm:uniquenesssvnm}.

There are certain special levels where the simple quotient $\cC_{\psi}(n, m)$ of $\cC^\psi(n,m)$ is isomorphic to a principal $\cW$-algebra $\cW_{r}(\gs\gl_s)$ for some $s \geq 3$, and these are classified by Corollary \ref{CWclassification}. The {\it reconstruction problem} at these levels is to consider the tensor product of $\cW_{r}(\gs\gl_s)$ with a homomorphic image of $V^{\psi-m-1}(\gg\gl_m)$, and try to realize $\cW_\psi(n, m)$ explicitly as an extension of this tensor product. This requires the above uniqueness theorem, and is easiest in the case $m=1$ since the extension is then expected to be of simple current type. Theorem \ref{thm:reconstructionm=1} shows that for $ \psi =  \frac{n + s+1}{n}$, if $s+1$ and $s+n+1$ are coprime, then the Heisenberg algebra $\cH \subseteq \cW_{\psi}(n,1)$ can be extended to a lattice vertex algebra $V_L \subseteq \cW_{\psi}(n,1)$ for $L = \sqrt{s(n+1)}\ \mathbb{Z}$, and $\cW_{\psi}(n,1)$ is a simple current extension of $V_L \otimes \cW_r(\gs\gl_s)$ for $ r = -s + \frac{s+1}{s+n+1}$. This proves an old conjecture of Blumenhagen et al \cite{B-H}, which was previously known only in the low rank cases of $r=2, 3, 4$ in \cite{ALY, ACLI, CLIV}. It gives a new and independent proof of Arakawa and van Ekeren's recent theorem that $\cW_k(\gs\gl_{n+1}, f_{\text{subreg}})$ is rational and lisse for these values of $k = \psi - n- 1$ \cite{AvEII}. Similarly, in Theorems \ref{thm:reconstructionm=2} and \ref{thm:reconstructionm=3}, we reconstruct $\cW_{\psi}(n,1)$ and $\cW_{\psi}(2,m)$ at certain levels where it is not rational or lisse. As a consequence, we obtain a vertex tensor category structure on the category of ordinary modules for certain affine vertex algebras at non-admissible levels.

\subsection{Outlook}
In this subsection, we list a few natural directions for future research.

\smallskip

\noindent {\bf Free field realization}. 
Theorem \ref{intro:mainthm} is a common generalization of Feigin-Frenkel duality in type $A$, and the coset realization theorem of \cite{ACLII} in type $A$, and provides new proofs of these results. Both of these results were proven originally using the free field realization of $\cW^{\ell}(\gs\gl_n)$ at generic level $\ell$ coming from the Miura map 
$$\gamma_{\ell}: \cW^{\ell}(\gs\gl_n) \hookrightarrow \pi,$$ where $\pi$ is the Heisenberg vertex algebra of rank $n-1$. This is obtained by applying the Drinfeld-Sokolov reduction functor to the Wakimoto free field realization $V^{\ell}(\gs\gl_n) \hookrightarrow M_{\gs\gl_n} \otimes \pi$, where $M_{\gs\gl_n}$ is the $\beta\gamma$-system of rank $\text{dim}(\gn_+)$, where $\gn_+$ denotes the upper nilpotent part of $\gs\gl_n$ \cite{F}. The difficult step of \cite{ACLII} is to construct another vertex algebra homomorphism $$\Psi_k: \text{Com}(V^{k+1}(\gs\gl_n),V^k(\gs\gl_n)\otimes L_1(\gs\gl_n)) \hookrightarrow \pi,\qquad  \ell +n = \frac{k+n}{k+n+1},$$ and show that its image coincides with the image of $\gamma_{\ell}$.

It is an important question whether this approach can be used to give an alternative proof of Theorem \ref{intro:mainthm}. In the case $m=1$, this was carried out by one of us together with Genra and Nakatsuka \cite{CGN}. We mention that a family of vertex algebras $\cW_{r_1, r_2, r_3}$ defined by free field realizations in \cite{BFM}, was conjectured by Proch\'azka and Rap\v{c}\'ak to be isomorphic to the $Y_{r_1,r_2, r_3}[\psi]$-algebras of \cite{GR}. These algebras manifestly satisfy the triality symmetry, so establishing their equivalence to the $Y$-algebras of \cite{GR} would provide another proof of our main result. In recent work of Rap\v{c}\'ak, Soibelman, Yang and Zhao \cite{RSYZ} which generalizes the results of \cite{SV,MO}, an action of $\cW_{r_1, r_2, r_3}$ on the equivariant cohomology of the moduli space of spiked instantons was constructed. They also conjecture the action of some vertex algebra for any toric Calabi-Yau threefold and there should be a gluing construction of $Y_{L, M, N}$-algebras that realizes these vertex algebras.

\smallskip

\noindent {\bf Reconstruction of $\cW$-algebras: general case}.  It would be very interesting to reconstruct {\it all} simple algebras $\cW_{\psi}(n,m)$ appearing in Corollary \ref{CWclassification} as extensions of type $A$ principal $\cW$-algebras times affine vertex algebras. In the first case, $\cW_r(\gs\gl_s)$ is lisse and rational as long as $m+s$ and $m+n+s$ are coprime, and the reconstruction problem can be approached using the theory of vertex algebra extensions \cite{CKMI, CKMII} once one understands fusion categories of type $A$ well enough. Note that the level of the affine subalgebra of $\cW_{\psi}(n,m)$ is admissible if $n$ is coprime to $m+s$.  It is not apparent that the simple affine vertex algebra $L_{\psi-m-1}(\gg\gl_m)$ embeds in $\cW_{\psi}(n,m)$ at these levels, but we expect this to be the case. Reconstructing $\cW_{\psi}(n,m)$ as an extension of $ L_{\psi-m-1}(\gg\gl_m)\otimes \cW_r(\gs\gl_s)$, would prove this conjecture.

Recall next that a level $k$ of $\gs\gl_n$ is called {\it boundary admissible} if $k = -n +\frac{n}{r}$ for some positive integer $r$ coprime with $n$. This case is special in the sense that the simple vertex algebra is the only simple ordinary module at this level \cite{ArV}. In the second case of Corollary \ref{CWclassification}, note that the level $k = \psi-n-m$ of $\cW_{\psi}(n,m)$ is boundary admissible for $\gs\gl_{n+m}$ when $s$ and $m$ are coprime. In the third case, the affine subalgebra of $\cW_{\psi}(n,m)$ has boundary admissible level if $n-s$ is coprime to $m$.

Vertex algebras can be associated to certain supersymmetric quantum field theories called Argyres-Douglas theories. These are labelled by pairs of Dynkin diagrams and the associated vertex algebra seems to usually be an extension of a $\cW$-algebra at boundary admissible level associated to the Lie algebras with corresponding Dynkin diagrams; see e.g. \cite{CS, CIII, XY}. Not all cases are understood yet in this context, but known ones of type $A$ seem to be covered by the second case of Corollary \ref{CWclassification}.

Another interesting series of cases are the {\it conformal embeddings}, which is an area of active recent study \cite{ACGY, AKMPPI, AKMPPII, AKMPPIII, AKMPPIV}. We have an embedding $\tilde{V}^{\psi-m-1}(\gg\gl_m) \hookrightarrow \cW_{\psi}(n,m)$ for some homomorphic image $\tilde{V}^{\psi-m-1}(\gg\gl_m)$ of $V^{\psi-m-1}(\gg\gl_m)$. We call this a conformal embedding if $\tilde{V}^{\psi-m-1}(\gg\gl_m)$ and $\cW_{\psi}(n,m)$ have the same Virasoro element; equivalently, $$\text{Com}(\tilde{V}^{\psi-m-1}(\gg\gl_m), \cW_{\psi}(n,m)) = \mathbb{C}.$$ Conformal embeddings occur for the following three values of $\psi$ as long as they are defined:
$$\frac{m + n -1}{n-1}, \qquad \frac{m + n}{n+1},\qquad \frac{m + n+1}{n}.$$
Besides being interesting in their own right, conformal embeddings in the case of minimal $\cW$-algebras are useful to prove semisimplicity of ordinary modules of affine vertex algebras at special non-admissible levels \cite{AKMPPII}, and to establish vertex tensor category structure on this category of ordinary modules \cite{CY}.

\smallskip

\noindent {\bf Triality from kernel vertex algebras}. Here we give a new perspective on trialities. 
 It is based on constructing a larger vertex algebra in which both $\cW^{\psi^{-1}}(n-m, m)$ and $\cV^{\psi'}(m, n)$ can be realized as cosets by certain affine vertex subalgebras. Davide Gaiotto and one of us studied vertex algebras in the context of $S$-duality in \cite{CG}. The set-up is again GL-twisted $\mathcal N=4$ supersymmetric gauge theory, and vertex algebras are associated to two-dimensional intersections of three-dimensional topological boundary conditions. Categories of vertex algebra modules arise as categories of line defects ending on these boundary conditions, and physics predicts that many interesting vertex algebras can be realized by gluing affine vertex algebras and $\cW$-algebras; see \cite[Section 1.1]{CGL} for a list of such predictions. The most important vertex algebra is the one that arises between Dirichlet boundary conditions and its $S$-duality image. This is actually a vertex superalgebra and it is called the quantum geometric Langlands kernel vertex algebra in \cite{CG}. If the gauge group is $\text{SU}(2)$, then this kernel algebra is $L_1(\gd(2, 1; -\lambda))$ and the gauge coupling $\psi$ is $\psi= \lambda +1$. The generalization to gauge group $\text{SU}(N)$ is supposed to be
\[
A[\gs\gl_N, \psi] := \bigoplus_{\lambda \in P^+} V^k(\lambda) \otimes V^\ell(\lambda) \otimes V_{\sqrt{N}\mathbb Z + \frac{s(\lambda)}{\sqrt{N}}}
\]
with $\psi= k+N, \psi'=\ell+N$ and 
$\frac{1}{\psi} +\frac{1}{\psi'} =1$.
The map $s : P^+ \rightarrow \mathbb Z/ N\mathbb Z$ is defined by $s(\lambda) = t \quad \text{if} \ \lambda = \omega_t \mod Q$, where $\omega_t$ is the $t$-th fundamental weight of $\gs\gl_N$ and we identify $\omega_0$ with $0$. 
The $V^k(\lambda)$ are generalized Verma modules at level $k$ whose top level is the integrable $\gs\gl_N$-module $\rho_\lambda$ of highest-weight $\lambda$. Conjecturally, $A[\gs\gl_N, \psi]$ can be given the structure of a simple vertex superalgebra for generic $\psi$. We also would like to include the case $N=1$ and so define $A[\gs\gl_1, \psi] := V_{\mathbb Z}$ to be just a pair of free fermions, i.e. the integer lattice vertex algebra. 
Let $f$ be a nilpotent element with corresponding complex $C_f$, i.e. the homology $H_f(V^k(\gg) \otimes C_f)$ is the $\cW$-algebra $\cW^k(\gg, f)$. We then denote the $\cW^k(\gg, f)$-module $H_f(M \otimes C_f)$ simply by  $H_f(M)$ for $M$ a $V^k(\gg)$-module. One then sets
\[
A[\gs\gl_N, f, \psi] := \bigoplus_{\lambda \in P^+} V^k(\lambda) \otimes H_f(V^\ell(\lambda)) \otimes V_{\sqrt{N}\mathbb Z + \frac{s(\lambda)}{\sqrt{N}}}
\]
and conjectures that this can be given the structure of a simple vertex superalgebra. Note that for $f$ the principal nilpotent, this is just $A[\gs\gl_N, f, \psi] \cong V^{k-1}(\gs\gl_N) \otimes \cF(2N)$ by the coset construction of principal $\cW$-algebras of \cite{ACLII}. Here $\cF(2N)$ is the vertex superalgebra of $2N$ free fermions.

Set $N=n+m$ and consider the nilpotent element $f=f_{n, m}$  corresponding to the partition $N=n+1+ \dots +1$ so that $\cW^\ell(\gs\gl_N, f)= \cW^{\ell+N}(n, m)$
is a hook-type $\cW$-algebra with $V^{\ell+n-1}(\gg\gl_m)$ as subalgebra. 
The top level corresponding to the standard representation of $\gs\gl_N$ in $A[\gs\gl_N, f, \psi] $ has conformal weight $\frac{N}{2} -\frac{n-1}{2}=\frac{m+1}{2}$ and it is expected to be odd. We want to take a coset that contains these elements. For this let $J$ be as in Lemma \ref{lem:vnmj} and let $\gamma$ be the generator of $\sqrt{N}Z=\gamma\mathbb Z$, i.e. $\gamma^2=N$.  Denote the corresponding Heisenberg field by $\gamma$ as well and set $H=J-\gamma$ and $\cH$ the Heisenberg vertex algebra generated by $H$. This ensures that the commutant with $V^{\ell+n-1}(\gs\gl_m) \otimes \cH$ contains the fields of conformal weight $\frac{m+1}{2}$ in the standard representation of $\gs\gl_N$, and its conjugate. 
 The conjecture motivated from and generalizing \cite{CG, GR} is that these fields actually generate a $\cW$-superalgebra of type $\gs\gl_{m|N}$.
\begin{conj}\label{conj:S-dualityintro} 
For generic $k$ and any nilpotent element $f$, the object $A[\gs\gl_N, f, \psi]$ can be given the structure of a simple vertex superalgebra, such that the top level of $V^k(\lambda) \otimes H_f(V^\ell(\lambda)) \otimes V_{\sqrt{N}\mathbb Z + \frac{s(\lambda)}{\sqrt{N}}}$ is odd for $\lambda = \omega_1, \omega_{N-1}$.
\end{conj}

\begin{thm}\label{thm:S-dualityintro}\textup{(Theorem \ref{thm: Sduality})}
 For generic $k$, if Conjecture \ref{conj:S-dualityintro}  is true for $f= f_{n, m}$, then 
$$\text{Com}\left(V^{\ell+n-1}(\gs\gl_m) \otimes \cH, A[\gs\gl_N, f_{n, m}, \psi]\right) \cong \cW^{-k-m+1}(\gs\gl_{m|N}, f_{m|N}).$$
\end{thm}
Note that $\cW^{-k-m+1}(\gs\gl_{m|N}, f_{m|N}) = \cV^{1-\psi}(m, n+m)$. It is immediate that $$\text{Com}\left(V^{k}(\gs\gl_N) \otimes \cH, A[\gs\gl_N, f_{n, m}, \psi]\right) \cong \cW^{\ell}(\gs\gl_N, f_{n,m}) = \cW^{\psi'}(n,m).$$ Therefore Theorem \ref{thm:S-dualityintro} gives a duality between the $\cW$-algebras $\cV^{1-\psi}(m, n+m)$ and $\cW^{\psi'}(n,m)$; both can be obtained as affine cosets of $A[\gs\gl_N, f_{n, m}, \psi]$. Note that the coset realization of all $\cW$-algebras $\cW^{\psi}(n,m)$ and $\cV^{\psi}(n,m)$ given by Theorem \ref{thm:S-dualityintro} vastly generalizes the coset realization of $\cW^{\psi}(n,0) \cong \cW^{\psi-n}(\gs\gl_n)$ from \cite{ACLII}.
Theorem \ref{thm:S-dualityintro} follows from a character statement that we prove 
in Appendix \ref{section:Sduality}, namely 
\begin{thm}\label{thm:S}
Graded characters agree,
\[
\text{ch}\left[\text{Com}\left(V^{\ell+n-1}(\gs\gl_m) \otimes \cH, A[\gs\gl_N, f_{n, m}, k]\right)  \right] = \text{ch}\left[ \cW^{-k-m+1}(\gs\gl_{m|N}, f_{m|N}) \right].
\]
\end{thm}
The idea of proof is inspired from the proof of \cite[Thm. 3.3]{CFL}.
While the characters of $\cW^\psi(n, m)$ and $\cV^\psi(n, m)$ do not have any good automorphic properties, it turns out that the character of 
$A[\gs\gl_N, f, \psi]$ is the expansion of a meromorphic Jacobi form in a certain domain. The decomposition problem of meromorphic Jacobi forms is an interesting problem in its own right, and depending on the index either mock modular forms (positive index) \cite{DMZ} or false theta functions (negative index) \cite{BCR} appear in the decomposition. The literature is mostly concerned with Jacobi forms in one variable, while we are effectively interested in $m$-variable meromorphic Jacobi forms. Our decomposition problem is a priori very difficult, but becomes feasible using representation theory of Lie superalgebras. 
As in \cite{BCR} we can use denominator identities of affine Lie superalgebras \cite{KWVI} to approach this problem. We are however not only interested in Fourier coefficients of meromorphic Jacobi forms, but actually into their decomposition into characters of highest-weight modules of $\widehat{\gs\gl}_m$. This turns out to be doable using denominator identities of finite Lie superalgebras. 

Relative semi-infinite Lie algebra cohomology acts on modules $M$ of an affine vertex algebra at level $-2h^\vee$ \cite{FGZ, ArIV}, and we use Section 2.5 of \cite{CFL} as background. Most importantly, it satisfies
\begin{equation}\nonumber
H^{\text{rel}, 0 }_{\infty}(\gg, V^k(\lambda) \otimes V^{-2h^\vee-k}(\mu)) = \begin{cases} \mathbb C & \ \text{if} \ \mu=-\omega_0(\lambda) \\ 0 & \ \text{otherwise} .\end{cases}
\end{equation}
Here $\omega_0$ is the unique Weyl group element that interchanges the fundamental Weyl chamber
with its negative. We consider $ \cW^\psi(n-m, m) \otimes A[\gs\gl_m, 1-\psi]  \otimes \pi^{k-\ell}$ where $\pi^{k-\ell}$ is a rank one Heisenberg vertex algebra. This ensures that if we take the appropriate relative semi-infinite Lie algebra cohomology, we obtain a vertex algebra that has odd generators of conformal weight $\frac{n+1}{2}$; see Section \ref{sec:kernel} for details. We believe that the following is true:
\begin{conj}\label{introconj:cohom} 
$H^{\text{rel}, 0 }_{\infty}(\gs\gl_m, \cW^\psi(n-m, m) \otimes A[\gs\gl_m, 1-\psi]  \otimes \pi^{k-\ell})$ is a simple vertex superalgebra. 
\end{conj} 
We can show that 
\begin{thm}\textup{(Theorem \ref{thm:relcoho})}
Let $k$ be generic and assume that Conjecture \ref{introconj:cohom} is true. Then 
$\cV^{\psi^{-1}}(n, m)\cong H^{\text{rel}, 0 }_{\infty}(\gs\gl_m, \cW^\psi(n-m, m) \otimes A[\gs\gl_m, 1-\psi]  \otimes \pi^{k-\ell})$.
\end{thm}
 
Conjecture \ref{introconj:cohom} was proven for $m=1$ in \cite{CGNS}, and was used to prove block-wise equivalences of categories of modules between $\cW^\psi(n-1, 1)$ and $\cV^{\psi^{-1}}(n, 1)$, as well as isomorphisms of superspaces of intertwining operators. It should also allow to investigate correspondences between correlation function and spaces of conformal blocks at arbitrary genus. We therefore consider Conjectures \ref{conj:S-dualityintro} and \ref{introconj:cohom}  to be an important problem. Our ideas are: Conjecture \ref{conj:S-dualityintro} might be provable using the Kazhdan-Lusztig equivalence of ordinary modules of an affine vertex algebra at generic level and corresponding quantum group modules \cite{KLI, KLII, KLIII, KLIV} together with the theory of gluing vertex algebras \cite{CKMII}. Conjecture \ref{introconj:cohom} might follow from \cite[Thm. 3.1]{ArIV}, but for this one first needs to be able to determine the associated varieties of the kernel vertex algebras $A[\gs\gl_m, \psi]$. This is interesting in its own right.

\smallskip

\noindent {\bf Gluing $Y$-algebras and equivalences of representation categories}. 
The first part of Conjecture \ref{conj:S-dualityintro} is the existence of a simple vertex superalgebra that extends the tensor product of an affine vertex algebra and a $\cW$-algebra. It is a general theorem that extensions of such a type are possible if and only if there is a braid-reversed equivalence of vertex tensor categories along which one glues \cite{CKMII}. The crucial assumption of \cite{CKMII} is the existence of vertex tensor category structure, which in general is very hard to prove. 
Let us consider $\cD^\psi(n, m) \cong  \cD^\phi(m, n)$ where $\phi^{-1} + \psi^{-1} =1$ and we take $\psi$ to be generic so that the categories $KL_{k}(\gs\gl_m)$ and  $KL_{\ell}(\gs\gl_n)$ of ordinary modules of $\gs\gl_m$ and $\gs\gl_n$ at levels $k=\psi-m+1$ and $\ell=\phi-n+1$ are semisimple. By our isomorphisms, $\cD^\psi(n, m)$ is a coset subalgebra of both $\cV^\psi(n, m)$ and $\cV^\phi(m, n)$. This means we have decompositions 
\begin{equation}
\begin{split}
\cV^\psi(n, m) &\cong \bigoplus_{\lambda \in P^+(\gs\gl_m)} V^k(\lambda) \otimes B_\lambda^\psi(n, m) \\
\cV^\phi(m, n) &\cong \bigoplus_{\lambda \in P^+(\gs\gl_n)} V^\ell(\lambda) \otimes C_\lambda^\phi(m, n).
\end{split}
\end{equation}
Here $P^+(\gg)$ denotes the set of dominant weights of $\gg$, and $B_\lambda^\psi(n, m), C_\lambda^\phi(m, n)$ are certain modules for $\cD^\psi(n, m)$ (times a Heisenberg vertex algebra if $n \neq m$).
\begin{conj}
Let $m, n\geq 2$ and $\psi$ be generic. Let $k=\psi-m+1$ and $\ell=\phi-n+1$. Then $\cD^\psi(n, m) \otimes \cH$ for $\cH$ a rank one Heisenberg algebra has a vertex tensor category of modules 
that is braid reversed equivalent to the Deligne product of $KL_k(\gs\gl_m)$ and $KL_\ell(\gs\gl_n)$. 
\end{conj}
For $n=0, 1$ one sets $KL_\ell(\gs\gl_n)$ to be trivial. 
Note that conjectures of this type for principal $\cW$-algebras have been made in the context of quantum geometric Langlands \cite[Conj. 6.4]{AFO} and proven for $n=2$ \cite[Prop. 5.5.2]{CJORY}. The difficult part of this conjecture is to establish the existence of rigid vertex tensor category structure which has been done for the Virasoro case in \cite{CJORY}.
As an example, consider the tensor product of $\cD^\psi(n, m)$, $\cD^{\psi'}(n', m)$  (and possibly a Heisenberg vertex algebra), assume that above conjecture holds in such a way that there is a braid-reversed equivalence $\tau$ between the categories of type $\gs\gl_m$.  They thus allow for an extension to a simple vertex algebra of the form \cite{CKMII}
\[
\bigoplus_{{\lambda \in P^+(\gs\gl_m)}} B_\lambda^\psi(n, m) \otimes \tau\left(B_\lambda^\psi(n, m)  \right)^*.
\]
Physics conjectures that these types of extensions exist and are isomorphic to other $\cW$-algebras; see \cite{PRI}. These gluing conjectures are tightly connected to certain magical properties of the quantum Hamiltonian reduction functor, e.g. if one of the factors is a prinicipal $\cW$-algebra ($n'=0$) and $\psi'=1-\psi^{-1}$, then such a gluing statement follows for $n\leq m$ from \cite{ACLII} together with the reduction functor commuting with tensoring with integrable representations \cite{ACF}.

As mentioned before, the $Y_{L, M, N}$-algebras are conjecturally isomorphic to the $W_{L, M, N}$-algebras of \cite{BFM}, and the latter act on the moduli space of spiked instantons of certain toric Calabi-Yau threefolds \cite{RSYZ}. The toric diagram of these examples has three two-dimensional faces and each face is labelled by non-negative integers $L, M, N$ that indicate an action of the gauge groups $U(L), U(M), U(N)$. In our case one of these labels is zero and our conjecture says that the corresponding $Y_{0, M, N}$-algebra has categories of type $KL_k(\gs\gl_M)$ and $KL_\ell(\gs\gl_N)$ for some $k, \ell$. The extension of a tensor product of two $Y$-algebras along a common $KL_k(\gs\gl_M)$ should geometrically correspond to a toric Calabi-Yau threefold whose toric diagram has four faces, and then iterating this procedure should 
correspond to diagrams with even more faces. Moreover, the resulting vertex algebras should still be cosets of $\cW$-superalgebras of type $A$. 
 This picture is currently a conjecture from physics considerations \cite{PRI}. Our results allow one to prove similar results as our Theorem \ref{thm:S}, i.e. show that extensions of certain tensor products of $Y$-algebras to simple vertex algebras exist and their characters coincide with the characters of the expected cosets of $\cW$-superalgebras.  We will report on this elsewhere.

\section{Vertex algebras} \label{section:VOAs}
We briefly define vertex algebras, which have been discussed from several points of view in the literature (see for example \cite{Bor,FLM,FHL,K,FBZ}). We will follow the formalism developed in \cite{LZ} and partly in \cite{LiI}. Let $V=V_0\oplus V_1$ be a super vector space over $\mathbb{C}$, $z,w$ be formal variables, and $\text{QO}(V)$ be the space of linear maps $$V\rightarrow V((z))=\{\sum_{n\in\mathbb{Z}} v(n) z^{-n-1}|
v(n)\in V,\ v(n)=0\ \text{for} \ n>\!\!>0 \}.$$ Each element $a\in \text{QO}(V)$ can be represented as a power series
$$a=a(z)=\sum_{n\in\mathbb{Z}}a(n)z^{-n-1}\in \text{End}(V)[[z,z^{-1}]].$$ We assume that $a=a_0+a_1$ where $a_i:V_j\rightarrow V_{i+j}((z))$ for $i,j\in\mathbb{Z}/2\mathbb{Z}$, and we write $|a_i| = i$.

For each $n \in \mathbb{Z}$, we have a bilinear operation on $\text{QO}(V)$, defined on homogeneous elements $a$ and $b$ by
\begin{equation*} \begin{split} a(w)_{(n)}b(w) & =\text{Res}_z a(z)b(w)\ \iota_{|z|>|w|}(z-w)^n \\ & - (-1)^{|a||b|}\text{Res}_z b(w)a(z)\ \iota_{|w|>|z|}(z-w)^n.\end{split} \end{equation*}
Here $\iota_{|z|>|w|}f(z,w)\in\mathbb{C}[[z,z^{-1},w,w^{-1}]]$ denotes the power series expansion of a rational function $f$ in the region $|z|>|w|$. For $a,b\in \text{QO}(V)$, we have the following identity of power series known as the {\it operator product expansion} (OPE) formula.
 \begin{equation}\label{opeform} a(z)b(w)=\sum_{n\geq 0}a(w)_{(n)} b(w)\ (z-w)^{-n-1}+:a(z)b(w):. \end{equation}
Here $:a(z)b(w):\ =a(z)_-b(w)\ +\ (-1)^{|a||b|} b(w)a(z)_+$, where $a(z)_-=\sum_{n<0}a(n)z^{-n-1}$ and $a(z)_+=\sum_{n\geq 0}a(n)z^{-n-1}$. Often, \eqref{opeform} is written as
$$a(z)b(w)\sim\sum_{n\geq 0}a(w)_{(n)} b(w)\ (z-w)^{-n-1},$$ where $\sim$ means equal modulo the term $:a(z)b(w):$, which is regular at $z=w$. 

Note that $:a(w)b(w):$ is a well-defined element of $\text{QO}(V)$. It is called the {\it Wick product} or {\it normally ordered product} of $a$ and $b$, and it
coincides with $a_{(-1)}b$. For $n\geq 1$ we have
$$ n!\ a(z)_{(-n-1)} b(z)=\ :(\partial^n a(z))b(z):,\qquad \partial = \frac{d}{dz}.$$
For $a_1(z),\dots ,a_k(z)\in \text{QO}(V)$, the $k$-fold iterated Wick product is defined inductively by
\begin{equation}\label{iteratedwick} :a_1(z)a_2(z)\cdots a_k(z):\ =\ :a_1(z)b(z):,\qquad b(z)=\ :a_2(z)\cdots a_k(z):.\end{equation}
We often omit the formal variable $z$ when no confusion can arise.

A subspace $\cA\subseteq \text{QO}(V)$ containing $1$ which is closed under all the above products is called a {\it quantum operator algebra} (QOA). We say that $a,b\in \text{QO}(V)$ are {\it local} if 
$$(z-w)^N [a(z),b(w)]=0$$ for some $N\geq 0$. A {\it vertex algebra} is a QOA whose elements are pairwise local. This is well known to be equivalent to the notion of a vertex algebra in the sense of \cite{FLM}. 

A vertex algebra $\cA$ is {\it generated} by a subset $S=\{\alpha^i|\ i\in I\}$ if $\cA$ is spanned by words in the letters $\alpha^i$, and all products, for $i\in I$ and $n\in\mathbb{Z}$. We say that $S$ {\it strongly generates} $\cA$ if $\cA$ is spanned by words in the letters $\alpha^i$, and all products for $n<0$. Equivalently, $\cA$ is spanned by $$\{ :\partial^{k_1} \alpha^{i_1}\cdots \partial^{k_m} \alpha^{i_m}:| \ i_1,\dots,i_m \in I,\ k_1,\dots,k_m \geq 0\}.$$ 
Suppose that $S$ is an ordered strong generating set $\{\alpha^1, \alpha^2,\dots\}$ for $\cA$ which is at most countable. We say that $S$ {\it freely generates} $\cA$, if $\cA$ has a Poincar\'e-Birkhoff-Witt basis
\begin{equation} \label{freegen} \begin{split} & :\partial^{k^1_1} \alpha^{i_1}  \cdots \partial^{k^1_{r_1}}\alpha^{i_1} \partial^{k^2_1}  \alpha^{i_2} \cdots \partial^{k^2_{r_2}}\alpha^{i_2}
 \cdots \partial^{k^n_1} \alpha^{i_n} \cdots  \partial^{k^n_{r_n}} \alpha^{i_n}:,\\ &
 1\leq i_1 < \dots < i_n,
 \\ & k^1_1\geq k^1_2\geq \cdots \geq k^1_{r_1},\quad k^2_1\geq k^2_2\geq \cdots \geq k^2_{r_2},  \ \ \cdots,\ \  k^n_1\geq k^n_2\geq \cdots \geq k^n_{r_n},
 \\ &  k^{t}_1 > k^t_2 > \dots > k^t_{r_t} \ \ \text{if} \ \ \alpha^{i_t}\ \ \text{is odd}. 
 \end{split} \end{equation}

We recall some important identities that hold in any vertex algebra $\cA$. For fields $a,b,c \in \cA$, we have
\begin{equation} \label{deriv} (\partial a)_{(n)} b = -na_{n-1}, \qquad \forall n\in \mathbb{Z},\end{equation}
\begin{equation} \label{commutator} a_{(n)} b  = (-1)^{|a||b|} \sum_{p \in \mathbb{Z}} (-1)^{p+1} (b_{(p)} a)_{(n-p-1)} 1,\qquad \forall n\in \mathbb{Z},\end{equation}
\begin{equation} \label{nonasswick} \begin{split}  :(:ab:)c:\  - \ :abc:\ 
& =  \sum_{n\geq 0}\frac{1}{(n+1)!}\big( :(\partial^{n+1} a)(b_{(n)} c):\ 
\\ & + (-1)^{|a||b|} (\partial^{n+1} b)(a_{(n)} c):\big), \end{split} \end{equation}
\begin{equation}  \label{ncw} \begin{split}  & a_{(n)}
(:bc:) -\ :(a_{(n)} b)c:\ - (-1)^{|a||b|}\ :b(a_{(n)} c): \\ & = \sum_{i=1}^n
\binom{n}{i} (a_{(n-i)}b)_{(i-1)}c, \qquad  \forall n \geq 0, \end{split}
\end{equation}
\begin{equation} \label{jacobi} a_{(r)}(b_{(s)} c) = (-1)^{|a||b|} b_{(s)} (a_{(r)}c) + \sum_{i =0}^r \binom{r}{i} (a_{(i)}b)_{(r+s - i)} c,\qquad \forall r,s \geq 0.\end{equation}

The identities \eqref{jacobi} are known as {\it Jacobi identities} of type $(a,b,c)$, and they play an important role in the proof of our main theorem.

\subsection{Conformal structure} A conformal structure with central charge $c$ on a vertex algebra $\cA$ is a Virasoro vector $L(z) = \sum_{n\in \mathbb{Z}} L_n z^{-n-2} \in \cA$ satisfying
\begin{equation} \label{virope} L(z) L(w) \sim \frac{c}{2}(z-w)^{-4} + 2 L(w)(z-w)^{-2} + \partial L(w)(z-w)^{-1},\end{equation} such that $L_{-1} \alpha = \partial \alpha$ for all $\alpha \in \cA$, and $L_0$ acts diagonalizably on $\cA$. We say that $\alpha$ has conformal weight $d$ if $L_0(\alpha) = d \alpha$, and we denote the conformal weight $d$ subspace by $\cA[d]$. In all our examples, this grading will be by $\mathbb{Z}_{\geq 0}$ or $\frac{1}{2} \mathbb{Z}_{\geq 0}$. We say $\cA$ is of type $\cW(d_1,d_2,\dots)$, if it has a minimal strong generating set consisting of one even field in each conformal weight $d_1, d_2, \dots $. 

\subsection{Coset construction} 
Given a vertex algebra $\cV$ and a subalgebra $\cA \subseteq \cV$, the {\it coset} or {\it commutant} of $\cA$ in $\cV$, denoted by $\text{Com}(\cA,\cV)$, is the subalgebra of elements $v\in\cV$ such that $$[a(z),v(w)] = 0,\qquad\forall a\in\cA.$$ This was introduced by Frenkel and Zhu in \cite{FZ}, generalizing earlier constructions in \cite{GKO,KP}. Equivalently, $v\in \text{Com}(\cA,\cV)$ if and only if $a_{(n)} v = 0$ for all $a\in\cA$ and $n\geq 0$. 
Note that if $\cV$ and $\cA$ have Virasoro elements $L^{\cV}$ and $L^{\cA}$, $\text{Com}(\cA,\cV)$ has Virasoro element $L = L^{\cV} - L^{\cA}$ as long as $ L^{\cV} \neq L^{\cA}$.

\subsection{Affine vertex algebras} Let $\gg$ be a simple, finite-dimensional, Lie (super)algebra with dual Coxeter number $h^{\vee}$, equipped with the standard supersymmetric invariant bilinear form $(-|-)$. The corresponding affine Lie algebra $\widehat{\gg} = \gg \otimes_{\mathbb{C}} \mathbb{C}[t,t^{-1}] \oplus \mathbb{C} K$ has bracket \begin{equation} \label{affineLie} [\xi \otimes t^n, \eta \otimes t^m] = [\xi, \eta] \otimes t^{n+m} + n \delta_{n+m, 0} (\xi|\eta) K,\end{equation} and $K$ is central. The {\it universal affine vertex (super)algebra} $V^k(\gg)$ is isomorphic to the vacuum $\widehat{\gg}$-module. It is freely generated by fields $X^{\xi}$ as $\xi$ runs over a basis of $\gg$, which satisfy
$$X^{\xi}(z)X^{\eta} (w)\sim k(\xi |\eta) (z-w)^{-2} + X^{[\xi,\eta]}(w) (z-w)^{-1} .$$ We may choose dual bases $\{\xi\}$ and $\{\xi'\}$ of $\gg$, satisfying $(\xi'|\eta)=\delta_{\xi,\eta}$. If $k+h^\vee\neq 0$, there is a Virasoro element
\begin{equation} \label{sugawara}
L^{\gg} = \frac{1}{2(k+h^\vee)}\sum_\xi :X^{\xi}X^{\xi'}:
\end{equation} of central charge $ c= \frac{k \cdot \text{sdim}(\gg)}{k+h^\vee}$. This is known as the {\it Sugawara conformal vector}, and each $X^{\xi}$ is primary of weight one. We denote by $L_k(\gg)$ the simple quotient of $V^k(\gg)$ by its maximal proper ideal graded by conformal weight.

\subsection{Free field algebras} 
\begin{defn} A free field algebra is a vertex superalgebra $\cV$ with weight grading
$$\cV = \bigoplus_{d \in \frac{1}{2} \mathbb{Z}_{\geq 0} }\cV[d],\qquad \cV[0] \cong \mathbb{C},$$ with strong generators $\{X^i|\ i \in I\}$ satisfying OPE relations
\begin{equation} \label{gffa:def} \begin{split} & X^i(z) X^j(w) \sim a_{i,j} (z-w)^{-\text{wt}(X^i) - \text{wt}(X^j)}, \\ & a_{i,j} \in \mathbb{C},\ \quad a_{i,j} = 0\ \ \text{if} \ \text{wt}(X^i) +\text{wt}(X^j)\notin \mathbb{Z}. \end{split} \end{equation}
\end{defn}

Note that we do not assume that $\cV$ has a conformal structure. We next introduce four families of standard free field algebras. They are either of symplectic or orthogonal type, and the generators are either even or odd.

\smallskip

\noindent {\it Even algebras of orthogonal type}. For each $n\geq 1$ and even $k \geq 2$, we define $\cO_{\text{ev}}(n,k)$ to be the vertex algebra with even generating fields $a^1,\dots, a^n$ of weight $\frac{k}{2}$, which satisfy 
$$a^i(z) a^j(w) \sim \delta_{i,j} (z-w)^{-k}.$$ In the case $k=2$, $\cO_{\text{ev}}(n,k)$ just the rank $n$ Heisenberg algebra $\cH(n)$. If we let $\alpha^1,\dots, \alpha^n$ denote the standard generators for $\cH(n)$ satisfying $\alpha^{i} (z) \alpha^{j}(w) \sim \delta_{i,j} (z-w)^{-2}$, then $\cO_{\text{ev}}(n,k)$ can be realized inside $\cH(n)$ by setting $$ a^i = \frac{\epsilon}{\sqrt{(k-1)!}} \partial^{k/2-1}\alpha^i,\qquad i=1,\dots, n,$$ where $\epsilon = \sqrt{-1}$ if $4|k$, and otherwise $\epsilon = 1$. Note that $\cH(n)$ has Virasoro element $L^{\cH} =  \frac{1}{2} \sum_{i=1}^n  :\alpha^i \alpha^i:$ of central charge $n$, under which $\alpha^i$ is primary of weight one, but $\cO_{\text{ev}}(n,k)$ has no conformal vector for $k>2$. However, for all $k$ it is a simple vertex algebra and has full automorphism group the orthogonal group $\text{O}_n$.

\smallskip

\noindent {\it Even algebras of symplectic type}. For each $n\geq 1$ and odd $k \geq 1$, we define $\cS_{\text{ev}}(n,k)$ to be the vertex algebra with even generators $a^i, b^i$ for $i=1,\dots, n$ of weight $\frac{k}{2}$, which satisfy \begin{equation} \begin{split} a^i(z) b^{j}(w) &\sim \delta_{i,j} (z-w)^{-k},\qquad b^{i}(z)a^j(w)\sim -\delta_{i,j} (z-w)^{-k},\\ a^i(z)a^j(w) &\sim 0,\qquad\qquad\qquad \ \ \ \ b^i(z)b^j (w)\sim 0.\end{split} \end{equation} In the case $k=1$,  $\cS_{\text{ev}}(n,k)$ is just the rank $n$ $\beta\gamma$-system. Let $\beta^{i}, \gamma^{i}$, $i=1,\dots, n$, be the standard generators of $\cS(n)$, which satisfy
\begin{equation} \label{eq:betagammaope} \begin{split}
\beta^i(z)\gamma^{j}(w) &\sim \delta_{i,j} (z-w)^{-1},\quad \gamma^{i}(z)\beta^j(w)\sim -\delta_{i,j} (z-w)^{-1},\\ 
\beta^i(z)\beta^j(w) &\sim 0,\qquad\qquad\qquad \gamma^i(z)\gamma^j (w)\sim 0.\end{split} \end{equation} 
Then $\cS_{\text{ev}}(n,k)$ can be realized as the subalgebra of $\cS(n)$ with generators
$$a^i =  \frac{\epsilon}{\sqrt{(k-1)!}} \partial^{(k-1)/2} \beta^i, \qquad b^i =  \frac{\epsilon}{\sqrt{(k-1)!}} \partial^{(k-1)/2}\gamma^i, \qquad i=1,\dots, n,$$ with $\epsilon$ as above. We give $\cS(n)$ the Virasoro element $L^{\cS} = \frac{1}{2} \sum_{i=1}^n \big(:\beta^{i}\partial\gamma^{i}: - :\partial\beta^{i}\gamma^{i}:\big)$ of central charge $-n$, under which $\beta^{i}$, $\gamma^{i}$ are primary of weight $\frac{1}{2}$. Note that $\cS_{\text{ev}}(n,k)$ has no conformal vector for $k>1$, but for all $k$ it is simple and has full automorphism group the symplectic group $\text{Sp}_{2n}$.

\begin{remark} If we change the weight grading and pass to completions, there can be additional automorphisms. For example, the algebra of chiral differential operators on the upper half plane $\cD^{\text{ch}}(\mathbb{H})$ is a completion of $\cS(1)$. There is an action of $V^{-2}(\gs\gl_2)$ on $\cS(1)$ given by 
$$h \mapsto -2 :\beta\gamma:,\qquad x \mapsto -\beta,\qquad y \mapsto \ :\beta\gamma\gamma: + 2 \partial \gamma,$$ and there is a compatible action of $\text{SL}_2$ on $\cD^{\text{ch}}(\mathbb{H})$ \cite{Dai}.
\end{remark}

\smallskip

\noindent {\it Odd algebras of symplectic type}. For each $n\geq 1$ and even $k \geq 2$, we define $\cS_{\text{odd}}(n,k)$ to be the vertex superalgebra with odd generators $a^i, b^i$ for $i=1,\dots, n$ of weight $\frac{k}{2}$, which satisfy
\begin{equation} \begin{split} a^{i} (z) b^{j}(w)&\sim \delta_{i,j} (z-w)^{-k},\qquad b^{j}(z) a^{i}(w)\sim - \delta_{i,j} (z-w)^{-k},\\ a^{i} (z) a^{j} (w)&\sim 0,\qquad\qquad\qquad\ \ \ \  b^{i} (z) b^{j} (w)\sim 0. \end{split} \end{equation} In the case $k=2$, $\cS_{\text{odd}}(n,k)$ is just the rank $n$ symplectic fermion algebra $\cA(n)$. Let $e^{i}, f^{i}$, $i =1,\dots, n$ be standard generators for $\cA(n)$ satisfying \begin{equation} \label{eq:sfope} \begin{split}
e^{i} (z) f^{j}(w)&\sim \delta_{i,j} (z-w)^{-2},\quad f^{j}(z) e^{i}(w)\sim - \delta_{i,j} (z-w)^{-2},\\
e^{i} (z) e^{j} (w)&\sim 0,\qquad\qquad\qquad f^{i} (z) f^{j} (w)\sim 0.
\end{split} \end{equation} 
Then $\cS_{\text{odd}}(n,k)$ is realized as the subalgebra of $\cA(n)$ with generators 
$$a^i =  \frac{\epsilon}{\sqrt{(k-1)!}} \partial^{k/2-1} e^i, \qquad b^i =  \frac{\epsilon}{\sqrt{(k-1)!}} \partial^{k/2-1} f^i,\qquad i = 1,\dots, n.$$ 
As above, $\cA(n)$ has Virasoro element $L^{\cA} =  - \sum_{i=1}^n  :e^i f^i:$ of central charge $-2n$, under which $e^i, f^i$ are primary of weight one, and $\cS_{\text{odd}}(n,k)$ has no conformal vector for $k>2$. However, it is simple and has full automorphism group $\text{Sp}_{2n}$.

\smallskip

\noindent {\it Odd algebras of orthogonal type}. For each $n\geq 1$ and odd $k \geq 1$, we define $\cO_{\text{odd}}(n,k)$ to be the vertex superalgebra with odd generators $a^i$ for $i=1,\dots, n$ of weight $\frac{k}{2}$, satisfying
\begin{equation}a^i(z) a^j(w) \sim \delta_{i,j} (z-w)^{-k}.\end{equation} For $k=1$, $\cO_{\text{odd}}(n,k)$ is just the free fermion algebra $\cF(n)$. Let $\phi_i$, $i=1,\dots, n$ be standard generators for $\cF(n)$, satisfying 
\begin{equation} \label{eq:ffope} \phi^{i} (z) \phi^{j}(w) \sim \delta_{i,j} (z-w)^{-1}.\end{equation} Then $\cO_{\text{odd}}(n,k)$ is realized as the subalgebra of $\cF(n)$ with generators 
$$a^i = \frac{\epsilon}{\sqrt{(k-1)!}} \partial^{(k-1)/2} \phi^i,\qquad i = 1,\dots, n.$$ Then $\cF(n)$ has Virasoro element $L^{\cF} =  -\frac{1}{2} \sum_{i=1}^n  :\phi^i \partial \phi^i:$ of central charge $\frac{n}{2}$, under which $\phi^i$ is primary of weight $\frac{1}{2}$, but $\cO_{\text{odd}}(n,k)$ has no conformal vector for $k>1$. However, it is simple and has full automorphism group $\text{O}_n$.

For later use, we mention that the $bc$-system $\cE(n)$ of rank $n$ is isomorphic to $\cF(2n)$; it has odd generators $b^i,c^i$, $i = 1,\dots, n$ and OPEs 
\begin{equation}  \label{eq:bcope} \begin{split}
b^i(z) c^{j}(w) & \sim \delta_{i,j} (z-w)^{-1},\quad c^{i}(z) b^j(w)\sim \delta_{i,j} (z-w)^{-1},\\ 
b^i(z) b^j(w) &\sim 0,\qquad\qquad\qquad c^i(z) c^j (w)\sim 0.\end{split} \end{equation}

A particularly important class of free field algebras is those which decompose as a finite tensor product of standard ones of the above four types. As we shall see in the next section, affine $\cW$-algebras admit a suitable limit which is a free field algebra of this form. This feature provides a powerful tool for analyzing the structure of orbifolds of $\cW$-algebras and cosets of $\cW$-algebras by affine subalgebras.

\subsection{Vertex algebras over commutative rings} \label{subsection:voaring}
Let $R$ be a finitely generated commutative $\mathbb{C}$-algebra. A vertex algebra over $R$ is an $R$-module $\cA$ with a vertex algebra structure which is defined as above. The theory of vertex algebras over general commutative rings was developed by Mason \cite{Ma}, but the main difficulties are not present when $R$ is a $\mathbb{C}$-algebra. We will use the notation and setup of Section 3 of \cite{LVI}.

Let $\cV$ be a vertex algebra over $R$ with conformal weight grading
$$\cV = \bigoplus_{d \in \frac{1}{2} \mathbb{Z}_{\geq 0}} \cV[d],\qquad \cV[0] \cong R.$$ Here $\frac{1}{2}\mathbb{Z}_{\geq 0}$ is regarded as a subsemigroup of $R$. A vertex algebra ideal $\cI \subseteq \cV$ is called {\it graded} if $$\cI = \bigoplus_{d \in \frac{1}{2} \mathbb{Z}_{\geq 0}} \cI[d],\qquad \cI[d] = \cI \cap \cV[d].$$ 
We say that $\cV$ is {\it simple} if there are no proper graded ideals $\cI$ such that $\cI[0] = \{0\}$. If $I\subseteq R$ is an ideal, we may regard $I$ as a subset of $\cV[0] \cong R$. Let $I \cdot \cV$ denote the set of $I$-linear combinations of elements of $\cV$, which is just the vertex algebra ideal generated by $I$. Then $$\cV^I = \cV / (I \cdot \cV)$$ is a vertex algebra over the ring $R/I$. Even if $\cV$ is simple as a vertex algebra over $R$, $\cV^I$ need not be simple as a vertex algebra over $R/I$.

If $\cV = \bigoplus_{d \in \frac{1}{2} \mathbb{Z}_{\geq 0}}\cV[d]$, as above, each $\cV[d]$ has a bilinear form 
\begin{equation} \label{bilinearform} \langle,\rangle_d: \cV[d]\otimes_{R} \cV[d] \rightarrow R,\qquad \langle u,v \rangle_d = u_{(2d-1)}v.\end{equation}
We declare $\langle \cV[d],\cV[e]\rangle=0$ if $d\neq e$, and we extend $\langle,\rangle$ linearly to all of $\cV$.

A vector $v$ in the radical of the Shapovalov form $\langle,\rangle$ is called a singular vector. Suppose now that each weight space $\cV[d]$ is a free $R$-module of finite rank. We then define the level $d$ Shapovalov determinant $\mathrm{det}_d \in R$ to be the determinant of the matrix of $\langle,\rangle_d$. The following lemma is known in the case of $\mathbb{Z}_{\geq 0}$-gradings \cite[Prop. 2.2]{KL}, and the proof for $\frac{1}{2} \mathbb{Z}_{\geq 0}$-gradings is the same.

\begin{lemma} \label{prop:shap} Let $\mathcal{V}$ be a $\frac{1}{2} \mathbb{Z}_{\geq 0}$-graded vertex algebra over $R$ where $\cV[0] \cong R$ and each $\cV[d]$ is a free $R$-module of finite rank. We also assume that $L_1\cV[1]=0$. Then a homogeneous vector of weight $d>0$ is in the radical of the Shapovalov form if and only if it is contained in a proper ideal of $\mathcal{V}$. 
\end{lemma}

Under the above hypotheses, if $R$ is in addition a unique factorization ring, each irreducible factor $a$ of $\mathrm{det}_d$ give rise to a prime ideal $(a) \subseteq R$. Clearly if $a| \mathrm{det}_d$, then $a|\mathrm{det}_e$ for all $e>d$. The set of distinct prime ideals of the form $ I = (a) \subseteq R$ such that $a$ is a divisor of $\mathrm{det}_d$ for some $d$, are precisely the prime ideals for which $\cV^I$ is not simple as a vertex algebra over $R/I$.

\section{$\cW$-algebras} \label{section:walgebras}

We use a mix of \cite{KWIII, DSKVIII} as reference. Let $\gg$ be a simple Lie superalgebra with nondegenerate invariant supersymmetric bilinear form 
\begin{equation}
( \ \ | \ \ ): \gg \times \gg \rightarrow \mathbb C.
\end{equation}
Let $\{q^\alpha\}_{\alpha \in S}$ be a basis of $\gg$ indexed by the set $S$ and homogeneous with respect to the grading by parity. We then define the corresponding structure constants and parity by
\[
[q^\alpha, q^\beta] = \sum_{\gamma \in S}{f^{\alpha\beta}}_\gamma q^\gamma\qquad \text{and} \qquad {|\alpha|} = \begin{cases} 0 & \ q^\alpha \ \text{even} \\ 1 & \ q^\alpha \ \text{odd} \end{cases}.
\]
The affine vertex algebra of $\gg$ associated to the bilinear form $( \ \ | \ \ )$  at level $k$ in $\mathbb C$ is strongly generated by $\{X^\alpha\}_{\alpha \in S}$  with operator products 
\[
X^\alpha(z)X^\beta(w) \sim \frac{k(q^\alpha|q^\beta)}{(z-w)^2} + \frac{\sum_{\gamma\in S}{f^{\alpha \beta}}_\gamma X^\gamma (w) }{(z-w)}.
\]
Also, we define $X_\alpha$ to be the field corresponding to $q_\alpha$ where $\{q_\alpha\}_{\alpha \in S}$ is the basis of $\gg$ dual with respect to $( \ \ | \ \  )$.

Let $f$ be a nilpotent element in the even part of $\gg$. By the Jacobson-Morozov theorem, $f$ can be completed to an $\gs\gl_2$ triple $\{f, x, e\} \subseteq \gg$ satisfying the standard relations $[x, e] =e, [x, f]=-f, [e, f]= 2x$. The $\cW$-superalgebra $\cW^k(\gg,f)$ we are going to define depends only on the conjugacy class of $f$ and not on this choice of embedding of $\gs\gl_2$.

Then $\gg$ decomposes as an $\gs\gl_2$-module as follows.
\[
\gg = \bigoplus_{k \in  \frac{1}{2}\mathbb Z} \gg_k, \qquad \gg_k = \{  a \in \gg | [x, a] = ka \}. 
\] 
Let $S_k$ be a basis of $\gg_k$ and extend to the corresponding basis of $\gg$, i.e. set $S=\bigcup_k S_k$. 
Let 
us also set 
\[
\gg_+ = \bigoplus_{k \in  \frac{1}{2}\mathbb Z_{>0}} \gg_k, \qquad \gg_- = \bigoplus_{k \in  \frac{1}{2}\mathbb Z_{<0}} \gg_k
\]
with corresponding bases $S_+$ of $\gg_+$ and $\gg_-$ is naturally identified with the dual of $\gg_+$. 
On $\gg_{\frac{1}{2}}$ one defines the invariant bilinear form
\[
\langle a, b \rangle := ( f | [a, b] ). 
\]
Let $F(\gg_+)$ be the vertex superalgebra associated to the vector superspace $\gg_+ \oplus \gg_+^*$. It is strongly generated by fields $\{\varphi_\alpha, \varphi^\alpha\}_{\alpha \in S_+}$, where $\varphi_\alpha$ and $\varphi^\alpha$ are odd if $\alpha$ is even and even if $\alpha$ is odd. The operator products are
\[
\varphi_\alpha(z) \varphi^\beta(w) \sim \frac{\delta_{\alpha, \beta}}{(z-w)}, \qquad \varphi_\alpha(z) \varphi_\beta(w) \sim 0 \sim \varphi^\alpha(z) \varphi^\beta(w).
\]
Let $F(\gg_{\frac{1}{2}})$ be the neutral vertex superalgebra associated to $\gg_{\frac{1}{2}}$ with bilinear form $\langle \ \ , \ \ \rangle$. This has strong generators $\{\Phi_\alpha\}_{\alpha \in S_{\frac{1}{2}}}$ with $\Phi_\alpha$ even if $\alpha$ is even and odd if $\alpha$ is odd. The operator products are
\begin{equation}\label{eq:neutral}
\Phi_\alpha(z) \Phi_\beta(w) \sim \frac{\langle q^\alpha , q^\beta \rangle}{(z-w)}  \sim \frac{ ( f | [q^\alpha, q^\beta] )}{(z-w)},
\end{equation}
and fields corresponding to the dual basis with respect to $\langle \ \ , \ \ \rangle$ are denoted by $\Phi^\alpha$. 
The complex is 
\[
C(\gg, f, k) := V^k(\gg) \otimes F(\gg_+) \otimes F(\gg_{\frac{1}{2}}).
\]
One defines a $\mathbb Z$-grading by giving the $\varphi_\alpha$ charge minus one, the $\varphi^\alpha$ charge one and all others charge zero. 
One further defines the odd field $d(z)$ of charge minus one by
\begin{equation}
\begin{split}
d(z) &= \sum_{\alpha \in S_+} (-1)^{|\alpha|} :X^\alpha \varphi^\alpha:  
-\frac{1}{2} \sum_{\alpha, \beta, \gamma \in S_+}  (-1)^{|\alpha||\gamma|} {f^{\alpha \beta}}_\gamma :\varphi_\gamma \varphi^\alpha \varphi^\beta: + \\
&\quad  \sum_{\alpha \in S_+} (f | q^\alpha) \varphi^\alpha + \sum_{\alpha \in S_{\frac{1}{2}}} :\varphi^\alpha \Phi_\alpha:.
\end{split} 
\end{equation}
The zero mode $d_0$ is a differential since $[d(z), d(w)]=0$ by \cite[Thm. 2.1]{KRW}. Set $m_\alpha=j$ if $\alpha \in S_j$. The $\mathcal W$-algebra is  defined to be its homology
\[
\cW^k(\gg, f) := H\left(C(\gg, f, k), d_0 \right). 
\]
The relevant Virasoro fields are 
\begin{equation}
\begin{split}
L_{\text{sug}} &=  \frac{1}{2(k+h^\vee)} \sum_{\alpha \in S} (-1)^{|\alpha|} :X_\alpha X^\alpha:,\\
L_{\text{ch}} &= \sum_{\alpha \in S_+} \left(-m_\alpha :\varphi^\alpha \partial \varphi_\alpha: + (1-m_\alpha) :(\partial\varphi^\alpha )\varphi_\alpha:  \right),\\
L_{\text{ne}} &= \frac{1}{2} \sum_{\alpha \in S_{\frac{1}{2}}} :(\partial \Phi^\alpha) \Phi_\alpha : ,\\
L &= L_{\text{sug}} + \partial x + 
L_{\text{ch}} + L_{\text{ne}}.
\end{split}
\end{equation}
$L$ is an element of $\cW^k(\gg, f)$ and has central charge 
\begin{equation}\label{eq:c}
c(\gg, f, k) = \frac{k\, \text{sdim}\, \gg}{k+h^\vee} -12 k(x|x) -\sum_{\alpha \in S_+} (-1)^{|\alpha|} (12m_\alpha^2-12m_\alpha +2) -\frac{1}{2}\, \text{sdim}\, \gg_{\frac{1}{2}}.
\end{equation}
Set
\begin{equation}
J^\alpha = X^\alpha + \sum_{\beta, \gamma \in S_+} (-1)^{|\gamma|} {f^{\alpha \beta }}_\gamma :\varphi_\gamma\varphi^\beta:.
\end{equation}
Their $\lambda$-bracket is \cite[Eq. 2.5]{KWIII}
\[
[J^\alpha {}_\lambda J^\beta] = {f^{\alpha \beta}}_\gamma J^\gamma + \lambda\left(k(q^\alpha | q^\beta) + \frac{1}{2}\left(\kappa_\gg(q^\alpha, q^\beta) -\kappa_{\gg_0}(q^\alpha, q^\beta) \right) \right)
\]
with $\kappa_\gg$, $\kappa_{\gg_0}$ the Killing forms, that is the supertrace of the adjoint representations, of $\gg$, $\gg_0$, respectively.
The action of $d_0$ is \cite[Eq. 2.6]{KWIII}
\begin{equation}
\begin{split}
d_0(J^\alpha) &=  \sum_{\beta \in S_+} ([f, q^\alpha], q^\beta)\varphi^\beta + \sum_{\substack{ \beta \in S_+ \\ \gamma \in S_{\frac{1}{2}}}} (-1)^{|\alpha|(|\beta|+1)} {f^{\alpha\beta}}_\gamma \varphi^\beta \Phi_\gamma -\\
&  \sum_{\substack{\beta \in S_+ \\ \gamma \in S\setminus S_+}} (-1)^{|\beta|(|\alpha|+1)}{f^{\alpha\beta}}_\gamma \varphi^\beta J^\gamma 
\\ & + \sum_{\beta \in S_+} \left( k ( q^\alpha | q^\beta ) 
 +  \text{str}_{\gg_+}\left( p_+ ( \text{ad}(q^\alpha)) \text{ad}(q^\beta) \right)\right) \partial\varphi^\beta,
\end{split}
\end{equation}
with $p_+$ the projection onto $\gg_+$ and $\text{str}_{\gg_+}$ the supertrace on $\gg_+$. Set 
\begin{equation}\label{eq:Ialpha}
I^\alpha := J^\alpha + \frac{(-1)^{|\alpha|}}{2} \sum_{\beta \in S_{\frac{1}{2}}}  {f^{\beta\alpha}}_\gamma \Phi^\beta\Phi_\gamma
\end{equation}
for $\alpha \in \gg_0$. Denote by $\gg^f$ the centralizer of $f$ in $\gg$, and set $\ga := \gg^f \cap \gg_0$. It is a Lie subsuperalgebra of $\gg$. The next theorem tells us that $\cW^k(\gg, f)$  contains an affine vertex superalgebra of type $\ga$. 
\begin{thm} \cite[Thm 2.1]{KWIII}\label{thm:structureI}
\begin{enumerate}
\item $d_0(I^\alpha)=0$ for $q^\alpha \in \ga$ and
\begin{equation*} \begin{split} 
[I^\alpha {}_\lambda I^\beta] & = {f^{\alpha \beta}}_\gamma I^\gamma 
\\ & + \lambda\bigg(k(q^\alpha | q^\beta) + \frac{1}{2}\bigg(\kappa_\gg(q^\alpha, q^\beta) -\kappa_{\gg_0}(q^\alpha, q^\beta)-\kappa_{\frac{1}{2}}(q^\alpha, q^\beta) \bigg) \bigg),
\end{split} \end{equation*}
with $\kappa_{\frac{1}{2}}$ the supertrace of $\gg_0$ on $\gg_{\frac{1}{2}}$. 
\item
\[
 [ L {}_\lambda J^\alpha ]=(\partial +(1- j)\lambda )J^\alpha + \delta_{j, 0} \lambda^2\left(\frac{1}{2}{\rm str}_{\gg_+}({\rm ad}\ q^\alpha) -(k+h^\vee)( q^\alpha |x)\right),
 \]
 for $\alpha \in S_j$, and the same formula holds for $I^\alpha$ if $q^\alpha \in \ga$.
\end{enumerate}
\end{thm}
The main structural theorem is  
\begin{thm} \cite[Thm 4.1]{KWIII} \label{thm:kacwakimoto}

Let $\gg$ be a simple finite-dimensional Lie superalgebra with an invariant bilinear
form $( \ \ |  \ \ )$, and let $x, f$ be a pair of even elements of $\gg$ such that ${\rm ad}\ x$ is diagonalizable with
eigenvalues in $\frac{1}{2} \mathbb Z$ and $[x,f] = -f$. Suppose that all eigenvalues of ${\rm ad}\ x$ on $\gg^f$ are non-positive: 
\[
\gg^f = \bigoplus_{j\leq 0} \gg^f_j.
\]
 Then
 \begin{enumerate}
\item For each $q^\alpha \in \gg^f_{-j}$, ($j\geq  0$) there exists a $d_0$-closed field $K^\alpha$ of conformal weight
$1 + j$ (with respect to $L$) such that $K^\alpha - J^\alpha$ is a linear combination of normal ordered products of the fields $J^\beta$, where $\beta \in S_{-s}, 0 \leq s < j$, the fields $\Phi_\alpha$, where $\alpha \in S_{\frac{1}{2}}$, and the derivatives of these fields.
\item The homology classes of the fields $K^\alpha$, where $\{ q^\alpha\}_{\alpha \in S^f}$ is a basis of $\gg^f$ indexed by the set $S^f$ and  compatible with its $\frac{1}{2}\mathbb Z$-gradation, strongly and freely generate the vertex algebra $\cW^k(\gg, f)$.
\item $H_0(C(\gg, f, k), d_0) = \cW^k(\gg, f)$ and $H_j(C(\gg, f, k), d_0) = 0$ if $j \neq 0$.
\end{enumerate}
\end{thm}
This theorem is proven by first observing that the complex splits into the tensor product of two complexes denoted by $C^+$ and $C^-$, which each are $d_0$-invariant and vertex subsuperalgebras of $C(\gg, f, k)$. It turns out that the homology on $C^+$ is one-dimensional and so one needs to compute the homology on $C^-$. This is done by introducing an ascending filtration and computing the homology of the associated graded algebra of the complex (whose differential is denoted by $d_1$). This homology turns out to be 
\[
H_0(\text{gr}\ C^-, d_1) \cong V(\gg^f), \qquad   H_j(\text{gr}\ C^-, d_1) =0,  \qquad j \neq 0.
\]
Charge considerations imply that the spectral sequence converges to this homology. The proof has two useful corollaries. First, $L_0$ and also the $\ga$-action given by the zero modes of the $I^\alpha$ for $q^\alpha \in \ga$, are preserved by $d_1$ and $d_0$, so that $\cW^k(\gg, f) \cong V(\gg^f)$ as $\mathbb C L_0 \oplus  \ga$-modules.  Second, $\cW^k(\gg, f)$ is a subalgebra of $C(\gg, f, k)$ consisting of $d_0$-closed elements of charge zero in $C^-$, see \cite[Rem. 4.2]{KWIII} and also \cite[Rem. 5.11]{DSKII}. This property is called formality. We record these two statements:
\begin{prop}(Corollary of proof of \cite[Thm. 4.1]{KWIII})\label{prop:formality}
\begin{enumerate}
\item $\cW^k(\gg, f) \cong V(\gg^f)$ as $\mathbb C L_0 \oplus  \ga$-modules. 
\item $\cW^k(\gg, f)$ is a subalgebra of $C(\gg, f, k)$ consisting of $d_0$-closed elements of charge zero in $C^-$.
\end{enumerate}
\end{prop}

\subsection{Quasi-classical limits and Poisson $\cW$-algebras}

\begin{defn}\cite[Def. 6.7]{DSKII}
Let $V_\epsilon$ be a vertex superalgebra over $\mathbb C[\epsilon]$. $V_\epsilon$ is a family of Lie conformal superalgebras if 
\[
[X{}_\lambda Y] \in \mathbb C[\lambda] \otimes \epsilon\ \mathbb C[\epsilon] \ V_\epsilon
\] 
for all $X, Y$ in $V_\epsilon$. $V_\epsilon$ is said to be regular, if multiplication by $\epsilon$ has no kernel. 
\end{defn}
One can then take the limit $\epsilon\rightarrow 0$. This is called the classical limit $V^{\text{cl}}$ of $V_\epsilon$, i.e. $V^{\text{cl}} = V_\epsilon/\epsilon V_\epsilon$. $V^{\text{cl}}$ is a commutative vertex superalgebra that inherits a Poisson bracket and thus a Poisson vertex superalgebra structure by setting
\[
[ a {}_\lambda b] = \epsilon \{ a {}_\lambda b\}
\]
and taking the image of $\{ a {}_\lambda b\}$ in $V^{\text{cl}}$.  Denote by $B(a, b)\in \mathbb C[\lambda]$ the constant term of $\{ a {}_\lambda b\}$.

\begin{defn} Let $V_\epsilon$ be a regular family of vertex superalgebras. 
Choose $\sigma$ such that $\sigma^2=\epsilon$ and assume that $V_\epsilon$ is strongly finitely generated by fields $\{X^\alpha\}_{\alpha \in S}$ for some finite index set $S$. Set $\widetilde X^\alpha = \sigma^{-1} X^\alpha$ and denote by $V_\sigma$ the vertex superalgebra generated by the $\{\widetilde X^\alpha\}_{\alpha \in S}$. 
We call $V^{\text{free}}:= V_\sigma/\sigma V_\sigma$ the free field limit of the regular family $V^\epsilon$.
\end{defn}
We will justify the name in a moment and also prove independence of the choice of root of $\epsilon$. 
For nonzero $\epsilon$, this definition is nothing but a rescaling of strong generators, and so
\[
V_\epsilon/(\epsilon-k) \cong V_{\sigma}/(\sigma -\sqrt{k}), \qquad \text{for} \ k \neq 0, 
\]
but 
\[
[\widetilde X{}_\lambda \widetilde Y] \in \mathbb C[\lambda] \otimes  \mathbb C[\sigma]V_\sigma
\]
is not necessarily commutative in the limit $\sigma \rightarrow 0$.

\begin{prop}\label{prop:ff}
$V^{\text{free}}=V_\sigma/\sigma V_\sigma$ is a free field algebra strongly generated by $\{\widetilde X^\alpha\}_{\alpha \in S}$ with $\lambda$-bracket
\[
[\widetilde X{}_\lambda \widetilde Y] = B(X, Y) \in \mathbb C[\lambda].
\]
\end{prop}
\begin{proof}
$[X{}_\lambda Y] $ is a normally ordered polynomial in the strong generators and their iterated derivatives, so we can decompose it into the constant term $b_{X, Y}$ (the multiplicity of the vacuum) and remainder $R_{X, Y}$,
\[
[X{}_\lambda Y] =  \epsilon b_{X, Y}(\lambda, \epsilon) +  \epsilon R_{X, Y}(\lambda, \epsilon).
\]
It follows that
\[
[\widetilde X{}_\lambda \widetilde Y] = b_{X, Y}(\lambda, \epsilon) +  R_{X, Y}(\lambda, \epsilon),
\]
and since $X^\alpha = \sigma \widetilde X^\alpha$ and $R_{X, Y}$ is a polynomial without constant term in the $X^\alpha$, it has the form $\sigma \widetilde R_{X, Y}$ for some polynomial $\widetilde R_{X, Y}$ in $\{\widetilde X^\alpha\}_{\alpha \in S}$ and $\sigma$. It follows that this term vanishes in the limit $\sigma \rightarrow 0$, so $V_\sigma/\sigma V_{\sigma}$ is a free field algebra with pairing 
\begin{equation*} \begin{split}  
\text{span}(  X^\alpha | \alpha \in S) \times \text{span}(  X^\alpha | \alpha \in S) & \rightarrow \mathbb C[\lambda], \\  (X^\alpha, X^\beta) & \mapsto b_{X^\alpha, X^\beta}(\lambda, 0) = B(X^\alpha, X^\beta). \end{split} \end{equation*}
\end{proof}
\begin{cor}\label{cor:GFFsimple}
$V^{\text{free}}=V_\sigma/\sigma V_\sigma$ is simple if and only if the pairing $B$ restricts to a nondegenerate pairing on the strong generators of the Poisson vertex superalgebra $V^{\text{cl}}=V_\epsilon/\epsilon V_{\epsilon}$. 
\end{cor}

\begin{example}\cite[End of Section 6]{DSKII}

Let $\gg$ be a Lie superalgebra, $( \ \ |  \ \ ): \gg \times \gg \rightarrow \mathbb C$ an invariant supersymmetric bilinear form, $\{q^\alpha\}_{\alpha \in S}$ a basis of $\gg$ and $V^\ell(\gg)$ be the corresponding vertex superalgebra at level $\ell$ with strong generators $\{X^\alpha\}_{\alpha \in S}$, so that 
\[
[ X^\alpha {}_\lambda X^\beta] = \lambda \ell (q^\alpha | q^\beta) + {f^{\alpha \beta}}_\gamma X^\gamma.
\]
Fix a nonzero $k$ in $\mathbb C$ and 
define the regular family $V_\epsilon^k(\gg)$ by scaling the $\lambda$-bracket by $\epsilon$, i.e. it is the vertex superalgebra strongly generated by $\{X_\epsilon^\alpha\}_{\alpha \in S}$ with
\[
[ X_\epsilon^\alpha {}_\lambda X_\epsilon^\beta] = \epsilon\left( \lambda k (q^\alpha | q^\beta) + {f^{\alpha \beta}}_\gamma X_\epsilon^\gamma\right).
\]
Set $\ell = \epsilon^{-1}k$ and consider $V^{\ell}(\gg)$.
Set $Y^\alpha := \epsilon X^\alpha$, so that 
\begin{equation*} \begin{split}
[Y^\alpha {}_\lambda Y^\beta] & = [ \epsilon X^\alpha {}_\lambda  \epsilon X^\beta]  = 
\epsilon^2\left( \lambda \ell (q^\alpha | q^\beta) + {f^{\alpha \beta}}_\gamma X^\gamma\right)  
\\ & = \epsilon\left( \lambda k (q^\alpha | q^\beta) + {f^{\alpha \beta}}_\gamma Y^\gamma\right). \end{split} \end{equation*}
Hence $V_\epsilon^k(\gg)/(\epsilon-a) \cong V^{a^{-1}k}(\gg)$ for $a\neq 0$. Thus $V_\epsilon^k(\gg)/\epsilon$ defines a classical limit (i.e. a $k \rightarrow \infty$ limit) of $V^k(\gg)$. A different limit is obtained by setting $\sigma^2=\epsilon$ and setting $Z^\alpha =  \sigma^{-1}Y^\alpha =\sigma X^\alpha$, so that 
\[
[ Z^\alpha {}_\lambda Z^\beta] =  
[ \sigma X^\alpha {}_\lambda \sigma X^\beta] = \sigma^2 \left(\lambda \ell (q^\alpha | q^\beta) + {f^{\alpha \beta}}_\gamma X^\gamma\right) =
\lambda k (q^\alpha | q^\beta) + {f^{\alpha \beta}}_\gamma \sigma Z^\gamma.
\]
We see that with this scaling, the large level limit just gives us the free field vertex superalgebra associated to the vector superspace $\gg$ and the invariant bilinear form $k( \ \ | \ \ )$, that is $\lambda$-bracket
\[
[ Z^\alpha {}_\lambda Z^\beta] =  
\lambda k (q^\alpha | q^\beta).
\]

Similarly, let $F$ be a free field vertex superalgebra, strongly generated by fields $\{\varphi^\alpha\}_{\alpha \in S}$ for some finite index set $S$. Then one defines a corresponding regular family $F_\epsilon$ via the $\lambda$-bracket $[ \ \ {}_\lambda \ \ ]_\epsilon$  as follows:
\[
[\varphi^\alpha {}_\lambda \varphi^\beta]_\epsilon = \epsilon [\varphi^\alpha {}_\lambda \varphi^\beta].
\]
For nonzero $\epsilon$, this just amounts to a rescaling of fields by $\sqrt{\epsilon}$, i.e. $F_\epsilon/(\epsilon-a)\cong F$ for $a\neq 0$. 
\end{example}
The following theorems are a detailed explanation of the last paragraph of Section 6 of \cite{DSKII}
\begin{thm}
Let $\gg$ be a Lie superalgebra with invariant bilinear form $( \ \ | \ \ )$ and $k$ a nonzero complex number. Let $f, x, e$ be an $\gs\gl_2$-triple in $\gg$. 
Let $K=\mathbb C[\epsilon, \epsilon^{-1}]$ and $f_\epsilon = \epsilon^{-1}f, x_\epsilon=x, e_\epsilon =\epsilon e$ be an $\gs\gl_2$-triple in $\gg \otimes_{\mathbb C}K$, so that $(C(\gg, \epsilon^{-1}f, \epsilon^{-1}k), d_0)$ is a complex of vertex superalgebras over $K$. 
Then there exists a regular family of complexes $(C_\epsilon(\gg, f, k), d^\epsilon_0)$ and a vertex superalgebra isomorphism
\[
\Upsilon :C_\epsilon(\gg, f, k) \otimes_{\mathbb C[\epsilon]}\mathbb C[\epsilon, \epsilon^{-1}] \xrightarrow{\cong} C(\gg, \epsilon^{-1}f, \epsilon^{-1}k)
\]
with $\Upsilon(d^\epsilon_0)= d_0$.
\end{thm}
\begin{proof}
Fix now an $\gs\gl_2$-triple $f, x, e$ in $\gg$ and consider the regular family of vertex superalgebras
$$
C_\epsilon(\gg, f, k):=  V_\epsilon^k(\gg) \otimes F_\epsilon(\gg_+) \otimes F_\epsilon(\gg_{\frac{1}{2}})  $$
and compare it to the complex
$
(C(\gg, \epsilon^{-1}f, \epsilon^{-1}k), d_0),
  $
  of vertex superalgebras over the ring $K=\mathbb C[\epsilon, \epsilon^{-1}]$,
  associated to the deformed $\gs\gl_2$-triple $f_\epsilon = \epsilon^{-1}f, x_\epsilon=x, e_\epsilon =\epsilon e$.
 Denote the strong generators of  $C(\gg, \epsilon^{-1}f, \epsilon^{-1}k)$  by $X^\alpha, \varphi^\alpha, \varphi_\alpha, \Phi_\alpha$ as before. 
  Note that the operator product of the charged (super)fermions is independent of $f$, however the one of the neutral ones depends on $f$, see \eqref{eq:neutral}, namely
  \[
  \Phi_\alpha(z) \Phi_\beta(w) \sim \frac{\langle q^\alpha , q^\beta \rangle}{(z-w)}  \sim \frac{ ( f^\epsilon | [q^\alpha, q^\beta] )}{(z-w)}
  \sim \frac{ \epsilon^{-1}( f | [q^\alpha, q^\beta] )}{(z-w)}.
  \]
  Denote the strong generators of the regular family $C_\epsilon(\gg, f, k)$ by $X_\epsilon^\alpha$, $\varphi^\alpha_\epsilon$, $\varphi^\epsilon_\alpha$, $\Phi_\alpha^\epsilon$ and define the map 
  \begin{equation}
  \begin{split}
  \Upsilon(X^\alpha_\epsilon) = \epsilon X^\alpha, \qquad   \Upsilon(\varphi^\alpha_\epsilon) = \varphi^\alpha, \qquad
    \Upsilon(\varphi_{\alpha}^\epsilon) = \epsilon\varphi_\alpha, \qquad   \Upsilon(\Phi_\alpha^\epsilon) = \epsilon \Phi_\alpha. 
  \end{split}
  \end{equation}
$\Upsilon$ preserves the operator products and thus induces a homomorphism $\Upsilon :C_\epsilon(\gg, f, k) \rightarrow C(\gg, \epsilon^{-1}f, \epsilon^{-1}k)$. Since it maps strong free generators to strong free generators it lifts after base change to an isomorphism
\[
\Upsilon :C_\epsilon(\gg, f, k) \otimes_{\mathbb C[\epsilon]}\mathbb C[\epsilon, \epsilon^{-1}] \xrightarrow{\cong} C(\gg, \epsilon^{-1}f, \epsilon^{-1}k)
\]
of vertex superalgebras.
Define $d^\epsilon$ by
\begin{equation}\nonumber
\begin{split}
\epsilon d^\epsilon (z) &= \sum_{\alpha \in S_+} (-1)^{|\alpha|} X^\alpha_\epsilon \varphi^\alpha_\epsilon  
-\frac{1}{2} \sum_{\alpha, \beta, \gamma \in S_+}  (-1)^{|\alpha||\gamma|} {f^{\alpha \beta}}_\gamma \varphi_\gamma^\epsilon \varphi^\alpha_\epsilon \varphi^\beta_\epsilon 
\\ & + 
  \sum_{\alpha \in S_+} (f | q^\alpha) \varphi^\alpha_\epsilon + \sum_{\alpha \in S_{\frac{1}{2}}} \varphi^\alpha_\epsilon \Phi_\alpha^\epsilon\\
\end{split} 
\end{equation}
so that 
  \begin{equation}\nonumber
\begin{split}
\epsilon \Upsilon(d^\epsilon)= & \sum_{\alpha \in S_+} (-1)^{|\alpha|} \epsilon X_\alpha \varphi^\alpha  
-\frac{1}{2} \sum_{\alpha, \beta, \gamma \in S_+}   (-1)^{|\alpha||\gamma|} {f^{\alpha \beta}}_\gamma \epsilon \varphi_\gamma \varphi^\alpha \varphi^\beta 
\\ & + \sum_{\alpha \in S_+} (f | q^\alpha) \varphi^\alpha + \sum_{\alpha \in S_{\frac{1}{2}}} \varphi^\alpha \epsilon\Phi_\alpha^\epsilon \\
= & \ \epsilon \Bigl(  \sum_{\alpha \in S_+} (-1)^{|\alpha|}  X_\alpha \varphi^\alpha  
-\frac{1}{2} \sum_{\alpha, \beta, \gamma \in S_+}   (-1)^{|\alpha||\gamma|} {f^{\alpha \beta}}_\gamma \varphi_\gamma \varphi^\alpha \varphi^\beta
\\ & + \sum_{\alpha \in S_+} \epsilon^{-1} (f | q^\alpha) \varphi^\alpha + \sum_{\alpha \in S_{\frac{1}{2}}} \varphi^\alpha \Phi_\alpha^\epsilon\Bigr) \\
= &\  \epsilon \Bigl(  \sum_{\alpha \in S_+} (-1)^{|\alpha|}  X_\alpha \varphi^\alpha  
-\frac{1}{2} \sum_{\alpha, \beta, \gamma \in S_+}   (-1)^{|\alpha||\gamma|} {f^{\alpha \beta}}_\gamma \varphi_\gamma \varphi^\alpha \varphi^\beta 
\\ & + \sum_{\alpha \in S_+}  (f_\epsilon | q^\alpha) \varphi^\alpha + \sum_{\alpha \in S_{\frac{1}{2}}} \varphi^\alpha \Phi_\alpha^\epsilon\Bigr) \\
= & \ \epsilon d.
\end{split} 
\end{equation}
Note that $d^\epsilon$ is not in $C_\epsilon(\gg, f, k)$ but only in $C_\epsilon(\gg, f, k) \otimes_{\mathbb C[\epsilon]}\mathbb C[\epsilon, \epsilon^{-1}]$.
\end{proof}
Define
\[
\cW^k_\epsilon(\gg, f) =  H(C_\epsilon(\gg, f, k), d^\epsilon_0).
\]
\begin{thm} ${}$
\begin{enumerate}
\item $\cW^k_\epsilon(\gg, f)$ is a regular family of vertex superalgebras.
\item $\cW^k_\epsilon(\gg, f) \otimes_{\mathbb C[\epsilon]}\mathbb C[\epsilon, \epsilon^{-1}] \cong  \cW^{\epsilon^{-1} k}(\gg, \epsilon^{-1}f)$ as vertex superalgebras over $\mathbb C[\epsilon, \epsilon^{-1}]$. 
\end{enumerate}
\end{thm}
\begin{proof}
This theorem follows by repeating the discussion of Section 4 of \cite{KWIII}, but for the complex $(C_\epsilon(\gg, f, k), d^\epsilon_0)$.
We set 
\begin{equation}
\begin{split}
J^\alpha_\epsilon &= X^\alpha_\epsilon + \sum_{\beta, \gamma \in S_+} (-1)^{|\gamma|} {f^{\alpha \beta }}_\gamma :\varphi_\gamma^ \epsilon\varphi^\beta_\epsilon:, \qquad \text{for} \ q^\alpha \in \gg , \\
I^\alpha_\epsilon &= J^\alpha_\epsilon + \frac{(-1)^{|\alpha|}}{2} \sum_{\beta \in S_{\frac{1}{2}}}  {f^{\beta\alpha}}_\gamma :\Phi^\beta_\epsilon\Phi_\gamma^\epsilon:, \qquad \text{for} \ q^\alpha \in \gg_0.
\end{split}
\end{equation}
It follows that $\Upsilon(J^\alpha_\epsilon) =\epsilon J^\alpha$ and $\Upsilon(I^\alpha_\epsilon) =\epsilon I^\alpha$.
Recall that $\Upsilon(X^\alpha_\epsilon) = \epsilon X^\alpha,   \Upsilon(\varphi^\alpha_\epsilon) = \varphi^\alpha, 
    \Upsilon(\varphi_{\alpha}^\epsilon) = \epsilon\varphi_\alpha$ and  $\Upsilon(\Phi_\alpha^\epsilon) = \epsilon \Phi_\alpha$.
 From \cite[Section 4]{KWIII} one has that 
\[
d_0(\epsilon \varphi_\alpha) = \begin{cases}  \epsilon J^\alpha + (-1)^{|\alpha|} \epsilon \Phi_\alpha  & \ \alpha \in S_{\frac{1}{2}} , \\  \epsilon J^\alpha + \epsilon (\epsilon^{-1} f | q^\alpha) = \epsilon J^\alpha + (f | q^\alpha)  & \ \alpha \in S\setminus S_{\frac{1}{2}} ,\end{cases}
\]
and hence using the isomorphism $\Upsilon$
\[
d_0^\epsilon(\varphi_\alpha) = \begin{cases}  J^\alpha_\epsilon + (-1)^{|\alpha|}  \Phi^\epsilon_\alpha & \ \alpha \in S_{\frac{1}{2}} , \\    J^\alpha_\epsilon + (f | q^\alpha)  & \ \alpha \in S\setminus S_{\frac{1}{2}}. \end{cases}
\]
From \cite[(2.6), (4.3), (4.4)]{KWIII}
\begin{equation}\nonumber
\begin{split}
d_0(\epsilon J^\alpha) &=  \sum_{\beta \in S_+} \epsilon ([\epsilon^{-1} f, q^\alpha], q^\beta)\varphi^\beta + \sum_{\substack{ \beta \in S_+ \\ \gamma \in S_{\frac{1}{2}}}} (-1)^{|\alpha|(|\beta|+1)} {f^{\alpha\beta}}_\gamma \varphi^\beta \epsilon\Phi_\gamma \\
&  - \sum_{\substack{\beta \in S_+ \\ \gamma \in S\setminus S_+}} (-1)^{|\beta|(|\alpha|+1)}{f^{\alpha\beta}}_\gamma \varphi^\beta \epsilon J^\gamma  
\\ & +
 \sum_{\beta \in S_+} \epsilon \left( \epsilon^{-1} k ( q^\alpha | q^\beta )  +  \text{str}_{\gg_+}\left( p_+ ( \text{ad}(q^\alpha)) \text{ad}(q^\beta) \right)\right) \partial\varphi^\beta ,\\
 d_0(\varphi^\alpha) &= -\frac{1}{2} \sum_{\beta, \gamma \in S_+} (-1)^{|\alpha||\gamma|} {f^{\beta\gamma}}_{\alpha} :\varphi^\beta\varphi^\gamma: ,\\
 d_0(\epsilon \Phi_\alpha) &= \sum_{\beta\in S_{\frac{1}{2}}} \epsilon (\epsilon^{-1} f | [q^\beta, q^\alpha]) \varphi^\beta =  \sum_{\beta\in S_{\frac{1}{2}}} ( f | [q^\beta, q^\alpha]) \varphi^\beta ,
\end{split}
\end{equation}
and hence using the isomorphism $\Upsilon$,
\begin{equation}\nonumber
\begin{split}
d^\epsilon_0(J^\alpha_\epsilon) &=  \sum_{\beta \in S_+}  ([ f, q^\alpha], q^\beta)\varphi_\epsilon^\beta + \sum_{\substack{ \beta \in S_+ \\ \gamma \in S_{\frac{1}{2}}}} (-1)^{|\alpha|(|\beta|+1)} {f^{\alpha\beta}}_\gamma \varphi^\beta_\epsilon \Phi_\gamma^\epsilon \\
&  - \sum_{\substack{\beta \in S_+ \\ \gamma \in S\setminus S_+}} (-1)^{|\beta|(|\alpha|+1)}{f^{\alpha\beta}}_\gamma \varphi^\beta_\epsilon J^\gamma_\epsilon 
\\ & + \sum_{\beta \in S_+} \left(  k ( q^\alpha | q^\beta )  +  \epsilon\,\text{str}_{\gg_+}\left( p_+ ( \text{ad}(q^\alpha)) \text{ad}(q^\beta) \right) \right)\partial\varphi^\beta_\epsilon, \\
 d^\epsilon_0(\varphi^\alpha_\epsilon) &= -\frac{1}{2} \sum_{\beta, \gamma \in S_+} (-1)^{|\alpha||\gamma|} {f^{\beta\gamma}}_{\alpha} :\varphi^\beta_\epsilon\varphi^\gamma_\epsilon: ,\\
 d^\epsilon_0(\Phi^\epsilon_\alpha) &=   \sum_{\beta\in S_{\frac{1}{2}}} ( f | [q^\beta, q^\alpha]) \varphi^\beta_\epsilon.
\end{split}
\end{equation}
Let $C^+_\epsilon$ be the subalgebra generated by $\varphi_\alpha^\epsilon, d_0(\varphi_\alpha^\epsilon)$ for $\alpha$ in $S_+$ and $C^-_\epsilon$ be the one generated by $J^\alpha_\epsilon$ for $\alpha \in S\setminus S_+$, $\varphi^\alpha_\epsilon$ for $\alpha$ in $S_+$ and $\Phi_\alpha^\epsilon$ for $\alpha$ in $S_{\frac{1}{2}}$. From \cite[Lemma 2.1]{KRW} it follows that
\[
[d_0(\epsilon\varphi_\alpha) {}_\lambda \epsilon\varphi_\beta ] = \epsilon  (-1)^{|\alpha|} \sum_{\gamma \in S_+} {f^{\alpha \beta}}_\gamma \epsilon\varphi_\gamma,
\]
and hence
\[
[d^\epsilon_0(\varphi^\epsilon_\alpha) {}_\lambda \varphi^\epsilon_\beta ] = \epsilon  (-1)^{|\alpha|} \sum_{\gamma \in S_+} {f^{\alpha \beta}}_\gamma \varphi^\epsilon_\gamma.
\]
One thus has that $C_\epsilon(\gg, f, k) \cong C^+_\epsilon \otimes C^-_\epsilon$ as $d_0^\epsilon$-invariant families of regular subalgebras. The homology of $C^+$ is clearly one-dimensional and hence
$\cW^k_\epsilon(\gg, f) = H(C^-_\epsilon, d_0^\epsilon)$.
Theorem 4.1 of \cite{KWIII} holds for $\cW^k_\epsilon(\gg, f)$ as the set-up and hence the proof is exactly the same as for  $\cW^k(\gg, f)$. In particular, the conclusion, Proposition \ref{prop:formality}, holds:
\begin{enumerate}
\item $\cW_\epsilon^k(\gg, f) \cong V_\epsilon(\gg^f)$ as $\mathbb C L_0 \oplus  \ga$-modules. 
\item $\cW_\epsilon^k(\gg, f)$ is a subalgebra of $C_\epsilon(\gg, f, k)$ consisting of $d^\epsilon_0$-closed elements of charge zero in $C_\epsilon^-$.
\end{enumerate}
A subalgebra of a regular family is of course a regular family as well and hence our first claim follows. The second one is true, since specialization commutes with restriction to a subalgebra. 
\end{proof}
\begin{cor} \label{cor:rescaling}
$\cW^k_\epsilon(\gg, f)/(\epsilon-a) \cong \cW^{a^{-1} k}(\gg, f)$ for $a\neq 0$. 
\end{cor}
\begin{proof}
Specializing the last theorem gives $\cW^k_\epsilon(\gg, f)/(\epsilon-a) \cong \cW^{a^{-1} k}(\gg, a^{-1}f)$. Moreover $\cW^{k}(\gg, f) \cong \cW^{k}(\gg, \omega(f))$ for any even automorphism $\omega$ of $\gg$ that preserves the bilinear form. The claim follows with $\omega = \text{ad}(e^{\mu x})$ for  $a =e^\mu$. 
\end{proof}
Let $J^f$ be an index set so that $\{q_j\}_{j \in J^f}$ is a basis of $\gg^f$ and $\{q^j\}_{j \in J^f}$ is a basis if $\gg^e$ such that $(q_i|q^j)=\delta_{i, j}$. The basis is chosen to consist of ad $x$ eigenvectors and the eigenvalue of $q^j$ is denoted by $\delta(j)$, that is $[x, q^j]=\delta(j)q^j$ and hence $[x, q_j] = -\delta(j)q_j$. 
One then sets 
\begin{equation*} \begin{split} 
& J^f_{-k} = \{ j \in J^f | \delta(j) = k\}, 
\\ &  J_{-k} = \{ (j, n) \in J^f \times \mathbb Z_+ | \ \delta(j) -|k| \in \mathbb Z_+, \ k = \delta(j) - n \},
\end{split} \end{equation*}
so that the sets
\[
S_k := 
 \{ q^j_n \ |\ (j, n) \in J^f_{-k},  \  q^j_n = (\text{ad}\ f)^n q^j, \ n= 0, \dots, 2\delta(j)\}
\]
form  a basis of $\gg_k$. The dual basis $S^k$ is denoted by $\{q_j^n\}$. Finally, for $a$ in $\gg$, $a^\sharp$ is its projection on $\gg^f$, that is
\[
a^\sharp = \sum_{j \in J^f} (a| q_j)q^j.
\]
\begin{defn}
Define $$B_k : \gg_{-k}^f \times \gg_{-k}^f \rightarrow \mathbb C,\qquad B_k(a, b):= (-1)^{2k} (({\rm ad}\ f)^{2k}b | a)$$ for every $k \in \frac{1}{2} \mathbb Z_+$, and extend it linearly to $B:\gg^f \times \gg^f \rightarrow \mathbb C$. 
\end{defn}
Since $( \ \ |  \ \ )$ is invariant and supersymmetric and since $\gg^f$ is a Lie subsuperalgebra of $\gg$ it follows easily that $B$ is invariant and $B_k$ is supersymmetric for $k \in \mathbb Z$ and antisupersymmetric for $k \in \mathbb Z +\frac{1}{2}$. Moreover if $( \ \ | \ \ )$ is a nondegenerate bilinear form on $\gg$, then each $B_k$ and hence $B$ is nondegenerate as well since the dual of a lowest weight vector $a$ of the $2k+1$-dimensional irreducible representation $\rho_{2k}$ of $\gs\gl_2$ is necessarily a highest-weight vector of $\rho_{2k}$ and $({\rm ad}\ f)^{2k}$ precisely maps such lowest weight vectors to highest-weight ones. 

As before, fix $\sigma$ so that $\sigma^2 = \epsilon$, and let $\cW^k_{\sigma}(\gg,f)$ denote the vertex algebra with the same strong generators as $\cW^k_{\epsilon}(\gg,f)$, but rescaled by $\sigma^{-1}$. Let $\cW^{\rm free}(\gg, f) := \cW^k_{\sigma}(\gg,f) / \sigma \cW^k_{\sigma}(\gg,f)$. The reason for this notation is that if $k$ is any fixed nonzero constant, $\cW^{\rm free}(\gg, f)$ is independent of $k$ due to Corollary \ref{cor:rescaling}.

\begin{thm} \label{thm:nondegeneracyw} 
Let $\gg$ be a Lie superalgebra with invariant bilinear form $( \ \ | \ \ )$, and let $f \in \gg$ be a nilpotent element. Then $\cW^{\rm{free}}(\gg, f)$ is a free field algebra with strong generators $X^\alpha$, $\alpha \in J^f$ and $\lambda$-brackets
\[
[X^\alpha {}_\lambda X^\beta] = \delta_{j, k}\lambda^{2k+1} B_k(q^\alpha, q^\beta)
\]
for $q^\alpha \in \gg^f_{-k}$ and $q^\beta \in \gg^f_{-j}$. 

\end{thm}
\begin{proof}
There are two notions of Poisson $\cW$-superalgebras: the homological construction that we presented here, and the so-called classical Hamiltonian reduction; see e.g. \cite{DSKII}. These two coincide by \cite{SuhI} for $\gg$ a Lie algebra, and the general Lie superalgebra case is \cite[Thm. 3.11]{SuhII}.
The classical Hamiltonian reduction approach is suited for explicit computations which are presented in \cite{DSKVIII}.
We need to explain \cite[Thm. 5.3]{DSKVIII}. Let $\mathcal W(\gg, f)$ be the Poisson vertex algebra of $\gg$ associated to $f$. Then there is a one-to-one map $\omega : V(\gg^f) \rightarrow \mathcal W(\gg, f)$ and the Poisson bracket is given explicitly in \cite[Thm. 5.3]{DSKVIII}. using this map. De Sole and Kac introduce the notation $h \prec k$ if and only if $h\leq  k-1$ and $\vec{k}:=(k_1, \dots, k_t)$, $J_{-\vec{k}}$ is the set of elements of the form $(\vec{j}, \vec{n})= (j_1, n_1) \times \dots \times (j_t, n_t)$ with $(j_i, n_i) \in J_{-k_i}$.  
For $a \in  \gg^f_{-h}$ and $b \in \gg^f_{-k}$, the $\lambda$-bracket is
\begin{equation}
\begin{split}
& \{\omega(a)_\lambda\omega(b)\}_{z,\rho} = \omega([a, b]) + (a|b)\lambda + z(s|[a, b])\\
&-\sum_{t=1}^\infty \sum_{\substack{ (\vec{j}, \vec{n}) \in J_{-\vec{k}}  \\-h+1\leq k_t \prec \dots \prec k_1 \leq k}}    
\left(  \omega([b, q_{n_1}^{j_1}]^\sharp) - (b| q_{n_1}^{j_1})(\lambda +\partial) + z(s| [b, q_{n_1}^{j_1}])  \right) \\
&\qquad\prod_{i=1}^{t-1} \left(  \omega([q_{j_i}^{n_i+1}, q_{n_{i+1}}^{j_{i+1}}]^\sharp) - (q_{j_i}^{n_i+1}| q_{n_{i+1}}^{j_{i+1}})(\lambda +\partial) + z(s| [q_{j_i}^{n_i+1}, q_{n_{i+1}}^{j_{i+1}}])  \right)\\
& \qquad\qquad\left(  \omega([q_{j_t}^{n_t+1}, a]^\sharp) - (q_{j_t}^{n_t+1}| a)\lambda + z(s| [q_{j_t}^{n_t+1}, a])  \right).
\end{split}
\end{equation}
Here $z \in \mathbb C$ and $s \in \gg$ can be chosen arbitrarily as they give rise to isomorphic Poisson vertex algebra structures. 
This sum is actually finite, since contributions for $t>k+h$ are zero. We are interested in the leading coefficient, i.e. the one of $\lambda^{k+h+1}$. Let $b=q_j$. Then this coefficient vanishes unless we take the summand with $j_i=j$, $n_i=i-1$ and $k_i = k +1-i$. It follows that 
\begin{equation}\label{eq:primary}
\begin{split}
\{\omega(a)_\lambda\omega(b)\}_{z,\rho} &=   \delta_{k, h}(-1)^{2k} (q_j^{2k}|a)\lambda^{2k+1} + \dots 
\end{split}
\end{equation}
The claim follows using Proposition \ref{prop:ff}.
\end{proof}
\begin{cor}
Let $\gg$ be a Lie superalgebra with invariant supersymmetric bilinear form $( \ \ | \ \ )$, and let $f \in \gg$ be a nilpotent element. Then $\cW^{\rm{free}}(\gg, f)$ is simple if and only if $( \ \ | \ \ )$ is nondegenerate.
\end{cor}
Also note that if $a$ is of conformal weight one, then \cite[Eq. 6.2]{DSKVIII}
\begin{equation}
\begin{split}
\{\omega(a)_\lambda\omega(b)\}_{z,\rho} &= \omega([a, b]) + (a|b)\lambda + z(s|[a, b]).
\end{split}
\end{equation}

\begin{cor} \label{cor:gfffadecomposition} Let $\gg$ be a Lie superalgebra with invariant nondegenerate bilinear form $(\ \ |\ \ )$, and let $f$ be a nilpotent element in $\gg$. Then $\cW^{\rm{free}}(\gg,f)$ decomposes as a tensor product of the standard free field algebras $\cS_{\text{ev}}(n,k)$, $\cO_{\text{ev}}(n,k)$, $\cS_{\text{odd}}(n,k)$, and $\cO_{\text{odd}}(n,k)$ introduced earlier.
\end{cor}

\begin{proof} For each $r$, we partition $J^f_{-r}$ into subsets $J^f_{-r, \text{ev}}$ and $J^f_{-r,\text{odd}}$ consisting of elements which are even and odd, respectively.

Suppose $r \in \frac{1}{2} \mathbb{Z}$. Then in the $\sigma \rightarrow 0$ limit, the even fields $\{X^{\alpha} | \ \alpha \in J^f_{-r, \text{ev}}\}$ of weight $r+1$ must have nondegenerate symplectic pairing given by Theorem \ref{thm:nondegeneracyw}. Therefore the cardinality of $J^f_{-r, \text{ev}}$ is an even integer $2p_{\text{ev},r}$, and these fields generate a copy of the free field algebra $\cS_{\text{ev}}(p_{\text{ev},r},2r+2)$ of even symplectic type. Next, if $r\in \mathbb{Z}$, then the fields $\{X^{\alpha}| \ \alpha \in J^f_{-r, \text{ev}}\}$ in the $\sigma\rightarrow 0$ limit must give an even algebra of orthogonal type, $\cO_{\text{ev}}(q_{\text{ev},r}, 2r+2)$ where $q_{\text{ev},r}$ is the cardinality of $J^f_{-r, \text{ev}}$.

Similarly, suppose  $r \in \frac{1}{2} \mathbb{Z}$, and consider the fields $\{X^{\alpha} | \ \alpha \in J^f_{-r, \text{odd}}\}$ in the $\sigma \rightarrow 0$ limit. These fields then generate a copy of $\cO_{\text{odd}} (s_{\text{odd},r}, 2r+2)$, where $q_{\text{odd},r}$ is the cardinality of $J^f_{-r, \text{odd}}$. Suppose next that $r \in \mathbb{Z}$, and consider the fields $\{X^{\alpha}| \ \alpha \in J^f_{-r, \text{odd}}\}$. Then in the  $\sigma \rightarrow 0$ limit, these fields generate a copy of $\cS_{\text{odd}}(p_{\text{odd,r}}, 2r+2)$. Here $2p_{\text{odd,r}}$ is the cardinality of $J^f_{-r, \text{odd}}$. 

Finally, the fact that in the $\sigma \rightarrow 0$ limit these free field algebras are mutually commuting follows from Theorem \ref{thm:nondegeneracyw}.
\end{proof}

\subsection{Simplicity}

Recall that by Corollary \ref{cor:rescaling}, $\cW^{\rm free}(\gg,f) = \cW^k_\sigma(\gg, f) / \sigma \cW^k_\sigma(\gg, f)$ is independent of $k$ for all nonzero $k$, and additionally all algebras $\cW^r(\gg,f)$ can be obtained by specializing $\cW^k_\sigma(\gg, f)$ to a suitable choice of $\sigma$. For convenience, we may take $k=1$.

\begin{thm} \label{thm:genericsimplicity} 
Let $\gg$ be a Lie superalgebra with nondegenerate invariant bilinear form $( \ \ | \ \ )$, and let $f$ be a nilpotent element in $\gg$. Then 
\begin{enumerate}
\item $\cW^k(\gg, f)$ is simple for generic $k$.
\item Let $V^{\ell'}(\gb)$ be the subalgebra of the affine subalgebra $V^{\ell}(\ga)\subseteq \cW^k(\gg,f)$, where $\gb$ is a reductive Lie subalgebra of $\ga$. Then $\text{Com}(V^{\ell'}(\gb), \cW^k(\gg,f))$ is simple for generic values of $k$.
\item If $\gg$ is a semisimple Lie algebra and $( \ \ | \ \ )$ is the Killing form normalized in the usual way, then $\cW^k(\gg, f)$ is simple for all non-rational $k$. 
\end{enumerate}
\end{thm}
\begin{proof}${}$
\begin{enumerate}
\item  If $\cW^1_\sigma(\gg, f)$ is not simple as a vertex algebra over $\mathbb{C}[\sigma]$ (equivalently, $\cW^k(\gg,f)$ is not simple for generic values of $k$), there is a nontrivial vertex algebra ideal $\cI \subseteq \cW^1_\sigma(\gg, f)$ which has minimal weight component in some weight $d > 0$. Fix a nontrivial singular vector in weight $d$. After suitable rescaling, this vector must remain nontrivial in the $\sigma \rightarrow 0$ limit and also must be a singular vector in the limit algebra $\cW^{\rm free}(\gg,f)$, since OPEs are continuous in the parameter $\sigma$. This contradicts the simplicity of  $\cW^{\rm free}(\gg,f)$.

\item This is immediate from the generic simplicity of $\cW^k(\gg,f)$ together with \cite[Lemma 2.1]{ACKL}. 

\item As explained around \cite[Eq. 51]{ArIV}  the equivariant $\cW$-algebra of $\gg$ associated to $f$ and $\kappa$, denoted by $\cW^\kappa_{G, f}$, is  a strict chiralization of a smooth symplectic variety and hence simple by \cite[Cor. 9.3]{AM}.  $\cW^k(\gg, f) \cong \text{Com}( V^{-k-2h^\vee}(\gg), \cW^\kappa_{G, f})$ for $k +h^\vee =\kappa$ \cite[Prop. 6.5]{ArIV} and $\cW^\kappa_{G, f}$ is an object in $KL_{-k-2h^\vee}(\gg)$  \cite[Prop. 6.6]{ArIV}. Since for non-rational levels $KL_k(\gg)$ is semisimple the claim follows from \cite[Prop. 5.4]{CGN}.
\end{enumerate}
\end{proof}

\subsection{Properties of strong generators}
Let $\gg$ be a simple Lie superalgebra. We recall the definition of the Kazhdan-Lusztig category $KL_k(\gg)$; see e.g. \cite{ArI}.
\begin{defn}\label{def:KL}
The category $KL_k= KL_k(\gg)$ is the full subcategory of $\widehat \gg$-modules that satisfy the following properties.
\begin{enumerate}
\item The central element $K \in \widehat{\gg}$ acts by multiplication by the scalar $k$. 
\item Any object $M$ of $KL_k$ is graded by conformal weight with finite-dimensional weight spaces
\[
M = \bigoplus_{n \in \mathbb C}M_n, \qquad  \text{dim} \ M_n < \infty,
\] 
such that conformal weight of any object is lower bounded, that is $M_n = 0$ unless Re$(n) > N$ for some bound $N$. 
\item There exists a finite set of numbers $h_1, \dots, h_s$, such that $M_n= 0$ unless $n\in \mathbb Z_{\geq 0} +h_i$ for some $i \in \{1, \dots, s\}$. 
\end{enumerate}
A module in $KL_k$ is called {\it almost simple} if every submodule intersects the top level non-trivially, and $KL_k$ is called {\it almost semisimple} if every indecomposable module is almost simple. 
\end{defn}
\begin{defn}
Let $\ga$ be a Lie superalgebra which is the sum of a reductive Lie algebra and finitely many simple Lie superalgebras, such that the bilinear form on each simple summand is normalized in the standard way. Let $M$ be an indecomposable module for the corresponding affine vertex superalgebra $V^k(\ga)$, such that
\[
M = \bigoplus_{n \in \mathbb Z_+ + h} M_n 
\]
for some complex $h$ and $M_h$ nonzero. 
A vector in $M_h$ is called an $\ga$-primary vector and a vector in $\bigoplus\limits_{n \in \mathbb Z_+ + h} M_{n+1}$ is called an $\ga$-descendent. Fields corresponding to primary/descendant vectors are called primary/descendant fields. 
\end{defn}
By an $\ga$-module we always mean a weight module, that is, the Cartan subalgebra of $\ga$ acts semisimply and weight spaces are finite-dimensional. For example, if $\ga$ is a simple Lie algebra, then modules are integrable highest-weight modules. By a projective module, we mean a projective module in the category of weight modules. 
\begin{lemma}\label{lemma:primary}
Let 
\[
V = \bigoplus_{n \in \frac{1}{2}\mathbb Z_+} V_n, \qquad {\rm dim}\ V_0 = 1, \qquad V_{\frac{1}{2}}=0\quad {\rm and} \quad {\rm dim} \ V_n < \infty
\]
be a vertex superalgebra graded by conformal weight.
 Let $I_1, \dots, I_d$ be finite sets, such that 
\[
\bigcup_{i=1}^d   S_i, \qquad S_i = \{ X_j |\  j \in I_i, \ X_j \in V_i \},
\]
is a minimal strong generating set of $V$ with $S_1$ a basis of $V_1$ generating an affine vertex superalgebra $V^k(\ga)$, where $\ga$ is as above. Suppose that $KL_k(\ga)$ is almost semisimple, and suppose that each $V_n$ is a projective $\ga$-module for $n>1$. Then there exists a minimal strong generating set 
\[
\bigcup_{i=1}^d   \widetilde S_i, \qquad \widetilde S_i = \{ \widetilde X_j |\  j \in I_i, \ \widetilde X_j \in V_i \},  
\]
such that $\widetilde S_1 = S_1$, the fields $\widetilde X_j$  in $\widetilde S_i$ for $i\geq 2$ are all $\ga$-primary fields, and their linear span is a projective $\ga$-module. 
\end{lemma}
\begin{proof}
We prove the statement by induction. Let $n\geq 1$ and let  
\[
 S_{\leq n} := \bigcup_{i=1}^n    S_i, \qquad  S_i = \{ X_j | \ j \in I_i, \ X_j \in V_i \}.  
\]
Assume that the $X_j$ in $S_i$ for $1< i\leq n$ can be changed to primary fields $\widetilde X_j$, such that their span is a projective $\ga$-module. Then there exists a set
\[
\widetilde S_{\leq n} := \bigcup_{i=1}^n   \widetilde S_i, \qquad \widetilde S_i = \{ \widetilde X_j |\  j \in I_i,\ \widetilde X_j \in V_i \},  
\]
such that the subspace spanned by normally ordered words in iterated derivatives of the elements in $S_{\leq n}$ and $\widetilde S_{\leq n}$ coincide. 
Denote this subspace by $W$. 
Since the $V_n$ for $n>1$ and also the $\widetilde X_j$ form projective $\ga$-modules, $W$ is also an infinite direct sum of projective $\ga$-modules. Let $W_m$ be the subspace of $W$ of elements of conformal weight $m$. Note that $W_m = V_m$ for $m\leq n$. $W_{n+1}$ is a projective $\ga$-submodule of $V_{n+1}$ and hence a direct summand
\[
V_{n+1} = W_{n+1} \oplus U_{n+1}
\]
for some projective $\ga$-module $U_{n+1}$. Every element in $W_{n+1}$ is by construction an $\ga$-descendant and especially a normally ordered word in the iterated derivatives of the elements of $S_{\leq n}$. It follows that none of the $X_j$ for $j\in S_{n+1}$ is in $W_{n+1}$ and hence the projection of $X_j$ onto $U_{n+1}$ is nontrivial. Denote this projection by $\widetilde X_j$. Since $W_{n+1}$ together with the $X_j$ for $j\in S_{n+1}$ must span $V_{n+1}$, the $\widetilde X_j$ for $j\in S_{n+1}$ must span $U_{n+1}$. They also must be linearly independent since a linear relation 
$ \sum a_j \widetilde X_j =0$ 
implies that the corresponding linear combination $\sum a_j  X_j $ lies in $W$, which contradicts the minimal generating assumption. 

If $U_{n+1}$ contains a vector $v$ that is not $\ga$-primary, then $x_{(1)}v \neq 0$ for some $x \in \ga$, in particular $x_{(1)} v \in W_n \subseteq W$. 
Let $M$ be the $V^k(\ga)$-module generated by $v$. This means that $x_{(1)} v$ must generate a proper $V^k(\ga)$-submodule of $M$.
This cannot happen if $KL_k(\ga)$ is almost semisimple.
\end{proof}

\begin{remark} \label{rem:primarycorrection} ${}$ 
\begin{enumerate}
\item Suppose that $\ga$ is the sum of a reductive Lie algebra and finitely many factors of the form $\go\gs\gp_{1|2n}$, and that $k$ is irrational. Then $KL_k(\ga)$ is semisimple and each $V_n$ is completely reducible as a $V^k(\ga)$-module, so the hypotheses of Lemma \ref{lemma:primary} are satisfied.
\item  If $\ga$ is the sum of a reductive Lie algebra and finitely many simple Lie superalgebras whose even subalgebras are semisimple (and not of type $\gd(2, 1 ; \alpha)$ and $\alpha$ irrational), and $k$ is irrational, then $KL_k(\ga)$ is almost semisimple and again the hypotheses are satisfied. 
The reason for this is as follows. First, it suffices to assume that $\ga$ is a simple Lie superalgebra with semisimple even subalgebra, and $\ga \neq \gd(2, 1 ; \alpha)$ with $\alpha$ irrational. Let $M$ be an indecomposable object in $KL_k(\ga)$. Then the top level of $M$ is finite-dimensional and thus must contain a highest-weight vector for $\ga$ of say highest-weight $\lambda$. The conformal weight of the top level is $\frac{\lambda (\lambda + 2 \rho)}{2(k+ h^\vee)}$. If $M$ has a proper submodule, say $N$, then the top level of $N$ also must contain a highest-weight vector $\mu$, and $\mu$ needs to be in the same coset of the root lattice $Q$ as $\lambda$, say $\mu = \lambda + \beta$ for some $\beta \in Q$. On the other hand, the difference between the conformal weight of the top level of $M$ and $N$ must be a (negative) integer $n$. We thus get the condition
\[
 n = \frac{\lambda (\lambda + 2 \rho)}{2(k+ h^\vee)} - \frac{\mu (\mu + 2 \rho)}{2(k+ h^\vee)} = - \frac{\beta  (2\lambda + \beta  + 2 \rho)}{2(k+ h^\vee)}
\]
that can only hold for rational $k + h^\vee$. Here we use that $\lambda \beta \in \mathbb Q$ for a highest-weight $\lambda$ of a finite dimensional highest-weight representation of $\ga$. This last statement does not in general hold for $\ga$  a simple Lie superalgebra with only reductive even subalgebra, and it only holds for $\gd(2, 1 ; \alpha)$ if $\alpha$ is rational.
\item The condition that $KL_k(\ga)$ be almost semisimple in Lemma \ref{lemma:primary} can be weakened to $V$ being an object in a completion of an almost semisimple subcategory of $KL_k(\ga)$. 
\end{enumerate}
\end{remark}

Suppose next that $V = \bigoplus_{n \in \frac{1}{2}\mathbb Z_+} V_n$ is a vertex superalgebra graded by conformal weight, satisfying the same conditions as Lemma \ref{lemma:primary}, except that $k$ is now regarded as a formal parameter. Assume that all structure constants appearing in the OPEs of the strong generators $\bigcup_{i=1}^d   S_i$ for $V$ are rational functions of $k$. Since there are only finitely many structure constants, the set $D$ of possible poles of the structure constants is finite, and we can regard $V$ as a vertex superalgebra over the ring of rational functions in $k$ with poles along $D$.

\begin{cor}\label{cor:primarygeneralcase}
Let $V = \bigoplus_{n \in \frac{1}{2}\mathbb Z_+} V_n$ be a vertex superalgebra defined over the ring of rational functions in $k$ with poles along $D$, where $V_1$ generates $V^k(\ga)$, as above. Suppose that $\ga$ satisfies either condition (1) or (2) of Remark \ref{rem:primarycorrection}. Then there exists a finite set $D'$ containing $D$ such that over the ring of rational functions in $k$ with poles along $D'$, we can replace the minimal strong generating set $\bigcup_{i=1}^d   S_i$ with a minimal strong generating set $\bigcup_{i=1}^d   \widetilde S_i$
such that $\widetilde S_1 = S_1$ and the fields $\widetilde X_j$  in $\widetilde S_i$ for $i\geq 2$ are $\ga$-primary, and their linear span is an $\ga$-module. 
\end{cor}

\begin{proof} For an irrational value of $k \in \mathbb{C} \setminus D$, by Lemma \ref{lemma:primary} there exist such corrections $\widetilde X_j = X_j + \cdots$, where the remaining terms are normally ordered monomials in the strong generators of lower weight. The coefficient of each such monomial must depend continuously on $k$, and since all structure constants among the OPEs of the generators are rational functions of $k$, these coefficients must in fact be rational functions of $k$. Note that the poles are no longer are required to lie in $D$, but since there are only finitely many such structure constants the new pole set $D'$ is still finite. Finally, note that the primary fields $\widetilde X_j$ exist for all $k \in \mathbb{C} \setminus D'$ even though $KL_k(\ga)$ need not be almost semisimple for nongeneric $k$. \end{proof}

Let $\gg$ be a Lie superalgebra with invariant nondegenerate bilinear form $(\ |\ )$, and let $f \in \gg$ be a nilpotent such that $\cW^k(\gg, f)$ has affine subalgebra $V^{\ell}(\ga)$, and $\ga$ satisfies either condition (1) or (2) of Remark \ref{rem:primarycorrection}. Recall that the strong generators of $\cW^k(\gg, f)$ are indexed by $J^f$, a basis of $\gg^f$. Moreover those strong generators that have conformal weight $r+1$ correspond to $\gg^f_{-r}$ which is ad $(\ga)$-invariant and hence an $\ga$-module that we denote by $M_r$.
 
\begin{cor}\label{cor:primary} Let $\cW^k(\gg, f)$ and $\ga$ be as above.
\begin{enumerate}
\item If we regard $\cW^k(\gg, f)$ as a one-parameter vertex algebra, the strong generators of conformal weight $r+1$ for $r>0$ can be chosen to be primary for the affine subalgebra $V^{\ell}(\ga)$. In particular, this holds for all but finitely values of $k$
\item If the trivial representation appears as a direct summand of $M_r$, then the corresponding strong generator is a field of ${\rm Com}(V^\ell(\ga), \cW^k(\gg, f))$. 
\end{enumerate}
\end{cor}

\begin{proof} The first statement is immediate from Corollary \ref{cor:primarygeneralcase}. For the second statement, a primary field of the trivial representation is a vacuum vector for the affine subalgebra, and hence lies in ${\rm Com}(V^\ell(\ga), \cW^k(\gg, f))$. 
\end{proof}

\begin{lemma} \label{lemma:norm} Let $\gg$ be a Lie superalgebra with invariant nondegenerate bilinear form $(\ |\ )$, and let $f$ be a nilpotent element in $\gg$ such that $\ga$ satisfies either condition (1) or (2) of Remark \ref{rem:primarycorrection}. Let $\{X^{\alpha}|\ \alpha \in J^f\}$ denote the strong generating set for $\cW^1_{\sigma}(\gg,f)$ which satisfies the $\lambda$-brackets of Theorem \ref{thm:nondegeneracyw} in the $\sigma \rightarrow 0$ limit. If we replace the fields $X^{\alpha}$ with corrected fields $\tilde{X}^{\alpha}$ which are $\ga$-primary as above, we again have $\lambda$-brackets 
$$
[\tilde{X}^\alpha {}_\lambda \tilde{X}^\beta] = \delta_{j, k}\lambda^{2k+1} B_k(q^\alpha, q^\beta)
$$ in the $\sigma \rightarrow 0$ limit.
\end{lemma}

\begin{proof} First, consider the fields $\{X^{\alpha}|\ \alpha \in J^f_{-1}\}$ of weight $2$. The corrected fields $\tilde{X}^{\alpha}$ must have the form
$$\tilde{X}^{\alpha} = X^{\alpha} + P$$ where $P \in V^{\ell}(\ga)$. Note that the coefficients of the monomials appearing in $P$ need not be polynomials in $\sigma$, but can be rational functions of $\sigma$. We claim that the coefficient of each such monomial must vanish in the $\sigma \rightarrow 0$ limit. This is apparent because $V^{\ell}(\ga)$ commutes with $X^{\alpha}$ in the $\sigma \rightarrow 0$ limit, but if $P$ has any term whose coefficient does not vanish, we can find some $j \in \ga$ such that $j_{(1)} P$ is an element of weight $1$ that also does not vanish in this limit. This would contradict the $\ga$-primality of $\tilde{X}^{\alpha}$, so we conclude that $\tilde{X}^{\alpha} = X^{\alpha}$ in the $\sigma \rightarrow 0$ limit, as desired.

Now we assume this statement for all fields $X^{\alpha}$ of weight $r \leq m$, and now we consider the corrected field $\tilde{X}^{\alpha}$ for some  $\alpha \in J^f_{-m}$, so that this field has weight $m+1$. Without loss of generality, we may write
$$\tilde{X}^{\alpha} = X^{\alpha} + P$$ where $P$ is a linear combination of monomials in the generators of $V^{\ell}(\ga)$ as well as the $\ga$-primary fields $\tilde{X}^{\beta}$ where $\beta \in J^f_{-r}$ for $r < m$. Using the $\ga$-primality of $\tilde{X}^{\alpha}$ and the fields $\tilde{X}^{\beta}$, as well as the fact that $V^{\ell}(\ga)$ commutes with $X^{\alpha}$ in the $\sigma \rightarrow 0$ limit, it follows by repeated applications of \eqref{ncw} that the coefficients of all monomials appearing in $P$ must vanish in the $\sigma \rightarrow 0$ limit. The claim follows.
\end{proof}

\subsection{Modules and characters}

Let $\gg$ be a Lie superalgebra with nondegenerate bilinear form as before. Let
\[
\widehat \gg = \gg \otimes_{\mathbb C}\mathbb C[t, t^{-1}] \oplus \mathbb C K
\]
be its affinization with $K$ the central element. There is also a derivation which is identified with $-L_0$, the zero mode of the Sugawara vector, and can thus be neglected in this discussion.
Let $\gg_0 = \gg \oplus \mathbb C K$ and $\gg_+ = \gg \otimes_{\mathbb C}\mathbb C[t] \oplus \mathbb C K$.
A $\gg$-module $M$ lifts to a $\widehat \gg$-module $\widehat M_k$ at level $k$ in the usual way. First extend $M$ to a $\gg$-module by letting $K$ act by multiplication with the scalar $k$ and then extend to a $\gg_+$-module by letting $ \gg \otimes_{\mathbb C} t\mathbb C[t]$ act trivially. Then 
\[
\widehat M_k := \text{Ind}^{\widehat{\gg}}_{\gg_+} M. 
\]
$\widehat M_k$ is called a generalized Verma module if a positive half of $\gg_0$ with respect to the weight grading acts locally nilpotently.
Let $\rho_\lambda$ be the irreducible highest-weight representation of $\gg$ at highest-weight $\lambda$. Then we write 
\[
V^k(\lambda) :=  \text{Ind}^{\widehat{\gg}}_{\gg_+} \rho_\lambda,
\]
and $L_k(\lambda)$ for its unique irreducible quotient. Let $|q| <1$ and $h \in \gh$ the Cartan subalgebra. The character of $V^k(\lambda)$ is
\begin{equation}\label{eq:chVk}
\begin{split}
& \text{ch}[V^k(\lambda)](h, q) = \text{tr}_{V^k(\lambda)}\left(q^{L_0 -\frac{c}{24}} e^{2\pi i h} \right)  
= q^{\frac{(\lambda+\rho)^2}{2(k+h^\vee)}} \frac{N_\lambda(h)}{\Pi(h, q)}, \\ &
N_\lambda = \sum_{\omega \in W} \epsilon(\omega) e^{\omega(\lambda+\rho)}, \\ &
\Pi(h, q) = q^{\frac{\text{dim}\ \gg}{24}} e^{\rho(h)} \prod_{r=1}^\infty (1-q^r)^{\text{rank}\ \gg} \prod_{\alpha \in \Delta_+} (1- e^{-\alpha(h)} q^{r-1}) (1- e^{\alpha(h)} q^{r}).
\end{split}
\end{equation}
The homology of the complex $\widehat M_k \otimes C(\gg, f, k)$ is a $\mathcal W^k(\gg, f)$-module that we will denote by $H_f(\widehat M_k)$.  
\begin{thm}${}$
\begin{enumerate}
\item $H_f^i(\widehat M_k)=0$  for $i \neq 0$ and $\widehat M_k$ a generalized Verma module  (\cite[Thm. 6.2]{KWIII})
\item  Let $\gg$ be a Lie algebra. Then $H_f^i(M)=0$  for $i \neq 0$ and any object $M$ in $KL_k$ (\cite[Thm. 4.5.7]{ArI}).
\end{enumerate} 
\end{thm}

Then in either of the two cases the character of the homology coincides with the Euler-Poincar\'e-character
\[
\text{ch}[H_f(\widehat M_k)] = \text{sch}[\widehat M_k \otimes C(\gg, f, k)].
\]
This character can be written down explicitly \cite[Section 6]{KWIII} and we reformulate their formula.  For this let $\mathfrak h^\sharp$ be a Cartan subalgebra of $\ga$ where $\ga$ is the subalgebra of $\gg$ that commutes with the $\gs\gl_2$ corresponding to the triple $\{f, x, e\}$. Choose a Cartan subalgebra  $\gh$ of $\gg_0$ and hence of $\gg$ that contains $\gh^\sharp$ and $x$. Let $\Delta_j = \Delta_j^{\text{even}} \cup  \Delta_j^{\text{odd}}$ be the roots in $\gg_j$ and fix a set of positive roots $\Delta_+ = \Delta_+^0 \cup \bigcup_{j > 0} \Delta_j$ with $\Delta_+^0$ a set of positive roots for $\gg_0$. Set $\delta(\alpha) =j$ for $\alpha\in \Delta_j$. 
Let $\Delta^{h.w}$  denote those roots that correspond to a highest-weight vector for the $\gs\gl_2$-triple $\{f, x, e\}$.
Then 
\begin{equation}\label{eq:chW}
\begin{split}
\text{ch}[\cW^k(\gg, f)](h, q) &= \text{sch}[V^k(\gg) \otimes C(\gg, f, k)] \\
&= q^{-\frac{c}{24}}  \frac{\Psi_{\text{odd}}}{\Psi_{\text{even}}}, \\
\Psi_{\text{odd}} &= \prod_{n=1}^\infty \prod_{\alpha \in \Delta^{\text{odd, h.w.}}} (1+e^{\alpha(h)}q^{n+\delta(\alpha)}) (1+e^{-\alpha(h)}q^{n+\delta(\alpha)}),\\
\Psi_{\text{even}} &= \prod_{n=1}^\infty \prod_{\alpha \in \Delta^{\text{even, h.w.}}} (1-e^{\alpha(h)}q^{n+\delta(\alpha)}) (1-e^{-\alpha(h)}q^{n+\delta(\alpha)}),
\end{split}
\end{equation}
where the domain is $|q| <1$ and $|e^{\pm \alpha(h)}q^{\delta(\alpha)}| <1$. 
Let $M_\lambda$ be a generalized Verma module of $\gg$ that is generated by a highest-weight vector of highest-weight $\lambda$. 
Then  the conformal weight of the top level of $\widehat M_k$ is $\frac{(\lambda| \lambda + 2\rho)}{2(k+h^\vee)}$. Since the character of $\widehat M_k$ is the character of $V^k(\gg)$ times $q^{(\lambda| \lambda + 2\rho)}\chi_{M_\lambda}(h)$ it follows immediately that
\begin{equation}
\begin{split}
\text{ch}[H_f(\widehat M_k)](h, q) &= \text{sch}[\widehat M_k\otimes C(\gg, f, k)] = q^{-\frac{c}{24}} q^{(\lambda| \lambda + 2\rho)}\chi_{M_\lambda}(h-x\tau) \frac{\Psi_{\text{odd}}}{\Psi_{\text{even}}}.
\end{split}
\end{equation}
Assume now that $\gg$ is a Lie algebra and 
we take $M_\lambda =\rho_\lambda$ the integrable irreducible highest-weight representation of the highest-weight $\lambda \in P^+$. Then we can use the Weyl character formula to get
\begin{equation}\label{eq:EP}
\begin{split}
\text{ch}[H_f(V^k(\lambda)](h, q) & = q^{-\frac{c}{24}} q^{\frac{(\lambda| \lambda + 2\rho)}{2(k+h^\vee)}}\chi_{\lambda}(h-x\tau) \frac{1}{\Psi_{\text{even}}}  \\
&= q^{-\frac{c}{24}} q^{\frac{(\lambda| \lambda + 2\rho)}{2(k+h^\vee)}} \frac{\sum\limits_{\omega \in W} \epsilon(\omega) e^{\omega(\lambda +\rho)(h-\tau x)}}{e^{\rho(h-x\tau)}\prod\limits_{\alpha \in \Delta_+} (1 - e^{-\alpha)(h-x\tau)})}  \frac{1}{\Psi_{\text{even}}} \\
&= q^{-\frac{c}{24}} q^{\frac{(\lambda| \lambda + 2\rho)}{2(k+h^\vee)}} {\sum_{\omega \in W} \epsilon(\omega) e^{\omega(\lambda +\rho)(h-\tau x)}}\frac{1}{\Psi} , \\
\Psi &= \prod_{n=1}^\infty \prod_{\alpha \in \Delta^{h.w.}} (1-e^{\alpha(h)}q^{n}) (1-e^{-\alpha(h)}q^{n-1}),
\end{split}
\end{equation}
where the domain is still $|q| <1$ and $|e^{\pm \alpha(h)}q^{\delta(\alpha)}| <1$. 

\subsection{Principal $\cW$-algebras} 
The best studied example is the case where $f$ is the principal nilpotent element, and we use the notation $\cW^k(\gg)$. These $\cW$-algebras have appeared in many problems in mathematics and physics including the AGT correspondence \cite{AGT,BFN,MO,SV} and the quantum geometric Langlands program \cite{AF,CG,Fr,FG,GI,GII}. They are closely related to the classical $\cW$-algebras which arose in the work of Adler, Gelfand, Dickey, Drinfeld, and Sokolov \cite{Ad,GD,Di,DS} in the context of integrable hierarchies of soliton equations. It was conjectured by Frenkel, Kac and Wakimoto \cite{FKW} and proven by Arakawa \cite{ArI,ArII} that for a nondegenerate admissible level $k$, $\cW_k(\gg)$ is lisse and rational. These are called minimal models since they are a generalization of the Virasoro minimal models of \cite{GKO}. For later use, we shall compute the weight where the first singular vector appears.

\begin{lemma}
Let $\gg$ be a simple Lie algebra, $\bar\rho, \bar \rho^\vee$ its Weyl vector and Weyl covector and set
$\bar\alpha =  -\theta$  if $(v, r^\vee) =1$  and $\bar \alpha = -\theta_s$ if $(v, r^\vee) =r^\vee$.
Here $r^\vee$ is the lacity of $\gg$ and $\theta, \theta_s$ are the highest root and highest short root. Set $\bar \lambda  =  n \frac{u}{v} \bar \alpha^\vee  - (\bar \rho, \bar \alpha^\vee)\bar\alpha$

Denote by Sing$(V)$ the weight of the singular vector of $V$ of lowest conformal weight. Then for $k = -h^\vee +\frac{u}{v}$ of (co)principal admissible level, singular vectors of affine and principal $\cW$-algebras have weight
\begin{enumerate}
\item Sing$(V^k(\gg)) =  \frac{v}{2u}\bar\lambda(\bar\lambda+2\bar\rho)$
\item Sing$(\cW^k(\gg)) =  \frac{v}{2u}\bar\lambda(\bar\lambda+2\bar\rho)- \bar\lambda\bar\rho^\vee$ for $k$ a non-degenerate admissible level.
\end{enumerate}
\end{lemma}
\begin{proof}  It is known, see \cite[Cor. 1]{KWVII} or \cite[Prop. 6.14]{ArVI}, that the simple affine vertex algebra of $\gg$ at admissible level $k = -h^\vee + \frac{u}{v}$ is the quotient of the universal affine vertex algebra $V^k(\gg)$ by the ideal $I_k(\gg)$ that is generated by the highest-weight vector $v_\lambda$ of highest-weight $\lambda$ where $\lambda$ is the shifted affine Weyl reflection on $k\omega_0$, that is $\lambda = s_\alpha(k\omega_0 +\rho) -\rho$ with $\omega_0$ the zeroth fundamental weight, $\rho$ the Weyl vector and 
\[
\alpha = \begin{cases} -\theta +v\delta & \quad (v, r^\vee) =1 \\ -\theta_s +\frac{v}{r^\vee}\delta&\quad (v, r^\vee) =r^\vee \end{cases}.
\]
Here $r^\vee$ is the lacity of $\gg$ and $\theta, \theta_s$ are the highest root and highest short root. 
Note that for $\alpha = \bar\alpha - n\delta$, the affine Weyl reflection $s_\alpha$ is $s_\alpha = t_{n\bar\alpha^\vee} \circ s_{\bar\alpha}$ with $\bar \alpha$ a root in $\gg$. Let $\bar \lambda$ be the restriction of $\lambda$ to the Cartan subalgebra of $\gg$. Then 
\[
\bar \lambda  =  n \frac{u}{v} \bar \alpha^\vee  - (\bar \rho, \bar \alpha^\vee)\bar\alpha \qquad \text{with} \qquad   \lambda =s_\alpha(k\omega_0 +\rho) -\rho. 
\]
Using the formula of the conformal weight $h_\lambda$ of the top level of a highest weight module $V^k(\lambda)$ of highest weight $\lambda$,
\[
h_\lambda = \frac{v}{2u}\lambda(\lambda+2\rho)
\]
one immediately gets the singular weight of the affine vertex algebra. 
For the $\cW$-algebra we use the following two facts. First, the quantum Hamiltonian reduction functor is exact on $KL_k(\gg)$ \cite{ArI} (see also \cite[Thm. 7.1]{ArII}). Second, for $k$ a non-degenerate admissible level and $f$ a principal nilpotent, one has that $H^0_f(L_k(\gg)) = \cW_k(\gg)$, i.e. the reduction of the simple affine vertex algebra is the simple $\cW$-algebra. These two results combined tell us that $\cW_k(\gg)$ is the quotient of $\cW^k(\gg)$
by $H^0_f(I_k(\gg))$, but the latter is generated by a highest-weight vector of top level $\Delta_\lambda$, with 
\[
\Delta_\lambda = h_\lambda -\rho^\vee \lambda.
\]
This gives the conformal weight of the singular vectors of the $\cW$-algebras. 
\end{proof}

We compute these for $\gg=\gs\gl_n$. For this consider the lattice $\mathbb Z^n$ with orthonormal basis $\epsilon_1, \dots, \epsilon_n$. We embed root and coroots in rescalings of this lattice in the usual way, i.e. the simple roots are $\alpha_1 = \epsilon-\epsilon_2, \alpha_2 = \epsilon_2-\epsilon_3, \dots$
 Then $\theta =\theta^\vee = \epsilon_1-\epsilon_n$ and $\rho = \rho^\vee = \frac{1}{2}( (n-1)\epsilon_1 + (n-3)\epsilon_2 + \dots + (1-n)\epsilon_n)$. Thus $\rho \theta^\vee = \theta\rho^\vee = n-1$ and $\bar\lambda = (u-n+1)\theta$. Hence

\begin{cor} \label{wprinsingular} Let $\gg=\gs\gl_n$ and let $k=-n+\frac{u}{v}$ be an admissible level, that is $u, v \in \mathbb Z_{> 0}$ are coprime and $u\geq n$. Then 
\begin{enumerate}
\item $V^k(\gs\gl_n)$ has a singular vector at conformal weight $v(u+n-1)$ and no singular vector at lower conformal weight. 
\item For $k$ nondegenerate admissible, that is also $v\geq n$, $\cW^k(\gs\gl_n)$ has a singular vector at conformal weight $(u-n+1)(v-n+1)$ and there is no singular vector at lower conformal weight.
\end{enumerate}
\end{cor}
We continue to study the case of principal $\cW$-algebras, that is $f= f_{\text{prin}}$. For later use we need to know some structure of $\cW$-modules associated to the standard representation and its conjugate. 
For this
let $\omega_1$ be the first fundamental weight of $\gs\gl_n$, so that $\rho_{\omega_1}$ is the standard representation of $\gs\gl_n$. The character is
\[
\chi_{\omega_1}(h) =  \sum_{i=0}^{n-1} e^{\beta_i(h)},\ \text{with} \  \beta_0 := \omega_1\ \text{and} \  \beta_i := \beta_{i-1} - \alpha_i,\ \text{for} \ i =1, \dots, n-1.
\]
It follows that 
\[
\chi_{\omega_1}(-h_\rho\tau) = q^{-(\omega_1|\rho)} \sum_{i=1}^{n-1} q^{i} =  q^{-(\omega_1|\rho)} \frac{1-q^n}{1-q}
\]
and
from \eqref{eq:EP} and \eqref{eq:chW} that
\begin{equation}\label{eq:EPstandard}
\begin{split}
\text{ch}[H_{f_{\text{prin}}}(V^k(\omega_1))](q) & =  q^{\frac{(\omega_1| \omega_1 + 2\rho)}{2(k+h^\vee)}}\chi_{\lambda}(-h_\rho\tau)\text{ch}[\cW^k(\gg)](q) \\
&=  q^{\frac{(\omega_1| \omega_1 + 2\rho)}{2(k+h^\vee)}- (\omega_1|\rho)}\frac{1-q^n}{1-q} \text{ch}[\cW^k(\gg)](q) \\ 
&= q^{-\frac{c}{24}}q^{\frac{(\omega_1| \omega_1 + 2\rho)}{2(k+h^\vee)}- (\omega_1|\rho)}\left( 1 + q +2q^2 + \cdots ). \right.
\end{split}
\end{equation}
Let $h_{\lambda}:= \frac{(\lambda| \lambda + 2\rho)}{2(k+h^\vee)}- (\lambda|\rho)$, let 
\[
H_{f_{\text{prin}}}(V^k(\lambda)) = \bigoplus_{n \in h_\lambda+ \mathbb Z_{\geq 0}} H_{f_{\text{prin}}}(V^k(\lambda))_{n}
\]
be the decomposition into conformal weight spaces, and let 
$L(z)$ be the Virasoro field. Then 
\begin{cor}\label{cor:nodescendants}
Let $n >2$, $\lambda \in  \{ \omega_1, \omega_{n-1}\}$ and $v$ be a nonzero element of $H_{f_{\text{prin}}}(V^k(\lambda))_{h_\lambda}$. If 
$k+n \notin \{  \frac{ n-1}{n}, \ \frac{n+1}{n},\ \frac{n+2}{n} \}$, $\{L_{-1}v\}$ and $\{L_{-2}v, L_{-1}L_{-1}v\}$ are bases of  $H_{f_{\text{prin}}}(V^k(\lambda))_{h_\lambda+1}$ and $H_{f_{\text{prin}}}(V^k(\lambda))_{h_\lambda+2}$, respectively. \end{cor}
\begin{proof} Let $\psi = k+n$ and $\lambda =\omega_1$.
The central charge and conformal weight of top level are
\begin{equation}\nonumber
\begin{split}
c &= (n-1) -n(n^2-1) (\psi +\psi^{-1}-2), \qquad 
h_{\omega_1}= h_{\omega_{n-1}}=  \frac{n^2-1}{2n \psi} -\frac{n-1}{2}.
\end{split}
\end{equation}
Recall some Virasoro commutation relations
\[
[L_1, L_{-1}] = 2L_0, \ \  [L_2, L_{-2}] = 4L_{0} + \frac{c}{2},  \ \  [L_2, L_{-1}] = 3L_1, \ \  [L_1, L_{-2}] = 3L_{-1}.
\]
It follows that $L_{-1}v=0$ implies $h_\lambda=0$ and hence $\psi = \frac{n+1}{n}$. Similarly, a linear relation between $L_{-2}v$ and $L_{-1}L_{-1}v$ is only possible if 
\[
c = - \frac{2 h_{\omega_1} (8 h_{\omega_1} - 5)}{ 2h_{\omega_1} +1}.
\] 
Such a linear relation can only occur at $\psi = \frac{ n-1}{n}, \ \frac{n+1}{n},\ \frac{n+2}{n}$. The claim thus follows for all values of $k$ except possibly for these three. The argument for $\lambda =\omega_{n-1}$ is exactly the same, since $\text{ch}[H_{f_{\text{prin}}}(V^k(\omega_{n-1}))](q)=\text{ch}[H_{f_{\text{prin}}}(V^k(\omega_1))](q)$.
\end{proof}
\begin{remark}
The three exceptional cases have the special property that their simple quotients are either trivial, a rational Virasoro algebra or the $p=2$ singlet algebra of \cite{Kau}. 
\begin{enumerate}
\item For $k = - n + \frac{n+1}{n}$  one has $\cW_k(\gs\gl_n) \cong \mathbb C$  \cite[Thm. 10.1]{LVI},
\item For $k = - n + \frac{n+2}{n}$  one has $\cW_k(\gs\gl_n) \cong \cW_{-2 + \frac{n+2}{2}}(\gs\gl_2)$ \cite[Thm. 10.2]{LVI},  
\item For $k = - n + \frac{n-1}{n}$  one has $\cW_k(\gs\gl_n) \cong \mathbb \cW_{-3+ \frac{2}{3}}(\gs\gl_3)$  \cite[Thm. 10.1]{LVI}.  
\end{enumerate}
\end{remark}

\subsection{Hook-type $\cW$-algebras of $\gs\gl_{n+m}$}
Let $n\geq 2$ and $m\geq 0$ be integers. Let $f_{n,m} \in \gs\gl_{n+m}$ denote the nilpotent element corresponding to the hook-type partition 
$n+m= n + 1 + \dots +1$. The Young tableau for this partition has the form of a hook, hence the name.  Let $\psi = k +n+m$. The corresponding affine $\cW$-algebra 
$$\cW^{\psi}(n,m):=\cW^k(\gs\gl_{n+m}, f_{n,m})$$ will be called a {\it hook-type $\cW$-algebra}. In the cases $m=0$ and $m=1$, $f_{n,0} \in \gs\gl_n$ and $f_{n,1} \in \gs\gl_{n+1}$ are the principal and subregular nilpotents, respectively, so $\cW^{\psi}(n,0) \cong \cW^k(\gs\gl_n)$ and $\cW^{\psi}(n,1) \cong \cW^k(\gs\gl_{n+1}, f_{\text{subreg}})$. For $m\geq 1$, $\cW^{\psi}(n,m)$ has affine vertex subalgebra
$$  \left\{
\begin{array}{ll}
V^{\psi-m-1}(\gg\gl_m) = \cH \otimes V^{\psi-m-1}(\gs\gl_m) & m\geq 2,
\smallskip
\\ \cH & m=1. \\
\end{array} 
\right.
$$
Here $\cH$ denotes the rank one Heisenberg vertex algebra and the level is obtained from Theorem \ref{thm:structureI} (1). 

There are additional even strong generators $T, X^3, \dots, X^n$ in weights $2,3,\dots, n$ which are invariant with respect to $\gg\gl_m$, together with $2m$ even fields $\{P^{\pm, i}|\ i = 1,\dots, m\}$ in conformal weight $\frac{n+1}{2}$, such that $\{P^{+,i}\}$ transforms as the $\gg\gl_m$-standard module $\mathbb{C}^m$, and $\{P^{-,i}\}$ transforms as the dual module $(\mathbb{C}^m)^*$. The Virasoro element $T$ has central charge computed from \eqref{eq:c}
\begin{equation} \label{cc:W(n,m)}
\begin{split}
c =  & \frac{(\psi - n -m)((n+m)^2-1)}{\psi } -n(n^2-1)(\psi - n -m)
\\ & -(n - 1) (n^3 + m n^2 - n^2 - 2 m n - m - n).
\end{split}\end{equation} 
For all $m\geq 1$ we denote the generator of $\cH$ by $J$. For $m\geq 2$, we shall work in the usual basis for $\gs\gl_m$ consisting of $$\{e_{i,j}|\  i \neq j,\ i,j = 1,\dots, m\},$$ together with Cartan elements $$\{h_k =  e_{1,1} - e_{k+1, k+1}|\ k= 1,\dots, m-1\}.$$ We use the same notation for the fields in $V^{\psi -m-1}(\gs\gl_m)$ when no confusion can arise. Then $\{J, e_{i,j}, h_k\}$ are primary of weight $1$ with respect to $T$, and $\{P^{\pm,i}\}$ are primary of weight $ \frac{n+1}{2}$. 

\begin{lemma} \label{lem:wnmj} For all $m\geq 1$, there is a unique choice of normalization of $J$ such that 
$$J(z) P^{\pm, i} (w) \sim  \pm P^{\pm,i}(w)(z-w)^{-1}.$$ With this normalization, $J$ satisfies
\begin{equation} \label{wnm:hnorm} J(z) J(w) \sim -\frac{m (m + n - n \psi)}{m + n}  (z-w)^{-2}. \end{equation} 
\end{lemma}

\begin{proof}
Up to normalization $J$ corresponds to the element $$j:= n(e_1 + \dots + e_m) - m(e_{m+1} + \dots + e_{m+n}).$$ Then
$P^{+,i}$ corresponds to $e_{n, n+i}$ and $P^{-,i}$ to $e_{n+i, 1}$ so that $j$ acts by $\pm(n+m)$ on $P^{\pm,i}$ and hence $J$ corresponds to the element $(n+m)^{-1}j$. The norm is computed from Theorem \ref{thm:structureI} (1). 
\end{proof}

Next, we give meaning to $\cW^{\psi}(n,m)$ in the cases $n = 1$ and $n=0$ as follows. The case $n=1$ should correspond to the $\cW$-algebra of $\gs\gl_{m+1}$ with trivial nilpotent element, hence $$\cW^{\psi}(1,m) \cong V^{\psi-m-1}(\gs\gl_{m+1}).$$ 
For $n=0$, $\cW^{\psi}(0,m)$ should contain a Heisenberg field for all $m\geq 1$, a copy of $V^{\psi -m -1 }(\gs\gl_m)$ for $m\geq 2$, and $2m$ additional even strong generators in weight $\frac{1}{2}$ which transform as $\mathbb{C}^m \oplus (\mathbb{C}^m)^*$ under $\gg\gl_m$. We define
$$\cW^{\psi}(0,m) =  \left\{
\begin{array}{ll}
V^{\psi  -m} (\gs\gl_m) \otimes \cS(m) & m\geq 2
\smallskip
\\  \cS(1) & m=1. \\
\end{array} 
\right.
$$
Here $\cS(m)$ denotes the rank $m$ $\beta\gamma$-system defined by \eqref{eq:betagammaope}. For all $m\geq 1$, $\cS(m)$ admits a homomorphism $$\cH \rightarrow \cS(m),\qquad J \rightarrow -\sum_{i=1}^m :\beta^i \gamma^i:,$$ and for $m\geq 2$ this extends to a map $$\cH\otimes L_{-1}(\gs\gl_m) \rightarrow \cS(m) $$ such that $\{\beta^i\}$ and $\{\gamma^i\}$ transform as $\mathbb{C}^m$ and $(\mathbb{C}^m)^*$ under $\gg\gl_m$. We use the same notation $\{J, e_{i,j}, h_k\}$ to denote the images of the generators of $\cH\otimes L_{-1}(\gs\gl_m)$ in $\cS(m)$. We therefore have a homomorphism 
\begin{equation} \begin{split}  V^{\psi -m -1 }(\gg\gl_m) & \rightarrow \cW^{\psi}(0,m) \cong V^{ \psi  -m}(\gs\gl_m) \otimes \cS(m), 
\\  e_{i,j} & \mapsto e_{i,j}  \otimes 1 + 1 \otimes e_{i,j} ,\quad k_k \mapsto h_k  \otimes 1 + 1 \otimes h_k,\quad J \mapsto 1 \otimes J .\end{split} \end{equation}
Finally, in the cases $n=0,1$ and $m=0$, we define $\cW^{\psi}(n,m) = \mathbb{C}$. Note that for all $n,m$, $\cW^{\psi}(n,m)$ has a uniform description in terms of strong generators: in weight $1$ we have generators of $V^{\psi-m-1}(\gg\gl_m)$, for $n\geq 2$ we have additional even fields $X^2,\dots, X^n$ that are $\gg\gl_m$-trivial, and for all $n\geq 0$ and $m\geq 1$ we have $2m$ additional fields of weight $\frac{n+1}{2}$ transforming under $\gg\gl_m$ as $\mathbb{C}^m \oplus (\mathbb{C}^m)^*$. To summarize, we define
$$ \cW^{\psi}(n,m) = \left\{
\begin{array}{ll}
\cW^{\psi - n -m}(\gs\gl_{n+m}, f_{n+m}) & n\geq 2, \quad m\geq 1,
\smallskip
\\ \cW^{\psi - n}(\gs\gl_n) & n\geq 2 , \quad m = 0 ,
\smallskip 
\\  V^{\psi-m-1}(\gs\gl_{m+1}) & n = 1, \quad m \geq 1 ,
\smallskip 
\\ V^{ \psi  -m}(\gs\gl_m) \otimes \cS(m) & n = 0, \quad m \geq 2 ,
\smallskip 
\\ \cS(1) & n = 0,\quad m =1,
\smallskip 
\\ \mathbb{C} & n = 1,\quad m =0,
\smallskip 
\\ \mathbb{C} & n = 0,\quad m =0. \\
\end{array} 
\right.
$$
We now define $ \cC^{\psi}(n,m)$ to be the affine coset of $\cW^{\psi}(n,m)$. More precisely,
$$ \cC^{\psi}(n,m) = \left\{
\begin{array}{ll}
\text{Com}(V^{\psi-m-1}(\gg\gl_m), \cW^{\psi}(n,m) & n\geq 2, \quad m\geq 1,
\smallskip
\\ \cW^{\psi - n}(\gs\gl_n) & n\geq 2 , \quad m = 0 ,
\smallskip 
\\  \text{Com}(V^{\psi-m-1}(\gg\gl_m), V^{\psi-m-1}(\gs\gl_{m+1})) & n = 1, \quad m \geq 1,
\smallskip 
\\ \text{Com}(V^{\psi -m -1 }(\gg\gl_m), V^{\psi  -m}(\gs\gl_m) \otimes \cS(m)) & n = 0, \quad m \geq 2 ,
\smallskip 
\\ \text{Com}(\cH(1), \cS(1)) & n = 0, \quad m =1,
\smallskip 
\\ \mathbb{C} & n = 1,\quad m =0,
\smallskip 
\\ \mathbb{C} & n = 0,\quad m =0. \\
\end{array} 
\right.
$$
Note that for $n\geq 2$ and $m\geq 1$, $\cC^{\psi}(n,m)$ has Virasoro element 
$L = T -  L^{\gs\gl_m} - L^{\cH}$, where $L^{\gs\gl_m}$ is the Sugawara vector for $V^{\psi-m-1}(\gs\gl_m)$ and $L^{\cH}$ is the Virasoro vector for the Heisenberg algebra with generator $J$. Then $L$ has central charge  
$$ c = -\frac{(n \psi  - m - n -1) (n \psi - \psi - m - n +1 ) (n \psi +  \psi  -m - n)}{(\psi -1) \psi}.$$

\begin{lemma} \label{lem:nondegwnm}${}$
\begin{enumerate}
\item $\cC^{\psi}(n,m)$ is simple for generic values of $\psi$.
\item
For $n\geq 3$, we may replace the fields $X^3, \dots, X^n$ in our strong generating set for $\cW^{\psi}(n,m)$ with fields $\omega^3,\dots, \omega^n \in \cC^{\psi}(n,m)$.
\item Let $U \cong \mathbb{C}^m \oplus (\mathbb{C}^m)^*$ be the space spanned by $\{P^{\pm, i}\}$, which has symmetric bilinear form $$\langle, \rangle: U \rightarrow \mathbb{C},\qquad \langle a,b \rangle = a_{(n)} b.$$ This form is nondegenerate and coincides with the standard pairing on $\mathbb{C}^m \oplus (\mathbb{C}^m)^*$. Hence without loss of generality, we may normalize $\{P^{\pm, i}\}$ so that
$$P^{+,i}(z) P^{-,j}(w) \sim \delta_{i,j} (z-w)^{n+1} + \cdots,$$ where the remaining terms depend only on $T, \omega^3,\dots, \omega^n$ and the generators of $V^{\psi-m-1}(\gg\gl_m)$.
\end{enumerate}
\end{lemma}
\begin{proof}
The first statement follows in all cases from Theorem \ref{thm:genericsimplicity}, parts (1) and (2). The second statement follows from Corollary \ref{cor:primary}, and the third one follows from Lemma \ref{lemma:norm}.
 \end{proof}

\subsection{A family of $\cW$-superalgebras associated to $\gs\gl_{n|m}$}
Let $n\geq 2$ and $m\geq 0$ be integers with $m\neq n$. We define a nilpotent element $f_{n|m}$ in the even part of $\gs\gl_{n|m}$ as follows. If $m = 0$, it is the principal nilpotent in $\gs\gl_n$. If $m  \geq 1$, it is principal in $\gs\gl_n$ and trivial in $\gg\gl_m$. In the case $n\geq 2$ and $m=n$, we let $f_{n|n} \in \gp\gs\gl_{n|n}$ be the nilpotent which is principal in the first copy of $\gs\gl_n$ and trivial in the second copy. Let $\psi = k +n-m$, and consider the $\cW$-superalgebra
$$\cV^{\psi}(n,m) =  \left\{
\begin{array}{ll}
\cW^k(\gs\gl_n) & n\geq 2, \quad m = 0,
\smallskip
\\ 
\cW^k(\gs\gl_{n|m}, f_{n|m}) & n\geq 2, \quad m \geq 1,\quad m \neq n,
\smallskip
\\ \cW^k(\gp\gs\gl_{n|n}, f_{n|n}) & n\geq 2,  \quad m = n. \\
\end{array} 
\right.
$$
For $m\geq 1$, $\cV^{\psi}(n,m)$ has affine vertex subalgebra
$$  \left\{
\begin{array}{ll}
V^{-\psi-m+1}(\gg\gl_m) = \cH \otimes V^{-\psi-m+1}(\gs\gl_m) & m\geq 2,
\smallskip
\\ \cH & m=1. \\
\end{array} 
\right.
$$
For $n\geq 2$, there are additional even strong generators $T, X^3, \dots, X^n$ in weights $2,3,\dots, n$ which are invariant with respect to $\gg\gl_m$, together with $2m$ odd fields $\{P^{\pm, i}|\ i = 1,\dots, m\}$ in conformal weight $\frac{n+1}{2}$, such that $\{P^{+,i}\}$ transforms as the $\gg\gl_m$-standard module $\mathbb{C}^m$, and $\{P^{-,i}\}$ transforms as the dual module $(\mathbb{C}^m)^*$. The Virasoro element $T$ has central charge
\begin{equation} \label{cc:V(n,m)} \begin{split} c = &\frac{(\psi -n +m)((n-m)^2-1)}{\psi} -n(n^2-1)(\psi -n +m) 
\\ &-(n - 1) (n^3 - m n^2 - n^2 + 2 m n + m - n) .\end{split} \end{equation}

As usual, let $J$ be a generator of $\cH$ and $\{e_{i,j}, h_k|\  i \neq j,\ i,j = 1,\dots, m,\ k = 1,\dots, m-1\}$ be our basis of $\gg\gl_m$. We use the same notation for the corresponding fields in $V^{-\psi-m+1}(\gg\gl_m)$, which are primary of weight $1$ with respect to $T$. The fields $\{P^{\pm,i}\}$ are primary of weight $\frac{n+1}{2}$. The proof of the next lemma is very similar to the one of Lemma \ref{lem:wnmj}.

\begin{lemma} \label{lem:vnmj}
 For all $m\geq 1$, there is a unique choice of normalization of $J$ such that 
$$J(z) P^{\pm, i} (w) \sim  \pm P^{\pm,i}(w)(z-w)^{-1}.$$ With this normalization, $J$ satisfies
\begin{equation} \label{vnm:hnorm} J(z) J(w) \sim \frac{m (n \psi + m - n)}{m - n} (z-w)^{-2}. \end{equation} 
\end{lemma}

If $n\geq 2$ and $n=m$, $\cV^{\psi}(n,n)$ has affine vertex subalgebra $V^{-\psi-n+1}(\gs\gl_n)$, additional even strong generators $T = X^2, X^3, \dots, X^n$ in weights $2,3,\dots, n$ which are invariant with respect to $\gs\gl_n$, together with $2n$ odd fields $\{P^{\pm, i}|\ i = 1,\dots, n\}$ in conformal weight $\frac{n+1}{2}$, such that $\{P^{+,i}\}$ transforms as the $\gs\gl_n$-standard module $\mathbb{C}^n$, and $\{P^{-,i}\}$ transforms as the dual module $(\mathbb{C}^n)^*$. The Virasoro element $T$ has central charge
\begin{equation} \label{cc:V(n,n)}
c = -1  + n^2 - n^3  + n \psi - n^3 \psi ,\end{equation} which is just the specialization of \eqref{cc:V(n,m)} to the case $n=m$. The generators of $V^{-\psi-n+1}(\gs\gl_n)$ are primary of weight $1$ with respect to $T$, and $\{P^{\pm,i}\}$ are primary of weight $\frac{n+1}{2}$. 
We also remark that $\cV^{\psi}(n,n)$ has an action of $GL_1$ by outer automorphisms. The origin of this action is as follows. If we consider $\cW^k(\gs\gl_{n|n}, f_{n|n})$ rather than $\cW^{\psi}(\gp\gs\gl_{n|n}, f_{n|n})$, there is an additional Heisenberg field $J$ which satisfies $J(z) J(w) \sim 0$, commutes with $V^{-\psi-n+1}(\gs\gl_n)$, and lies in a proper ideal of $\cW^{\psi}(\gs\gl_{n|n}, f_{n|n})$. Without loss of generality, we may normalize $J$ such that $J(z) P^{\pm, i}(w) \sim \pm P^{\pm, i}(w)(z-w)^{-1}$. The action of the zero mode $J_{(0)}$ exponentiates to a nontrivial $GL_1$-action on the fields $\{P^{\pm,i}\}$, and this action survives in the simple quotient $\cW^{\psi}(\gp\gs\gl_{n|n}, f_{n|n})$ of $\cW^{\psi}(\gs\gl_{n|n}, f_{n|n})$, for generic $\psi$.

Next, if $n = 1$ and $m\geq 2$, we take the $\cW$-algebra of $\gs\gl_{1|m}$ with trivial nilpotent element, hence 
$$\cV^{\psi}(1,m) \cong V^{\psi +m-1}(\gs\gl_{1|m}) \cong V^{-\psi -m +1}(\gs\gl_{m|1}).$$ Similarly, if $n=1$ and $m=1$ we have 
$$\cV^{\psi}(1,1) \cong \cA(1),$$  where $\cA(1)$ is the rank one symplectic fermion algebra defined by \eqref{eq:sfope}. As in the case $n=m$ and $n\geq 2$, there is a natural action of $GL_1$ on $\cA(1)$.

For $n=0$, $\cV^{\psi}(0,m)$ should contain a Heisenberg field for all $m\geq 1$, a copy of $V^{-\psi-m+1}(\gs\gl_m)$ for $m\geq 2$, and $2m$ additional odd strong generators in weight $\frac{1}{2}$ which transform as $\mathbb{C}^m \oplus (\mathbb{C}^m)^*$ under $\gg\gl_m$. We define
$$\cV^{\psi}(0,m) =  \left\{
\begin{array}{ll}
V^{-\psi -m}(\gs\gl_m) \otimes \cE(m) & m\geq 2,
\smallskip
\\  \cE(1) & m=1. \\
\end{array} 
\right.
$$
Here $\cE(m)$ denotes the rank $m$ $bc$-system defined by  \eqref{eq:bcope}. For all $m\geq 1$, $\cE(m)$ admits a homomorphism $$\cH \rightarrow \cE(m),\qquad J \rightarrow \sum_{i=1}^m :b^i c^i:,$$ and for $m\geq 2$ this extends to a map $$\cH\otimes L_{1}(\gs\gl_m) \rightarrow \cE(m) $$ such that $\{b^i\}$ and $\{c^i\}$ transform as $\mathbb{C}^m$ and $(\mathbb{C}^m)^*$ under $\gg\gl_m$. We use the same notation $\{J, e_{i,j}, h_k\}$ to denote the images of the generators of $\cH\otimes L_{1}(\gs\gl_m)$ in $\cE(m)$. We therefore have a homomorphism 
\begin{equation} \begin{split}  V^{-\psi -m+1}(\gg\gl_m) & \rightarrow \cV^{\psi}(0,m) \cong V^{-\psi -m}(\gs\gl_m) \otimes \cE(m), 
\\  e_{i,j} & \mapsto e_{i,j}  \otimes 1 + 1 \otimes e_{i,j} ,\quad h_k \mapsto h_k  \otimes 1 + 1 \otimes h_k,\quad J \mapsto 1 \otimes J .\end{split} \end{equation}
Finally, in the cases $n=0,1$ and $m=0$, we define $\cV^{\psi}(n,m) = \mathbb{C}$. Note that the strong generating set for $\cV^{\psi}(n,m)$ has a uniform description. In weight $1$ we have generators of $V^{-\psi-m+1}(\gg\gl_m)$ if $n\neq m$ and $V^{-\psi-n+1}(\gs\gl_n)$ for $n=m$. For $n\geq 2$, we have additional even fields $X^2,\dots, X^n$ that are $\gg\gl_m$-trivial, and for $n\geq 0$ and $m\geq 1$ we have $2m$ additional odd fields in weight $\frac{n+1}{2}$ transforming under $\gg\gl_m$ as $\mathbb{C}^m \oplus (\mathbb{C}^m)^*$. To summarize, we define
$$ \cV^{\psi}(n,m) = \left\{
\begin{array}{ll}
\cW^{\psi - n}(\gs\gl_n) & n\geq 2, \quad m = 0 ,
\smallskip
\\ \cW^{\psi -n +m}(\gs\gl_{n|m}, f_{n|m}) & n\geq 2, \quad m\geq 1, \quad m\neq n ,
\smallskip
\\ \cW^{\psi}(\gp\gs\gl_{n|n}, f_{n|n})  & n\geq 2, \quad m = n ,
\smallskip 
\\ V^{-\psi-m+1}(\gs\gl_{m|1}) & n = 1, \quad m \geq 2 ,
\smallskip 
\\ \cA(1) & n = 1,\quad  m =1,
\smallskip 
\\ V^{-\psi -m}(\gs\gl_m) \otimes \cE(m) & n = 0, \quad m \geq 2 ,
\smallskip 
\\ \cE(1) & n = 0,\quad m =1,
\smallskip 
\\ \mathbb{C} & n = 1,\quad m =0,
\smallskip 
\\ \mathbb{C} & n = 0, \quad m =0 .  \\
\end{array} 
\right.
$$
We now define $\cD^{\psi}(n,m)$ to be the affine coset of $\cV^{\psi}(n,m)$ in all cases except for $ n =m$, in which case it is the $GL_1$-orbifold of the affine coset. Note that for $n = m = 1$, there is no affine subalgebra so $\cD^{\psi}(n,m)$ is just the $GL_1$ orbifold. More precisely,
\begin{equation*} \begin{split} & \cD^{\psi}(n,m) = 
\\ & \left\{
\begin{array}{ll}
\cW^{\psi - n}(\gs\gl_n) & n\geq 2, \quad m = 0 ,
\smallskip
\\ \text{Com}(V^{-\psi-m+1}(\gg\gl_m), \cV^{\psi}(n,m)) & n\geq 2, \quad m\geq 1, \quad m\neq n ,
\smallskip 
\\ \text{Com}(V^{-\psi-n+1}(\gs\gl_n), \cV^{\psi}(n,n))^{GL_1}  & n\geq 2, \quad m = n ,
\smallskip 
\\ \text{Com}(V^{-\psi-m+1}(\gg\gl_m), V^{-\psi-m+1}(\gs\gl_{m|1})) & n = 1, \quad m \geq 2  ,
\smallskip 
\\ \cA(1)^{GL_1} & n = 1, \quad  m =1,
\smallskip 
\\ \text{Com}(V^{-\psi-m+1}(\gg\gl_m),V^{-\psi -m}(\gs\gl_m) \otimes \cE(m)) & n = 0, \quad m \geq 2 ,
\smallskip 
\\ \mathbb{C} & n = 0 ,\quad m =1,
\smallskip 
\\ \mathbb{C} & n = 1,\quad m =0,
\smallskip 
\\ \mathbb{C} & n = 0, \quad m =0 . \\
\end{array} 
\right.
\end{split} \end{equation*}

In the case $n\geq 2$, $m\geq 1$, $n\neq m$, $\cD^{\psi}(n,m)$ has Virasoro element 
$L = T -  L^{\gs\gl_m} - L^{\cH}$, where $L^{\gs\gl_m}$ is the Sugawara vector for $V^{-\psi-m+1}(\gs\gl_m)$ and $L^{\cH}$ is the Virasoro vector for the Heisenberg algebra with generator $J$. Then $L$ has central charge  
$$ c = -\frac{(n \psi + m - n -1) (n \psi  - \psi + m - n +1) (n \psi + \psi +m - n )}{(\psi -1) \psi} .$$
Similarly, if $n\geq 2$ and $m=n$,  $\cD^{\psi}(n,n)$ has Virasoro element 
$L = T -  L^{\gs\gl_n}$, which has central charge 
$$c = -\frac{(1 + n) (\psi n-1) (\psi n - \psi +1)}{\psi-1}.$$

\begin{lemma}  \label{lem:nondegvnm} ${}$
\begin{enumerate}
\item $\cD^{\psi}(n,m)$ is simple for generic values of $\psi$.
\item
For $n\geq 3$, we may replace the fields $X^3, \dots, X^n$ in our strong generating set for $\cV^{\psi}(n,m)$ with fields $\omega^3,\dots, \omega^n \in \cD^{\psi}(n,m)$.
\item Let $U \cong \mathbb{C}^m \oplus (\mathbb{C}^m)^*$ be the space spanned by $\{P^{\pm, i}\}$, which has symmetric bilinear form $$\langle, \rangle: U \rightarrow \mathbb{C},\qquad \langle a,b \rangle = a_{(n)} b.$$ This form is nondegenerate and coincides with the standard pairing on $\mathbb{C}^m \oplus (\mathbb{C}^m)^*$. Hence without loss of generality, we may normalize $\{P^{\pm, i}\}$ so that
$$P^{+,i}(z) P^{-,j}(w) \sim \delta_{i,j} (z-w)^{n+1} + \cdots,$$ where the remaining terms depend only on $L, \omega^3,\dots, \omega^n$ and the generators of $V^{-\psi-m+1}(\gg\gl_m)$.

\end{enumerate}
\end{lemma}
\begin{proof}
For the first statement, $\cV^{\psi}(n,m)$ is simple in all cases by Theorem \ref{thm:genericsimplicity}, part (1). For $n\neq m$, the simplicity of $\cD^{\psi}(n,m)$ follows from Theorem \ref{thm:genericsimplicity}, part (2). In the case $n=m$, the simplicity of the affine coset is preserved by taking the $GL_1$-orbifold \cite{DLM}, so $\cD^{\psi}(n,m)$ is simple in this case as well. The second statement follows from Corollary \ref{cor:primary}, and the third one follows from Lemma \ref{lemma:norm}.
 \end{proof}

\section{Orbifolds and cosets of $\cW$-algebras} \label{section:orbifoldsandcosets}
The main result in this section is the following
\begin{thm} \label{thm:sfgorbandcoset} Let $\cW^k(\gg,f)$ be a $\cW$-(super)algebra associated to a simple Lie (super)algebra $\gg$ and a nilpotent element $f$. \begin{enumerate}

\item If $G$ is a reductive group of automorphisms of $\cW^k(\gg,f)$, then $\cW^k(\gg,f)^G$ is strongly finitely generated for generic values of $k$. 

\smallskip

\item If $\cW^k(\gg,f)$ has affine subalgebra $V^{\ell}(\ga)$, and $V^{\ell'}(\gb)\subseteq V^{\ell}(\ga)$ is a subalgebra corresponding to a reductive Lie subalgebra $\gb \subseteq \ga$, then $\text{Com}(V^{\ell'}(\gb), \cW^k(\gg,f))$ is strongly finitely generated for generic values of $k$.
\end{enumerate}
\end{thm}

This result is constructive modulo a classical invariant theory problem, namely, the first and second fundamental theorems of invariant theory for some reductive group $G$ and some finite-dimensional $G$-module $V$. In general, this is a hard problem, but in special cases this classical problem has been solved (see for example \cite{W}). In Sections \ref{section:cnm} and \ref{section:dnm}, we will use Theorem \ref{thm:sfgorbandcoset} to give an explicit minimal strong generating set for the affine cosets of all hook-type $\cW$-algebras and $\cW$-superalgebras.

The proof of Theorem \ref{thm:sfgorbandcoset} is based on \cite[Thm. 6.10]{CLIII} together with the fact that $\cW^k(\gg,f)$ admits a limit which is a tensor product of free field algebras of standard type. First, we recall the notion of a {\it deformable family} from \cite{CLI,CLIII}. It is a vertex algebra $\cV$ defined over a ring of rational functions of degree at most zero in some formal variable $\kappa$, with poles lying in some prescribed subset $K\subseteq \mathbb{C}$ which is at most countable. Then $\cV^{\infty} := \lim_{\kappa \rightarrow \infty} \cV$ is well defined, and certain features of $\cV^{\infty}$, such as the weights of a strong generating set, graded character, etc., will also hold for the specialization $\cV^k: = \cV/(\kappa - k) \cV$, for generic values of $k \in \mathbb{C} \setminus K$.

It follows from Theorem \ref{thm:nondegeneracyw} that $\cW^k(\gg,f)$ is a deformable family if the usual generators given by Theorem \ref{thm:kacwakimoto} are rescaled by $\frac{1}{\sqrt{k}}$, and that $$\cW^{\infty}(\gg,f) := \lim_{k \rightarrow \infty} \cW^k(\gg,f) \cong \cW^{\rm free}(\gg,f).$$ By Corollary \ref{cor:gfffadecomposition}, this is a free field algebra of the form $\bigotimes_{i=1}^m \cV_i$, where each $\cV_i$ is one of the standard free field algebras $\cS_{\text{ev}}(n,k)$, $\cS_{\text{odd}}(n,k)$, $\cO_{\text{ev}}(n,k)$, or $\cO_{\text{odd}}(n,k)$. In these cases $\text{Aut}(\cV_i)$ is either $\text{Sp}_{2n}$ or $\text{O}_n$. Moreover, if two of these factors are of the same type, say $\cO_{\text{ev}}(n,k)$ and $\cO_{\text{ev}}(m,k)$, they can be combined into a single one of this type since $\cO_{\text{ev}}(n,k) \otimes \cO_{\text{ev}}(m,k) \cong \cO_{\text{ev}}(n+m,k)$, and similarly for the other types. Therefore in this decomposition, we may assume that the types are distinct; for a fixed $k$, there is at most one integer $n\geq 1$ such that $\cO_{\text{ev}}(n,k)$ appears, and similarly for the other types.

\begin{lemma} \label{lemma:orbifoldofw} Let $G$ be a reductive group of automorphisms of $\cW^k(\gg,f)$ as a one-parameter vertex algebra which acts trivially on the ring of rational functions of $k$. Then $\cW(\gg,f)^G$ is a deformable family and 
$$\lim_{k \rightarrow \infty} \big(\cW^k(\gg,f)^G\big) \cong  \big(\lim_{k \rightarrow \infty} \cW^k(\gg,f)\big)^G \cong \bigg(\bigotimes_{i=1}^m \cV_i\bigg)^G.$$
Moreover, $G$ preserves the tensor factors in this decomposition, so that if $G_i$ is the full automorphism group of $\cV_i$, then $G \subseteq G_1 \times \cdots \times G_m$.
\end{lemma}

\begin{proof} The proof of the first statement is similar to proof of \cite[Cor. 5.2]{CLIII}. We first need a good increasing filtration on $\cW^k(\gg,f)$ in the sense of \cite{LiII}. For each strong generating field $X$ in weight $d$, assign degree $d$ to $X$ and all its derivatives. We define a filtration on $\cW^k(\gg,f)$
$$\cW^k(\gg,f)_{(0)} \subseteq \cW^k(\gg,f)_{(1)} \subseteq \cdots, \qquad \cW^k(\gg,f) = \bigcup_{d \geq 0} \cW^k(\gg,f)_{(d)},$$ where $\cW^k(\gg,f)_{(d)}$ is the span of monomials
$$:\partial^{k_1}X^{i_1}\cdots \partial^{k_r} X^{i_r}:,$$ where $X^{i_t}$ has weight $d_{i_t}$ and $\sum_i d_i \leq d$. Setting $\cW^k(\gg,f)_{(-1)} = \{0\}$, it is easy to verify that this is a good increasing filtration, so that the associated graded algebra $\text{gr}(\cW^k(\gg,f)) = \bigoplus_{i \geq 0}\cW^k(\gg,f)_{(i)} / \cW^k(\gg,f)_{(i-1)}$ is commutative and associative. Using this filtration, the analogue of \cite[Lemma 5.1]{CLIII} is proved in the same way, and the first statement follows.

For the second statement, we just need to show that for any reductive group $G$ of automorphisms of $\cW^k(\gg,f)$, $G$ preserves the distinct tensor factors in the free field algebra limit. But this is clear from the fact that the type of each factor is completely determined by conformal weight and parity of its strong generators, and these are preserved by automorphism groups.
\end{proof}

Suppose that $\cW^k(\gg,f)$ has affine subalgebra $V^{\ell}(\ga)$ where the even part of $\ga$ has dimension $d$ and the odd part has dimension $2e$. Note that $$\lim_{k\rightarrow \infty} V^{\ell}(\ga) \cong \cO_{\text{ev}}(d,2) \otimes \cS_{\text{odd}}(e,2).$$ Then in the decomposition $\cW^{\infty}(\gg,f) \cong \bigotimes_{i=1}^m \cV_i$, we may assume that $\cV_1 \cong  \cO_{\text{ev}}(d,2)$ and $\cV_2 \cong \cS_{\text{odd}}(e,2)$ if $e>0$.

\begin{lemma} \label{lemma:cosetofw} Let $\gb \subseteq \ga$ be a reductive Lie subalgebra of dimension $r$, and let $V^{\ell'}(\gb) \subseteq V^{\ell}(\ga) \subseteq \cW^k(\gg,f)$ be the corresponding affine subalgebra. Write $\cV_1 \cong \ \cO_{\text{ev}}(r,2)  \otimes \cO_{\text{ev}}(d-r,2)$, so that 
$$\cW^{\infty}(\gg,f) \cong \cO_{\text{ev}}(r,2)  \otimes \cO_{\text{ev}}(d-r,2) \otimes \big(\bigotimes_{i=2}^m \cV_i\big).$$
Then the action of $\gb$ coming from the zero modes of the generators of $V^{\ell'}(\gb)$ lifts to an action of a connected Lie group $G$ on $\cW^k(\gg,f)$, and $G$ preserves each of the factors $\cO_{\text{ev}}(d-r,2)$ and $\cV_i$ for $i = 2,\dots, m$. Moreover, $\cC^k = \text{Com}(V^{\ell'}(\gb), \cW^k(\gg,f))$ is a deformable family with limit
$$\cC^{\infty} \cong  \bigg(\cO_{\text{ev}}(d-r,2) \otimes \big(\bigotimes_{i=2}^m \cV_i\big)\bigg)^G.$$ 
\end{lemma}

\begin{proof} This is just the specialization \cite[Thm. 6.10]{CLIII} to our setting. \end{proof}

Lemmas \ref{lemma:orbifoldofw} and \ref{lemma:cosetofw} imply that the strong generating types of both $\cW^k(\gg,f)^G$ and $\cC^k$ are determined by the strong generating types of certain orbifolds of free field algebras. The rest of this section is devoted to studying these orbifolds. We begin by considering the minimal strong generating type of $\cV^{\text{Aut}(\cV)}$ in the case when $\cV$ is one of the standard free field algebras $\cS_{\text{ev}}(n,k)$, $\cS_{\text{odd}}(n,k)$, $\cO_{\text{ev}}(n,k)$, or $\cO_{\text{odd}}(n,k)$. In these cases $\text{Aut}(\cV)$ is either $\text{Sp}_{2n}$ or $\text{O}_n$.

\begin{thm} \label{thm:s-even} For all $n\geq 1$ and odd $k\geq 1$, $\cS_{\text{ev}}(n,k)^{\text{Sp}_{2n}}$ has a minimal strong generating set 
$$\omega^{j} = \frac{1}{2} \sum_{i=1}^n \big(:a^i \partial^{j} b^i: \  - \  :(\partial^{j} a^i )b^i:\big),\quad j = 1,3,\dots, (2n+1)(n+1)  + n k - 2. $$ Since $\omega^{j}$ has weight $k+j$, $\cS_{\text{ev}}(n,k)^{\text{Sp}_{2n}}$ is of type $$\cW \big(k+1,k+3,\dots, (2 n + k + 1) (n + 1) - 2\big).$$ Moreover, $\cS_{\text{ev}}(n,k)$ is completely reducible as an $\cS_{\text{ev}}(n,k)^{\text{Sp}_{2n}}$-module, and all irreducible modules in this decomposition are highest-weight and $C_1$-cofinite according to Miyamoto's definition \cite{Mi}.
\end{thm}

\begin{proof} The first statement can be reduced to showing that, in the notation of \cite[Eq. 9.1]{LV}, $R_n(I) \neq 0$, where $I$ is the following list of length $2n+2$:
$$I = (t, t +1,t+2,\dots,t + 2n+1),\qquad t = \frac{k-1}{2}.$$ The explicit formula for $R_n(I)$ is given by \cite[Thm. 4]{LV}, and it is clear that it is nonzero. Next, the Zhu algebra \cite{Z} of $\cS_{\text{ev}}(n,k)^{\text{Sp}_{2n}}$ is abelian; the proof is similar to that of \cite[Thm. 13]{LV}. This implies that the admissible irreducible modules of $\cS_{\text{ev}}(n,k)^{\text{Sp}_{2n}}$ are all highest weight modules, i.e., they are induced from one-dimensional modules for the Zhu algebra. The proof of $C_1$-cofiniteness is the same as the proof of \cite[Lemma 8]{LII}. \end{proof}

\begin{thm} \label{thm:s-odd}  For all $n\geq 1$ and even $k\geq 2$, $\cS_{\text{odd}}(n,k)^{\text{Sp}_{2n}}$ has a minimal strong generating set 
$$\omega^{j} = \frac{1}{2} \sum_{i=1}^n \big( :a^i \partial^{j} b^i: + :(\partial^j a^i) b^i:\big),\qquad j=0,2,\dots, k n -2.$$ Since $\omega^j$ has weight $k+j$, $\cS_{\text{odd}}(n,k)^{\text{Sp}_{2n}}$ is of type $$\cW(k, k+2\dots,k(n+1) -2).$$
Moreover, $\cS_{\text{odd}}(n,k)$ is completely reducible as an $\cS_{\text{odd}}(n,k)^{\text{Sp}_{2n}}$-module, and all irreducible modules in this decomposition are highest-weight and $C_1$-cofinite according to Miyamoto's definition.
\end{thm}

\begin{proof} This can be reduced to showing that, in the notation of \cite[Eq. 11.1]{LV}, $R_n(I) \neq 0$ for the following list $I$ of length $2n+2$: 
$$I = \bigg(\frac{k}{2}-1,\frac{k}{2}-1,\dots, \frac{k}{2}-1\bigg).$$ This follows easily from the recursive formula given by \cite[Eq. 11.5]{LV}.
\end{proof}

\begin{thm} \label{thm:o-odd}  For all $n\geq 1$ and odd $k\geq 1$, $\cO_{\text{odd}}(n,k)^{\text{O}_{n}}$ has a minimal strong generating set 
$$\omega^{j} = \frac{1}{2} \sum_{i=1}^n :\phi^i \partial^{j} \phi^i:,\qquad j=1,3,\dots, n(k+1) -1.$$ Since $\omega^j$ has weight $k+j$, $\cO_{\text{odd}}(n,k)^{\text{O}_{n}}$ is of type $$\cW(k+1, k+3,\dots,(n+1)(k+1) -2).$$ Moreover, $\cO_{\text{odd}}(n,k)$ is completely reducible as an $\cO_{\text{odd}}(n,k)^{\text{O}_{n}}$-module, and all irreducible modules in this decomposition are highest-weight and $C_1$-cofinite according to Miyamoto's definition.
\end{thm}

\begin{proof} The fact that the above elements are a minimal strong generating set can be reduced to showing that, in the notation of \cite[Eq. 11.1]{LV}, $R_n(I,J) \neq 0$ where $I$ and $J$ are the following lists of length $n+1$:
$$I = \bigg(\frac{k-1}{2},\frac{k-1}{2},\dots, \frac{k-1}{2}\bigg),\qquad J =  \bigg(\frac{k+1}{2},\frac{k+1}{2},\dots, \frac{k+1}{2}\bigg).$$ This follows easily from the recursive formula given by \cite[Eq. 11.5]{LV}. The proof of the remaining statements is the same as proof of Theorem \ref{thm:s-odd}.
\end{proof}

Unfortunately, we are unable to give a minimal strong generating set for $\cO_{\text{ev}}(n,k)^{\text{O}_{n}}$ at present, even in the case $k=2$ which coincides with $\cH(n)^{\text{O}_{n}}$. However, based on Weyl's first and second fundamental theorems of invariant theory for the standard module of $\text{O}_{n}$, we make the following conjecture.

\begin{conj} \label{conj:o-even}  For all $n\geq 1$ and even $k\geq 2$, $\cO_{\text{ev}}(n,k)^{\text{O}_{n}}$ has a minimal strong generating set 
$$\omega^{j} = \sum_{i=1}^n :a^i \partial^{j} a^i:,\qquad j=0,2,\dots,  n(n+1) + nk -2.$$ Since $\omega^j$ has weight $k+j$, $\cO_{\text{ev}}(n,k)^{\text{O}_{n}}$ is of type $$\cW(k, k+2,\dots, (n+k)(n+1) -2).$$
\end{conj}

This generalizes the conjecture given in \cite{LIII} that $\cH(n)^{\text{O}_{n}}$ is of type $\cW(2,4,\dots, n^2+3n)$. In \cite{LIV}, it was shown that this holds for $n\leq 6$, and also that $\cH(n)^{\text{O}_{n}}$ is strongly finitely generated for all $n$. Using the same approach, we can prove

\begin{thm} For all $n\geq 1$ and even $k\geq 2$, $\cO_{\text{ev}}(n,k)^{\text{O}_{n}}$ is strongly generated by $$\omega^{j} = \sum_{i=1}^n :a^i \partial^{j} a^i:,\qquad j=0,2,\dots, s.$$
for some even integer $s \geq n(n+1) + nk -2$. Also, Conjecture \ref{conj:o-even} holds for all $k$ when $n=1$.
\end{thm}

\begin{cor} For all $n\geq 1$ and even $k\geq 2$, $\cO_{\text{ev}}(n,k)$ is completely reducible as an $\cO_{\text{ev}}(n,k)^{\text{O}_{n}}$-module, and all irreducible modules in this decomposition are highest-weight and $C_1$-cofinite according to Miyamoto's definition.
\end{cor}

\begin{cor} Let $\cV$ be any of the free field algebras $\cS_{\text{ev}}(n,k)$, $\cS_{\text{odd}}(n,k)$, $\cO_{\text{ev}}(n,k)$, or $\cO_{\text{odd}}(n,k)$. For any reductive group $G\subseteq \text{Aut}(\cV)$, $\cV^G$ is strongly finitely generated.
\end{cor}

\begin{proof} This is the same as the proof of \cite[Thm. 15]{LV}. First, $\cV^G$ is completely reducible as a $\cV^{\text{Aut}(\cV)}$-module. Here $\text{Aut}(\cV)$ is either $\text{Sp}_{2n}$ or $\text{O}_n$. By a classical theorem of Weyl \cite[Thm. 2.5A]{W}, $\cV^G$ has an (infinite) strong generating set that lies in the sum of finitely many irreducible $\cV^{\text{Aut}(\cV)}$-modules. The result then follows from the strong finite generation of $\cV^{\text{Aut}(\cV)}$ and the $C_1$-cofiniteness of these modules.\end{proof}

\subsection{Some special cases} For a general $G$, it is difficult to give an explicit minimal strong generating set for $\cV^G$. However, there are a few special case which we shall need later, in which it can be done. First, we have a natural embedding \begin{equation} \label{embedding:glsp} \text{GL}_n \hookrightarrow \text{Sp}_{2n}\end{equation} such that the standard module $\mathbb{C}^{2n}$ decomposes as $\mathbb{C}^n \oplus (\mathbb{C}^n)^*$ as a $\text{GL}_n$-module.

\begin{thm}  \label{spevengln} For all $n\geq 1$ and odd $k\geq 1$, $\cS_{\text{ev}}(n,k)^{GL_n}$ has a minimal strong generating set $$\omega^{j} = \sum_{i=1}^n :a^i \partial^{j} b^i: ,\qquad j = 0,1,\dots,  n(n+1) + n k -1. $$ Since $\omega^{j}$ has weight $k+j$, $\cS_{\text{ev}}(n,k)^{\text{Sp}_{2n}}$ is of type $$\cW \big(k,k+1,k+2,\dots, (n+k)(n+1) -1 \big).$$ \end{thm}

\begin{proof} The method of \cite{LI} for studying the $\cW_{1+\infty,-n}$-algebra via its realization as the $\text{GL}_n$-invariants in the rank $n$ $\beta\gamma$-system can be applied in this case. By Weyl's first and second fundamental theorems for the standard representation of $\text{GL}_n$, the generators are given as above, and relations are $(n+1) \times (n+1)$-determinants. The relation of lowest weight has weight $(n+k)(n+1)$ and corresponds to $I = (0,1,\dots, n) = J$. The proof that the coefficient of $\omega^{n(n+1) + n k}$ in this relation is nonzero is similar to the proof of \cite[Thm. 4.15]{LI}. This relation allows $\omega^{n(n+1) + n k}$ to be expressed as a normally ordered polynomial in $\{\omega^j|\ 0\leq j \leq n(n+1) + n k -1\}$ and their derivatives. Finally, higher decoupling relations expressing all $\omega^j$ for $j > n(n+1) + n k$ as normally ordered polynomials in $\{\omega^j|\ 0\leq j \leq n(n+1) + n k -1\}$ and their derivatives, can be constructed inductively starting from this relation. \end{proof}

Next, we consider $\cS_{\text{odd}}(n,k)^{\text{GL}_n}$ where $ \text{GL}_n$ embeds in $\text{Sp}_{2n}$ as above.

\begin{thm}  \label{spoddgln} For all $n\geq 1$ and even $k\geq 2$, $\cS_{\text{odd}}(n,k)^{\text{GL}_n}$ has a minimal strong generating set $$\omega^{j} = \sum_{i=1}^n :a^i \partial^{j} b^i: ,\qquad j = 0,1,\dots, n k -1. $$ Since $\omega^{j}$ has weight $k+j$, $\cS_{\text{odd}}(n,k)^{\text{GL}_n}$ is of type $$\cW \big(k,k+1,k+2,\dots, k(n+1) -1 \big).$$  \end{thm}

\begin{proof} The proof is the same as the proof of \cite[Thm. 4.3]{CLII}, which is the special case $k=2$. \end{proof}

Next, recall that $\text{GL}_n\hookrightarrow \text{O}_{2n}$ such that the standard $\text{O}_{2n}$-module $\mathbb{C}^{2n}$ decomposes under $\text{GL}_n$ as $\mathbb{C}^n \oplus (\mathbb{C}^n)^*$. 

We can choose a generating set $\{e^i, f^i|\ i = 1,\dots, n\}$ for $\cO_{\text{odd}}(2n,k)$ such that $\{e^i\}$ and $\{f^i\}$ transform as $\mathbb{C}^n$ and $ (\mathbb{C}^n)^*$ under $\text{GL}_n$, respectively, and $e^i(z) f^j(w) \sim \delta_{i,j}  (z-w)^{-k}$.

\begin{thm} \label{oroddgln} For all $n\geq 1$ and odd $k\geq 1$,  $\cO_{\text{odd}}(2n,k)^{\text{GL}_n}$ has a minimal strong generating set $$\omega^j = \sum_{i=1}^n :e^i \partial^j f^i:,\qquad j =0,1,\dots, n k -1,$$ and hence is of type $\cW(k,k+1,\dots, k(n+1) -1)$. \end{thm}

\begin{proof} The argument is similar to the proof of \cite[Thm. 4.3]{CLII}. The infinite generating set $\{\nu_{j,k} = \sum_{i=1}^n :(\partial^j e^i )(\partial^k f^i):|\ j,k \geq 0\}$ coming from classical invariant theory can be replaced with the set $\{\partial^k \omega^j|\ j,k\geq 0\}$. The relations among these generators are $(n+1) \times (n+1)$ fermionic determinants (that is, determinants without the usual signs) $D_n(I,J)$ for weakly increasing lists of indices $I = (i_0,i_1,\dots, i_n)$ and $J = (j_0, j_1,\dots, j_n)$, with suitable quantum corrections. A recursive formula for the coefficient $R_n(I,J)$ of $\omega^{|I| + |J| + k(n+1)}$ appearing in $D_n(I,J)$ can be given; here $|I| = \sum_{t=0}^n i_t$ $|J| = \sum_{t=0}^n j_t$. The relation of minimal weight occurs at weight $k(n+1)$ and has the form $\omega^{n k} = P(\omega^0, \omega^1,\dots, \omega^{n k -1})$. As usual, higher decoupling relations expressing all $\omega^j$ for $j > n k$ as normally ordered polynomials in $\omega^0, \omega^1,\dots, \omega^{n k -1}$ and their derivatives can be constructed inductively starting from this relation. \end{proof}

Finally, we consider $\cO_{\text{ev}}(2n,k)^{\text{GL}_n}$ where $\text{GL}_n$ embeds in $\text{O}_{2n}$ as above. Fix a generating set $\{e^i, f^i|\ i = 1,\dots, n\}$ for $\cO_{\text{ev}}(2n,k)$ such that $\{e^i\}$ and $\{f^i\}$ transform as $\mathbb{C}^n$ and $ (\mathbb{C}^n)^*$ under $\text{GL}_n$, respectively, and $e^i(z) f^j(w) \sim \delta_{i,j}  (z-w)^{-k}$.

\begin{thm} \label{orevengln} For all $n\geq 1$ and even $k\geq 2$,  $\cO_{\text{ev}}(2n,k)^{\text{GL}_n}$ has a minimal strong generating set $$\omega^j = \sum_{i=1}^n :e^i \partial^j f^i:,\qquad j =0,1,\dots, n(n+1) + n k -1,$$ and hence is of type $\cW(k,k+1,\dots, (n+k)(n+1)-1)$. \end{thm}

\begin{proof} The case $k=2$ is given by \cite[Thm. 8.1]{LVI}, and the general case is similar. \end{proof}

\begin{thm} \label{gffa:generalorbifold} Let $\cV = \bigotimes_{i=1}^m \cV_i$ where each factor $\cV_i$ is one of the above free field algebras $\cS_{\text{ev}}(n,k)$, $\cS_{\text{odd}}(n,k)$, $\cO_{\text{ev}}(n,k)$, or $\cO_{\text{odd}}(n,k)$. Let $G_i = \text{Aut}(\cV_i)$ which is either an orthogonal or a symplectic group, and let $G$ be a reductive group of automorphisms of $\cV$ which preserves the tensor factors, so that $G \subseteq G_1\times \cdots \times G_m$. Then $\cV^G$ is strongly finitely generated.
\end{thm}

\begin{proof} The argument is the same as the proof of \cite[Thm. 4.2]{CLIII}. \end{proof}

\noindent {\it Proof of Theorem \ref{thm:sfgorbandcoset}}. This is now an immediate consequence of Lemmas \ref{lemma:orbifoldofw} and \ref{lemma:cosetofw}, and Theorem \ref{gffa:generalorbifold}. $\qquad \qquad \qquad \qquad \qquad \qquad \qquad \qquad \qquad \qquad   \Box$

 \smallskip

The next result is a generalization of Theorem \ref{thm:genericsimplicity}, and will be useful in future applications.

\begin{thm} \label{thm:completereducibility} Let $\gg$ be a Lie superalgebra with nondegenerate invariant bilinear form $( \ \ | \ \ )$, and let $f$ be a nilpotent element in $\gg$. Let $V^{\ell'}(\gb)$ be the subalgebra of the full affine subalgebra $V^{\ell}(\ga)$ of $\cW^k(\gg,f)$, where $\gb$ is a reductive Lie subalgebra of $\ga$. Then for generic $k$, $\cW^k(\gg,f)$ admits a decomposition
\begin{equation} \label{KLdecomp} \cW^k(\gg,f) \cong \bigoplus_{\lambda \in P^+}V^{\ell'}(\lambda) \otimes \cC^{k}(\lambda),\end{equation} where $P^+$ denotes the set of dominant weights of $\gb$, $V^{\ell'}(\lambda)$ are the corresponding Weyl modules, and the multiplicity spaces $\cC^{k}(\lambda)$ are irreducible modules for the coset $\cC^k = \text{Com}(V^{\ell'}(\gb), \cW^k(\gg,f))$.
\end{thm}

\begin{proof} For generic values of $k$, since the Kazhdan-Lusztig category $KL_{\ell'}(\gb)$ is semisimple, the existence of a decomposition of the form \eqref{KLdecomp} is clear, where the Weyl modules $V^{\ell'}(\lambda)$ are irreducible. What remains to prove is that the multiplicity spaces $\cC^{k}(\lambda)$ are irreducible $\cC^k$-modules. Let $r = \text{dim}\ \gg$, and as in Lemma \ref{lemma:cosetofw}, write $$\cW^{\infty}(\gg,f) = \lim_{k\rightarrow \infty} \cW^k(\gg,f) \cong \cO_{\text{ev}}(r,2) \otimes \cO_{\text{ev}}(d-r,2) \otimes \big( \bigotimes_{i=2}^m \cV_i\big),$$ where $\cO_{\text{ev}}(r,2) = \cH(r)$ is the rank $r$ Heisenberg algebra coming from the limit of $V^{\ell'}(\gb)$. For convenience, we write $\tilde{\cW} = \cO_{\text{ev}}(d-r,2) \otimes  \big( \bigotimes_{i=2}^m \cV_i\big)$, so that $\cC^{\infty} =  \lim_{k \rightarrow \infty} \cC^k  \cong \tilde{\cW}^G$. 

The module $V^{\ell'}(\lambda)$ has limit 
$$V^{\infty}(\lambda) = \lim_{\ell' \rightarrow \infty} V^{\ell'}(\lambda) \cong \cH(r) \otimes L_{\lambda},$$ where $L_{\lambda}$ is the finite-dimensional irreducible $\gb$-module with highest weight $\lambda$. Then we have a decomposition
$$\cW^{\infty}(\gg,f) \cong \cH(r) \otimes \tilde{\cW} \cong \cH(r) \otimes \bigg(\bigoplus_{\lambda \in P^+} L_{\lambda} \otimes D^{\lambda}\bigg),$$ Here $D^{\lambda} = \lim_{\ell' \rightarrow \infty} \cC^{k}(\lambda)$, which is a module over $\tilde{\cW}^G \cong \lim_{k \rightarrow \infty} \cC^k$. Since $\tilde{\cW}$ is simple and $G$ is reductive, by passing to a compact form of $G$, it follows from \cite{DLM} that each of the modules $D^{\lambda}$ is an irreducible $\tilde{\cW}^G$-module.

Finally, if some $\cC^{k}(\lambda)$ were reducible, it would have a nontrivial singular vector $\omega$ in some weight space higher than the lowest weight component. In the $k\rightarrow \infty$ limit, after suitable scaling $\omega$ must survive and would have to be singular for the action of $\tilde{\cW}^G$. This is impossible since $\cD^{\lambda}$ is irreducible.
\end{proof}

A similar statement to Theorem \ref{thm:completereducibility} in fact holds in a much more general setting. Let $\gg$ be a reductive Lie algebra of dimension $d$, and let $\cA^k$ be a vertex algebra whose structure constants depend algebraically on $k$, which admits a homomorphism $V^k(\gg)\rightarrow \cA^k$ with the following properties:
\begin{enumerate}
\item There exists a deformable family $\cA$ defined over ring $F_K$ of rational functions of degree at most zero in $\kappa$, with poles in some at most countable set $K$, such that $$\cA/ (\kappa - \sqrt{k})\cA \cong \cA^k,\  \text{for all}\ \sqrt{k} \notin K.$$

\item The homomorphism $V^k(\gg) \rightarrow \cA^k$ is induced by a map $\cV \rightarrow \cW$, where $\cV$ is the deformable family such that $\cV / (\kappa - \sqrt{k}) \cV \cong V^k(\gg)$ for all $k\neq 0$.
 
\item $\cA^{\infty} = \lim_{\kappa \rightarrow \infty} \cA$ decomposes as 
$$\cA^{\infty} \cong \cV^{\infty} \otimes \tilde{\cA} \cong \cH(d) \otimes \tilde{\cA},$$ where $\tilde{\cA}$ is a tensor product of standard free field algebras $\cO_{\text{ev}}(n,k)$, $\cS_{\text{ev}}(n,k)$, $\cO_{\text{odd}}(n,k)$, and $\cS_{\text{odd}}(n,k)$.

\item The action of $\gg$ on $\cA^k$ integrated to an action of a connected Lie group $G$ and $\cA^k$ decomposes into finite-dimensional $G$-modules.
\end{enumerate}

By the same argument as the proof of Theorem \ref{thm:completereducibility}, we obtain
\begin{thm} Let $V^k(\gg)\rightarrow \cA^k$ be as above, and let $\cC^k = \text{Com}(V^k(\gg), \cA^k)$. Then for generic $k$, $\cA^k$ admits a decomposition
\begin{equation} \cA^k \cong \bigoplus_{\lambda \in P^+}V^{k}(\lambda) \otimes \cC^{k}(\lambda),\end{equation} where $P^+$ denotes the set of dominant weights of $\gg$, $V^{k}(\lambda)$ are the corresponding Weyl modules, and the multiplicity spaces $\cC^{k}(\lambda)$ are irreducible $\cC^k$-modules.
\end{thm}
\begin{remark}\label{rem:int-set}
If $G$ acts faithfully, then the $\cC^k(\lambda)$ are all nonzero. In general, the set of dominant weights for which the $\cC^k(\lambda)$ are nonzero can be determined by a classical invariant theory problem. We illustrate this in the example of interest to us, that is $G = U(m)$ and as a $G$-module our free field algebra is either isomorphic to
\[
S(m) :=\text{Sym}\left( \bigoplus_{n=1}^\infty ( \mathbb{C}^m_n \oplus (\mathbb{C}^m_n)^*) \right)\ \text{or} \  E(m):=\bigwedge\left( \bigoplus_{n=1}^\infty ( \mathbb{C}^m_n \oplus (\mathbb{C}^m_n)^*) \right).
\]
Here $\mathbb{C}^m_n$ denote each a copy of the standard representation of $U(m)=SU(m) \times U(1)$, which is given weight one corresponding to the $U(1)$-action. Similarly, $(\mathbb{C}^m_n)^*$ denotes the conjugate which then has $U(1)$-weight minus one. It is clear that $\wedge^m \mathbb{C}^m$ as well as the determinant of the product of $m$ distinct copies of $\mathbb{C}^m$ is trivial as an $SU(m)$-module and has $U(1)$-weight $m$. 
Replacing $\mathbb{C}^m$ by $(\mathbb{C}^m)^*$ gives modules that are trivial as $SU(m)$-modules and have $U(1)$-weight $-m$. Call these one-dimensional modules $A^\pm$ and we clearly have infinitely many copies of this type of module in both $S(m)$ and $E(m)$. Let
\[
S(m) \cong \bigoplus_{\lambda \in P^+ \times \mathbb Z} \rho_{\lambda, n} \otimes {M_{\lambda, n}}, \qquad
E(m) \cong \bigoplus_{\lambda \in P^+ \times \mathbb Z} \rho_{\lambda, n} \otimes {N_{\lambda, n}}
\]
be the decompositions as $SU(m) \times U(1)$-modules. Here the $M_{\lambda, n}, N_{\lambda, n}$ are multiplicity spaces and the $\rho_{\lambda, n}$ are the irreducibles of highest weight $(\lambda, n)$.
Let $i(n)$ be the integer in $[0, n-1]$ with $i(n) = n \mod m$.
 Clearly the multiplicities can only be nonzero if $\lambda \equiv \omega_{i(n)} \mod A_{m-1}$. Here $A_{m-1}$ denotes the root lattice of $\gs\gl_m$, $\omega_i$ are the fundamental weights, and we set $\omega_0=0$. Since $SU(m)$ acts faithfully, for each $\lambda$ there exists at least one $n, n' \in\mathbb Z$ with $M_{\lambda, n}, N_{\lambda, n'}$ nonzero. But then $M_{\lambda, n\pm mr}, N_{\lambda, n'\pm mr}$ are also nonzero for any $r\in \mathbb Z_{>0}$ by multiplying with $|r|$ modules of type  $A^{\pm}$.
\end{remark}

\section{Universal two-parameter $\cW_{\infty}$-algebra} \label{section:winf} 

In this section, we briefly recall some features of the universal two-parameter vertex algebra $\cW(c,\lambda)$ constructed in \cite{LVI}. The algebra is defined over the ring $\mathbb{C}[c,\lambda]$, and is generated by a Virasoro field $L$ of central charge $c$ and a primary weight $3$ field $W^3$ which is normalized so that $ (W^3)_{(5)} W^3 = \frac{c}{3} 1$. The remaining strong generators $W^i$ of weight $i \geq 4$ are defined inductively by $$W^i = (W^{3})_{(1)} W^{i-1},\qquad i \geq 4.$$ Then $\cW(c,\lambda)$ is defined over the ring $\mathbb{C}[c,\lambda]$, and is freely generated by $\{L,W^i|\ i\geq 3\}$. It has a conformal weight grading $$\cW(c,\lambda) = \bigoplus_{n\geq 0} \cW(c,\lambda)[n],$$ where each $\cW(c,\lambda)[n]$ is a free $\mathbb{C}[c,\lambda]$-module and $\cW(c,\lambda)[0] \cong \mathbb{C}[c,\lambda]$. There is a symmetric bilinear form on $\cW(c,\lambda)[n]$ given by
$$\langle ,  \rangle_n : \cW(c,\lambda)[n] \otimes_{\mathbb{C}[c,\lambda]} \cW(c,\lambda)[n] \rightarrow \mathbb{C}[c,\lambda],\qquad \langle \omega, \nu \rangle_n = \omega_{(2n-1)} \nu.$$ The determinant $\text{det}_n$ of this form is nonzero for all $n$; equivalently, $\cW(c,\lambda)$ is a simple vertex algebra over $\mathbb{C}[c,\lambda]$.

Let $p$ be an irreducible factor of $\text{det}_{N+1}$ and let $I = (p) \subseteq \mathbb{C}[c,\lambda] \cong \cW(c,\lambda)[0]$ be the corresponding ideal. Consider the quotient
$$ \cW^I(c,\lambda) = \cW(c,\lambda) / I \cdot \cW(c,\lambda),$$ where $I$ is regarded as a subset of the weight zero space $\cW(c,\lambda)[0] \cong \mathbb{C}[c,\lambda]$, and $I \cdot \cW(c,\lambda)$ denotes the vertex algebra ideal generated by $I$.  This is a vertex algebra over the ring $\mathbb{C}[c,\lambda]/I$, which is no longer simple. It contains a singular vector $\omega$ in weight $N+1$, which lies in the maximal proper ideal $\cI\subseteq \cW^I(c,\lambda)$ graded by conformal weight. If $p$ does not divide $\text{det}_{m}$ for any $m<N+1$, $\omega$ will have minimal weight among elements of $\cI$. Often, $\omega$ has the form \begin{equation} \label{sing:intro} W^{N+1} - P(L, W^3,\dots, W^{N-1}),\end{equation} possibly after localizing, where $P$ is a normally ordered polynomial in the fields $L,W^3,\dots,$ $W^{N-1}$, and their derivatives. If this is the case, there will exist relations in the simple graded quotient $\cW_I(c,\lambda) :=\cW^I(c,\lambda) / \cI$ of the form
$$W^m = P_m(L, W^3, \dots, W^N),$$ for all $m \geq N+1$ expressing $W^m$ in terms of $L, W^3,\dots, W^N$ and their derivatives. Then $\cW_I(c,\lambda)$ will be of type $\cW(2,3,\dots, N)$. Conversely, any one-parameter vertex algebra $\cW$ of type $\cW(2,3,\dots, N)$ for some $N$ satisfying mild hypotheses, is isomorphic to $\cW_I(c,\lambda)$ for some $I = (p)$ as above, possibly after localizing. The corresponding variety $V(I) \subseteq\mathbb{C}^2$ is called the {\it truncation curve} for $\cW$.

Note that if $I=(p)$ for some irreducible $p$, then $\cW^I(c,\lambda)$ and $\cW_I(c,\lambda)$ are one-parameter vertex algebras since $\mathbb{C}[c,\lambda] /(p)$ has Krull dimension $1$. We also consider $\cW^I(c,\lambda)$ when $I\subseteq \mathbb{C}[c,\lambda]$ is a maximal ideal, which has the form $I = (c- c_0, \lambda- \lambda_0)$ for some $c_0, \lambda_0\in \mathbb{C}$. Then $\cW^I(c,\lambda)$ and its quotients are ordinary vertex algebras over $\mathbb{C}$. Given maximal ideals $I_0 = (c- c_0, \lambda- \lambda_0)$ and $I_1 = (c - c_1, \lambda - \lambda_1)$, let $\cW_0$ and $\cW_1$ be the simple quotients of $\cW^{I_0}(c,\lambda)$ and $\cW^{I_1}(c,\lambda)$, respectively. There is an easy criterion for $\cW_0$ and $\cW_1$ to be isomorphic. We must have $c_0 = c_1$, and if $c_0 \neq 0$ or $-2$, there is no restriction on $\lambda_0$ and $\lambda_1$. For all other values of the central charge, we must have $\lambda_0 = \lambda_1$. This criterion implies that aside from the coincidences at $c=0$ and $-2$, all other coincidences among simple one-parameter quotients of $\cW(c,\lambda)$ must correspond to intersection points on their truncation curves; see \cite[Cor. 10.1]{LVI}.

Often, a vertex algebra $\cC^k$ arising as a coset of the form $\text{Com}(V^k(\gg), \cA^k)$ for some vertex algebra $\cA^k$, can be identified with a one-parameter quotient $\cW_I(c,\lambda)$ for some $I$. Here $k$ is regarded as a formal variable, and we have a homomorphism
\begin{equation} \label{isoviaparametrization} L \mapsto \tilde{L}, \qquad W^3 \mapsto \tilde{W}^3, \qquad c\mapsto c(k), \qquad \lambda\mapsto \lambda(k). \end{equation} Here $\{\tilde{L}, \tilde{W}^3\}$ are the standard generators of $\cC^k$, where $(\tilde{W}^3)_{(5)} \tilde{W}^3 = \frac{c(k)}{3}1$, and $k \mapsto (c(k), \lambda(k))$ is a rational parametrization of the curve $V(I)$. 

There are two subtleties that need to be mentioned. First, for a complex number $k_0$, the specialization $\cC^{k_0} := \cC^k / (k-k_0) \cC^k$ typically makes sense for all $k_0 \in \mathbb{C}$, even if $k_0$ is a pole of $c(k)$ or $\lambda(k)$. At these points, $\cC^{k_0}$ need not be obtained as a quotient of $\cW^I(c,\lambda)$. Second, even if $k_0$ is not a pole of $c(k)$ or $\lambda(k)$, the specialization $\cC^{k_0}$ can be a proper subalgebra of the \lq\lq honest" coset $\text{Com}(V^{k_0}(\gg), \cA^{k_0})$, even though generically these coincide. By \cite[Cor. 6.7]{CLII}, under mild hypotheses that are satisfied in all our examples, if $\gg$ is simple this can only occur for rational numbers $k_0 \leq -h^{\vee}$, where $h^{\vee}$ is the dual Coxeter number of $\gg$. Additionally, if $\gg$ contains an abelian subalgebra $\gh$, the coset becomes larger at the levels where the corresponding Heisenberg fields become degenerate, since it now contains these fields.

For later use, we recall some OPEs in the algebra $\cW(c,\lambda)$ from \cite{LVI}.
\begin{equation}\label{ope:first} \begin{split} W^3(z) W^3(w) & \sim  \frac{c}{3}(z-w)^{-6} +  2 L(w)(z-w)^{-4} + \partial L(w)(z-w)^{-3} \\ &\quad  + W^4(w)(z-w)^{-2} + \bigg(\frac{1}{2}  \partial W^4 -\frac{1}{12} \partial^3 L\bigg)(w)(z-w)^{-1}. \end{split} \end{equation}
\begin{equation} \label{ope:second} \begin{split}  L(z) W^4(w) & \sim 3 c (z-w)^{-6} +  10 L(w)(z-w)^{-4} + 3 \partial L(w)(z-w)^{-3} \\ & \quad + 4 W^4(w)(z-w)^{-2} +  \partial W^4(w)(z-w)^{-1}.\end{split} \end{equation}
\begin{equation} \label{ope:third}  \begin{split} L(z)  W^5(w) & \sim \big(185 - 80 \lambda (2 + c)\big) W^3(z)(z-w)^{-4} \\ &\quad  + \big(55 - 16\lambda (2 + c) \big) \partial W^3(z) (z-w)^{-3}\\ & \quad + 
5 W^5(w)(z-w)^{-2} + \partial W^5(w) (z-w)^{-1}.\end{split} \end{equation}
\begin{equation} \label{ope:fourth} \begin{split} W^3(z) W^4(w) & \sim   \bigg(31-16 \lambda (2+c)\bigg)  W^3(w)(z-w)^{-4}
\\ &\quad + \frac{8}{3} \bigg(5 - 2\lambda (2 + c)  \bigg) \partial  W^3(w)(z-w)^{-3} 
\\ & \quad  + W^5(w)(z-w)^{-2}  
\\ &\quad +\bigg(\frac{2}{5}
\partial W^5 + \frac{32}{5} \lambda  :L\partial W^3: -\frac{48}{5} \lambda :(\partial L)  W^3: \\ \quad &\quad  + \frac{2}{15}  \bigg(-5 + 2 \lambda (-1 + c)\bigg)\partial^3 W^3\bigg) (w)(z-w)^{-1}. \end{split} \end{equation}

\section{The structure of $\cC^{\psi}(n,m)$} \label{section:cnm}

The main result in this section is the explicit realization of the affine coset $\cC^{\psi}(n,m)$ of $\cW^{\psi}(n,m)$ as a simple, one-parameter quotient of $\cW(c,\lambda)$ of the form $\cW_{I_{n,m}}(c,\lambda)$ for an ideal $I_{n,m} \subseteq \mathbb{C}[c,\lambda]$.

\begin{thm} \label{c(n,m)asquotient} For $m \geq 1$ and $n\geq 0$, and for $m=0$ and $n\geq 3$, $\cC^{\psi}(n,m) \cong \cW_{I_{n,m}}(c,\lambda)$, where $I_{n,m}$ is described explicitly via the parametrization
\begin{equation} \label{cnmparametrization} \begin{split} c(\psi) &=  -\frac{(n \psi  - m - n -1) (n \psi - \psi - m - n +1 ) (n \psi +  \psi  -m - n)}{(\psi -1) \psi},
\\ \ 
\\ \lambda(\psi) & = -\frac{(\psi-1) \psi}{(n \psi - n - m -2) (n \psi - 2 \psi - m - n +2 ) (n \psi + 2 \psi  -m - n )}.
\end{split} \end{equation}
Moreover, after a suitable localization, $\cW^{I_{n,m}}(c,\lambda)$ has a singular vector of the form
$$W^{(m+1)(m+n+1)} - P(L, W^3,\dots, W^{(m+1)(m+n+1)-1})$$ and no singular vector of lower weight, where $P$ is a normally ordered polynomial in the fields $L, W^3,\dots, W^{(m+1)(m+n+1)-1}$, and their derivatives. Therefore $\cW_{I_{n,m}}(c,\lambda)$ has minimal strong generating type $\cW(2,3,\dots, (m+1)(m+n+1)-1)$.
\end{thm}

\begin{remark} By changing variables, one verifies easily that this truncation curve is precisely the one for $Y_{0,m,m+n}[\psi]$ given by \cite[Eq. 2.14]{PRI}.
\end{remark} 

In the case $m=0$ and $n\geq 3$, there is nothing to prove since the truncation curve for $\cC^{\psi}(n,0) \cong \cW^{\psi-n}(\gs\gl_n)$ is given by \cite[Thm. 7.4]{LVI}, and agrees with \eqref{cnmparametrization}. It is worth mentioning that in \cite{LVI}, we actually computed the truncation curve for the coset $\text{Com}(V^{k+1}(\gs\gl_n), V^k(\gs\gl_n) \otimes L_1(\gs\gl_n))$, and concluded that this was the truncation curve for the $\cW$-algebra using \cite[Thm. 8.7]{ACLII}. It is not difficult to compute $\lambda(\psi)$ for $\cW^{\psi-n}(\gs\gl_n)$ directly using the Miura realization, and hence get an independent proof of \cite[Thm. 8.7]{ACLII} for type $A$.

For the rest of this section, we assume $m\geq 1$. Before proceeding with the proof of Theorem \ref{c(n,m)asquotient}, we shall outline our strategy. We first show that $\cC^{\psi}(n,m)$ is of type $\cW(2,3,\dots, (m+1)(m+n+1)-1)$ using its free field limit. Next, we show that $\cC^{\psi}(n,m)$ is at worst an extension of a vertex subalgebra $\tilde{\cC}^{\psi}(n,m)$ which is of type $\cW(2,3,\dots, N)$ for some $N \leq (m+1)(m+n+1)-1$, and is a one-parameter quotient of $\cW(c,\lambda)$. Therefore $\cW^{\psi}(n,m)$ is an extension of $\cH \otimes V^{\psi-m-1}(\gs\gl_m) \otimes \tilde{\cC}^{\psi}(n,m)$. 

The key step, which we call the {\it reconstruction argument}, is to prove that if $\cW$ is any one-parameter quotient of $\cW^I(c,\lambda)$ for some ideal $I$ with the property that $\cH \otimes V^{\psi-m-1}(\gs\gl_m) \otimes \cW$ admits an extension containing fields $\{P^{\pm, i}\}$ transforming as $\mathbb{C}^m \oplus (\mathbb{C}^m)^*$, as well as some mild properties possessed by $\cW^{\psi}(n,m)$, then $\cW$ must be a quotient of $\cW^{I_{n,m}}(c,\lambda)$, where $I_{n,m}$ is the ideal given in Theorem \ref{c(n,m)asquotient}. In particular, this proves that $\tilde{C}^{\psi}(n,m)$ must be a quotient of $\cW^{I_{n,m}}(c,\lambda)$.
 
The final step, which we call the {\it exhaustiveness argument}, is to prove that $\tilde{C}^{\psi}(n,m) = \cC^{\psi}(n,m)$. By finding coincidences between the simple quotient $\tilde{\cC}_{\psi}(n,m)$ and certain principal $\cW$-algebras of type $A$, and making use of Corollary \ref{wprinsingular}, we show that $\tilde{\cC}^{\psi}(n,m)$ has type $\cW(2,3,\dots, N)$ for some $N\geq  (m+1)(m+n+1)-1$. It follows that $N =  (m+1)(m+n+1)-1$ and that $\tilde{\cC}^{\psi}(n,m)=\cC^{\psi}(n,m)$. Since $\cC^{\psi}(n,m)$ is generically simple by Lemma \ref{lem:nondegwnm}, we must have $\cC^{\psi}(n,m) \cong \cW_{I_{n,m}}(c,\lambda)$.

\begin{lemma} \label{lemma:stronggencnm} For $m\geq 1$ and $n\geq 0$, $\cC^{\psi}(n,m)$ is of type $\cW(2,3,\dots, (m+1)(m+n+1)-1)$ as a one-parameter vertex algebra. Equivalently, this holds for generic values of $\psi$. \end{lemma}

\begin{proof} Consider the free field limit $\cW^{\text{free}}(n,m) := \cW^{\text{free}}(\gs\gl_{n+m}, f_{n,m})$. Then
$$ \cW^{\text{free}}(n,m) \cong  \left\{
\begin{array}{ll}
\cO_{\text{ev}}(m^2,2) \otimes \big( \bigotimes_{i=2}^n \cO_{\text{ev}}(1,2i)\big) \otimes \cS_{\text{ev}}(m, n+1), & n \ \text{even},
\smallskip
\\ \cO_{\text{ev}}(m^2,2) \otimes \big( \bigotimes_{i=2}^n \cO_{\text{ev}}(1,2i)\big) \otimes \cO_{\text{ev}}(2m, n+1), & n \ \text{odd}. \\
\end{array} 
\right.
$$

In this notation, $\cO_{\text{ev}}(m^2,2) = \cH(m^2)$ is just the rank $m^2$ Heisenberg algebra coming from the affine subalgebra, $\cO_{\text{ev}}(1,2i)$ is the algebra generated by $\omega^i$ for $i = 2,\dots, n$, and the fields $\{P^{\pm, i}\}$ generate $\cS_{\text{ev}}(m, n+1)$ or $\cO_{\text{ev}}(2m, n+1)$ when $n$ is even or odd, respectively. By Lemma \ref{lemma:cosetofw},
\begin{equation} \begin{split} \lim_{k\rightarrow \infty} \cC^k(n,m) & \cong \big( \big( \bigotimes_{i=2}^n \cO_{\text{ev}}(1,2i)\big) \otimes \cS_{\text{ev}}(m, n+1)\big)^{\text{GL}_m} 
\\ & \cong \big( \bigotimes_{i=2}^n \cO_{\text{ev}}(1,2i)\big) \otimes \big(\cS_{\text{ev}}(m, n+1)\big)^{\text{GL}_m} ,\qquad n \ \text{even},\end{split} \end{equation}
\begin{equation} \begin{split} \lim_{k\rightarrow \infty} \cC^k(n,m) & \cong \big( \big( \bigotimes_{i=2}^n \cO_{\text{ev}}(1,2i)\big) \otimes \cO_{\text{ev}}(2m, n+1)\big)^{\text{GL}_m} 
\\ & \cong\big( \bigotimes_{i=2}^n \cO_{\text{ev}}(1,2i)\big)  \otimes \big(\cO_{\text{ev}}(2m, n+1)\big)^{\text{GL}_m} ,\qquad n \ \text{odd}.\end{split} \end{equation}

It follows from Theorems \ref{spevengln} and \ref{orevengln} that $\big(\cS_{\text{ev}}(m, n+1)\big)^{\text{GL}_m}$ and $\big(\cO_{\text{ev}}(2m, n+1)\big)^{\text{GL}_m}$ are both of type $\cW(n+1,n+2,\dots, (m+1)(m+n+1)-1)$. Since $ \cO_{\text{ev}}(1,4) \otimes  \cO_{\text{ev}}(1,6) \otimes \cdots \otimes \cO_{\text{ev}}(1,2n)$ is of type $\cW(2,3,\dots, n)$, it follows that $ \lim_{\psi \rightarrow \infty} \cC^{\psi}(n,m)$ is of type $\cW(2,3,\dots, (m+1)(m+n+1)-1)$. Therefore $\cC^{\psi}(n,m)$ has the same type as a one-parameter vertex algebra.
\end{proof} 

In addition to the fields $L, \omega^3,\dots, \omega^n \in \cC^{\psi}(n,m)$, it is apparent from the proof of Theorems \ref{spevengln} and \ref{orevengln} that the additional strong generators $\omega^r$ for $n+1 \leq r \leq (m+1)(m+n+1)-1$, have the form
$$\omega^r = \sum_{i=1}^m :P^{+,i} (\partial^{r-n-1}P^{-,i}): + \cdots,$$ where the remaining terms are normally ordered monomials in the fields $\{J, e_{i,j}, h_k, L, \omega^3,\dots, \omega^n\}$ and their derivatives. It is not yet apparent that $\cC^{\psi}(n,m)$ is a one-parameter quotient of $\cW(c,\lambda)$ because we don't know that it is generated by the weight $3$ field $\omega^3$. Without loss of generality, we can modify $\omega^3$ by subtracting an appropriate multiple of $\partial L$ to make it $L$-primary. We then rescale it so that its $6$th order pole with itself is $\frac{c}{3} 1$, and we denote this modified field by $W^3$. We now consider the vertex subalgebra $\tilde{C}^{\psi}(n,m) \subseteq \cC^{\psi}(n,m)$ generated by $W^3$.

\begin{lemma} \label{lem:tildec} For $m \geq 1$ and $n\geq 0$, $\tilde{C}^{\psi}(n,m)$ is a quotient of $\cW(c,\lambda)$ for some ideal $I$, and is therefore of type $\cW(2,3,\dots, N)$ for some $N \leq  (m+1)(m+n+1)-1$.
\end{lemma}

\begin{proof} 
In the case $m=1$ and $n=0$, $\cC^{\psi}(1,0)$ is just the Heisenberg coset of the rank one $\beta\gamma$-system, which is known to be isomorphic to the Zamolodchikov $\cW_3$-algebra with $c=-2$ \cite{Wa}. For $m=1$ and $n =1$, $2$, and $3$, the claim is known by \cite[Theorems 7.1, 7.2, and 7.3]{LVI}, respectively. So we assume that $m\geq 1$ and $n\geq 4$.

Set $W^r = W^3_{(1)} W^{r-1}$, for $r \geq 4$. First, for $3 \leq r \leq (m+1)(m+n+1)-1$, we can write $$W^r = \lambda_r \omega^r + \cdots,\qquad \lambda_r \in \mathbb{C},$$ where the remaining terms are normally ordered monomials in $\{L, \omega^s |\ 3 \leq s < r\}$. If $\lambda_r \neq 0$ for all $r$, then $\tilde{\cC}^{\psi}(n,m) = \cC^{\psi}(n,m)$. Otherwise, let $N\geq 3$ be the first integer such that $\lambda_{N+1} = 0$. Then $\{L,W^3,\dots, W^N\}$ close under OPE, so that $\tilde{C}^{\psi}(n,m)$ is of type $\cW(2,3,\dots, N )$.

Even though $\cC^{\psi}(n,m)$ is generically simple by Lemma \ref{lem:nondegwnm}, it is not yet apparent that $\tilde{C}^{\psi}(n,m)$ is generically simple. However, by \cite[Thm. 5.2 and Rem. 5.1]{LVI}, it suffices to prove that the generators $\{L, W^r|\ 3 \leq r \leq 7 \}$ satisfy the OPE relations \eqref{ope:first}-\eqref{ope:fourth}, as well as (A1)-(A6) of \cite{LVI}; equivalently, all Jacobi identities \eqref{jacobi} of type $(W^i, W^j, W^k)$ for $i+j+k \leq 11$ hold as a consequence of \eqref{deriv}-\eqref{ncw} alone. In this notation, $W^2 = L$, as in \cite{LVI}.

By \cite[Thm. 6.2]{LVI}, the above condition is automatic if the graded character of $\tilde{\cC}^{\psi}(n,m)$ coincides with that of $\cW(c,\lambda)$ up to weight $8$. It follows from Theorems \ref{spevengln} and \ref{orevengln} in the cases $n$ even and $n$ odd, respectively, that there are no normally ordered relations in $\cC^{\psi}(n,m)$ among $\{L, \omega^r|\ 3 \leq r \leq  (m+1)(m+n+1)-1\}$ and their derivatives, in weight below $(m+1)(m+n+1)$. Therefore the character of $\cC^{\psi}(n,m)$ coincides with that of $\cW(c,\lambda)$ in weight up to $8$. If $N \geq 8$, $\tilde{\cC}^{\psi}(n,m)$ and $\cC^{\psi}(n,m)$ have the same graded character up to weight $8$, so the conclusion holds.

Finally, suppose that $N < 8$. Since $\lambda_{N+1} = 0$ and $\lambda_r \neq 0$ for $3\leq r \leq N$, there are no nontrivial normally ordered relations among $\{L, W^3,\dots, W^{N}\}$ in weight up to $N$, since this property holds for the corresponding fields $\{L, \omega^3,\dots, \omega^{N}\}$. Therefore all Jacobi relations among $\{L, W^3,\dots, W^{N}\}$ of type 
$(W^{i}, W^{j}, W^{k})$ for $i+j+k \leq N+2$, must hold as a consequence of \eqref{deriv}-\eqref{ncw} alone. It follows that the OPEs $W^{i}(z) W^{j}(w)$ for $i+j \leq N$ are  the same as those of $\cW^{I}(c,\lambda)$ for some ideal $I \subseteq \mathbb{C}[c,\lambda]$.

If we use the same procedure as the construction $\cW(c,\lambda)$ given by \cite[Thm. 5.1]{LVI}, beginning with the fields $L,W^3,\dots, W^{N}$ and the OPEs $W^{i}(z) W^{j}(w)$ for $i+j \leq N$, we can formally define new fields $W^{N+r} = (W^3_{(1)})^r W^{N}$ for all $r \geq 1$, and then define the OPE algebra of all fields $\{L, W^{3},\dots, W^{N}, W^{N+r}|\ r\geq 1\}$ recursively so that they are the same as the OPEs in $\cW^{I}(c,\lambda)$. In particular, this realizes $\tilde{\cC}^{\psi}(n,m)$ as a one-parameter quotient of $\cW^{I}(c,\lambda)$ by some vertex algebra ideal $\cI$ containing a field in weight $N+1$ of the form $W^{N+1} - P(L,W^3,\dots, W^{N})$, where $P$ is a normally ordered polynomial in $L, W^3,\dots, W^{N}$ and their derivatives. \end{proof}

Since $\cC^{\psi}(n,m)$ is at worst an extension of  $\tilde{\cC}^k(n,m)$, we obtain
\begin{cor} For $m \geq 2$ and $n\geq 0$, $\cW^{\psi}(n,m)$ is an extension of $\cH \otimes V^{\psi -m-1}(\gs\gl_m) \otimes \tilde{\cC}^{\psi}(n,m)$. Similarly, for $m=1$ and $n\geq 0$, $\cW^{\psi}(n,1)$ is an extension of $\cH \otimes \tilde{\cC}^{\psi}(n,1)$.
\end{cor}

\subsection{The reconstruction argument}
Let $\cW$ be any vertex algebra arising as a quotient of the algebra $\cW(c,\lambda)$ constructed in \cite{LVI}, with the usual strong generating set $\{L, W^i|\ i\geq 3\}$. First, we assume that $m\geq 2$ and $n\geq 0$, and we deal with the case $m=1$ and $n\geq 0$ later. We are interested in the problem of classifying certain extensions of $\cH \otimes V^{\psi -m-1}(\gs\gl_m) \otimes \cW$. We set the central charge of $L$ to be 
$$c =  -\frac{(n \psi  - m - n -1) (n \psi - \psi - m - n +1 ) (n \psi +  \psi  -m - n)}{(\psi -1) \psi},$$ and as in Lemma \ref{lem:wnmj} we normalize the generator $J$ of $\cH$ so that 
$$J(z) J(w) \sim -\frac{m (m + n - n \psi)}{m + n}  (z-w)^{-2}.$$
In $\cH \otimes V^{\psi -m-1}(\gs\gl_m) \otimes \cW$, the total Virasoro field is $T = L + L^{\gs\gl_m} + L^{\cH}$.

We now postulate that $\cH \otimes V^{\psi -m-1}(\gs\gl_m) \otimes \cW$ admits an extension which has $2m$ additional even strong generators $\{P^{\pm, i}|\ i = 1,\dots,m\}$ which are primary of conformal weight $\frac{n+1}{2}$ with respect to $T$, and satisfy
\begin{equation} \label{recon:1} 
\begin{split} & J(z) P^{\pm,i}(w) \sim \pm P^{\pm,i}(w)(z-w)^{-1},
\\ & e_{i,j}(z) P^{+,k}(w) \sim \delta_{j,k} P^{+,i}(w)(z-w)^{-1},
\\ & h_i(z) P^{+,j}(w) \sim (\delta_{1,j} - \delta_{i,j}) P^{+,j}(w)(z-w)^{-1}, 
\\ & e_{i,j}(z)  P^{-,k}(w) \sim - \delta_{i,k} P^{-,j}(w)(z-w)^{-1},
\\ & h_i(z) P^{-,j}(w) \sim (-\delta_{1,j} + \delta_{i,j}) P^{-,j}(w)(z-w)^{-1}.
\end{split}
\end{equation}
This forces
\begin{equation} \label{recon:2} \begin{split} & L(z) P^{+,1}(w) \sim \bigg(   \frac{n+1}{2} -  \frac{m^2-1}{2 m (\psi-1)} - \frac{m + n}{2m (n \psi -m - n)}   \bigg) P^{+,1}(w) (z-w)^{-2} 
\\ & + \bigg(\partial P^{+,1}  - \frac{m + n}{m (n \psi -m - n)} :JP^{+,1}: -\frac{1}{m(\psi-1)}  \sum_{i=1}^{m-1} :h_i P^{+,1}: 
\\ & - \frac{1}{\psi-1} \sum_{j=2}^m :e_{1,j} P^{+,j}: \bigg)(w)(z-w)^{-1}.
 \end{split} \end{equation}
There are similar expressions for $L(z) P^{+,i}(w)$ for $i>1$ and for $L(z) P^{-,j}(w)$, which we omit because we don't need them explicitly.

Since the fields $\{e_{i,j}, h_k,  L, W^i, P^{\pm, i} \}$ close under OPE, and $W^3$ commutes with $e_{i,j}, h_k$, the most general OPEs of $W^3$ with $P^{+,1}$ is
\begin{equation} \label{W3P1} 
\begin{split} W^3(z) P^{+,1}(w) & \sim a_0 P^{+,1}(w)(z-w)^{-3} + \bigg(a_1 \partial P^{+,1} + \dots  \bigg)(w)(z-w)^{-2} 
\\ & + \bigg(a_2 :LP^{+,1}: + a_3  \partial^2 P^{+,1} + \dots \bigg)(w) (z-w)^{-1},
\end{split} \end{equation}
Here the omitted expressions are not needed. We are going to impose just three Jacobi identities of type $(L,W^3, P^{+,1})$, and this will determine the constants $a_0, a_1, a_2$ in terms of $a_3$. First, we impose 
\begin{equation} \label{eqn:jacobi1} \begin{split} & L_{(2)} (W^3_{(1)} P^{+,1}) - W^3_{(1)} (L_{(2)} P^{+,1}) -   (L_{(0)} W^3)_{(3)} P^{+,1}   -2 (L_{(1)} W^3)_{(2)} P^{+,1} \\ & - (L_{(2)} W^3)_{(1)} P^{+,1} = 0.\end{split} \end{equation}
Using the above OPEs \eqref{recon:1}, \eqref{recon:2}, and \eqref{W3P1}, we get
\begin{equation} \label{eq:a0a1a2first} -3 a_0 + a_1 + a_1 n - \frac{(m^2-1) a_1}{m (\psi-1)} - \frac{a_1 (m + n)}{m (n \psi -m - n)} = 0.\end{equation}
Next, we impose 
\begin{equation} \label{eqn:jacobi2} \begin{split} & L_{(3)} (W^3_{(0)} P^{+,1}) - W^3_{(0)} (L_{(3)} P^{+,1}) -   (L_{(0)} W^3)_{(3)} P^{+,1}   -3 (L_{(1)} W^3)_{(2)} P^{+,1} \\ &  - 3(L_{(2)} W^3)_{(1)} P^{+,1} -(L_{(3)} W^3)_{(0)} P^{+,1} = 0.\end{split} \end{equation}
We get 
\begin{equation} \begin{split} \label{eq:a0a1a2second} & -6 a_0 + 2 a_2 + 3 a_3 + \frac{a_2 c}{2} + n (2 a_2  + 3 a_3)  -\frac{(m^2-1) (2 a_2 + 3 a_3)}{m (\psi - 1)} 
\\ &-  \frac{(2 a_2 + 3a_3) (m + n)}{m (n \psi -m - n)} =0.\end{split} \end{equation}
 Finally, we impose 
\begin{equation} \begin{split} \label{eqn:jacobi3} & L_{(2)} (W^3_{(0)} P^{+,1}) - W^3_{(0)} (L_{(2)} P^{+,1}) -   (L_{(0)} W^3)_{(2)} P^{+,1}   -2 (L_{(1)} W^3)_{(1)} P^{+,1}  \\ & - (L_{(2)} W^3)_{(0)} P^{+,1} = 0.\end{split} \end{equation}
 Extracting the coefficient of $\partial P^{+,1}$ yields
 \begin{equation} \label{eq:a0a1a2third} -4 a_1 + 3 a_2 + 4 a_3 + 2 a_3 n - \frac{2 (m^2-1) a_3}{m (\psi - 1)}- \frac{2 a_3 (m + n)}{m (n \psi -m - n)}.\end{equation}
 Solving \eqref{eq:a0a1a2first}, \eqref{eq:a0a1a2second}, and \eqref{eq:a0a1a2third}, we obtain
 \begin{equation} \label{a0a1a2} \begin{split} & a_0 = 
 \\ & \frac{(n \psi - m - n -2) (n \psi - m - n -1) (n \psi + \psi -m - n) (n \psi  + 2 \psi -m - n)}{6 (\psi-1)^2 (n \psi -m - n)^2} a_3,
\\  & a_1 =  \frac{(n \psi - m - n -2) (n \psi + 2 \psi -m - n )}{2 (\psi -1) (n \psi -m - n )} a_3,
\\ & a_2 = -\frac{2 \psi}{(\psi -1) (n \psi -m - n)} a_3.\end{split} \end{equation}

Next, we have
\begin{equation} \label{W4W5P1} 
\begin{split} & W^4(z) P^{+,1}(w) \sim b_0 P^{+,1}(w)(z-w)^{-4} + \cdots,
\\ & W^5(z) P^{+,1}(w) \sim b_1 P^{+,1}(w)(z-w)^{-5} + \cdots,\end{split} \end{equation} for some constants $b_0, b_1$. We will see that by imposing just four Jacobi identities, the constants $a_3,b_0, b_1$ are determined up to a sign, and the parameter $\lambda$ in $\cW(c,\lambda)$ is uniquely determined. First, we impose
\begin{equation} \label{eqn:jacobi4} \begin{split} &  W^3_{(3)} (W^3_{(1)} P^{+,1}) - W^3_{(1)} (W^3_{(3)} P^{+,1}) - (W^3_{(0)} W^3)_{(4)} P^{+,1} -3 (W^3_{(1)} W^3)_{(3)} P^{+,1} \\ & -3 (W^3_{(2)} W^3)_{(2)} P^{+,1}  - (W^3_{(3)} W^3)_{(1)} P^{+,1} =0.\end{split} \end{equation} This has weight $\frac{n+1}{2}$, and is therefore a scalar multiple of $P^{+,1}$. Using the OPE relations \eqref{ope:first}-\eqref{ope:fourth} together with \eqref{recon:1},  \eqref{recon:2}, \eqref{W3P1}, and \eqref{W4W5P1}, we compute this scalar to obtain
\begin{equation} \label{genjac:1}
1 + 3 a_0 a_1 - b_0 + n - \frac{m^2-1}{m (\psi-1)} - \frac{m + n}{m (n \psi  -m - n)}=0.
\end{equation}
Next, we impose
\begin{equation} \label{eqn:jacobi5}
\begin{split} & W^3_{(4)} (W^3_{(0)} P^{+,1})  - W^3_{(0)} (W^3_{(4)} P^{+,1}) 
- (W^3_{(0)} W^3)_{(4)} P^{+,1} - 4 (W^3_{(1)} W^3)_{(3)} P^{+,1} 
\\ & -6 (W^3_{(2)} W^3)_{(2)} P^{+,1}  - 4 (W^3_{(3)} W^3)_{(1)} P^{+,1} -  (W^3_{(4)} W^3)_{(0)} P^{+,1} =0.\end{split}\end{equation}  Again, this has weight $\frac{n+1}{2}$, and is therefore a scalar multiple of $P^{+,1}$, and we obtain
\begin{equation} \label{genjac:2}
1 + 6 a_0 (a_2 + 2 a_3) - 2 b_0 + n - \frac{m^2-1}{m (\psi -1)} - 
\frac{m + n}{m (n \psi  -m - n)} = 0.
\end{equation}
Next we impose \begin{equation} \label{eqn:jacobi6} W^3_{(0)} (W^4_{(5)} P^{+,1}) - W^4_{(5)} (W^3_{(0)} P^{+,1})  - (W^3_{(0)} W^4)_{(5)} P^{+,1} = 0,\end{equation} which yields
\begin{equation} \label{genjac:3} \begin{split} & \frac{1}{2}  \bigg(-40 a_3 b_0 + 5 a_0 \big((2 + c) \lambda - 16\big) + 4 b_1 \bigg) \\ & -  \frac{8 a_2 (m n \psi + m \psi - m^2 - m n  - m +1)}{m (\psi -1)} 
 \\ & -  a_2  (3 c + 8 b_0) + \frac{8a_2(m + n)}{m (n \psi  -m - n)} = 0.\end{split} \end{equation}
Finally, we impose
\begin{equation} \label{eqn:jacobi7} W^3_{(1)} (W^4_{(4)} P^{+,1})  - W^4_{(4)} (W^3_{(1)} P^{+,1})  - (W^3_{(0)} W^4)_{(5)} P^{+,1} - (W^3_{(1)} W^4)_{(4)} P^{+,1},\end{equation} which yields
\begin{equation} \label{genjac:4}
-8 a_1 b_0 + 5 a_0 \big((2 + c) \lambda-16 \big) + 2 b_1 =0.\end{equation}
Substituting the values of $a_0, a_1, a_2$ in terms of $a_3$ given by \eqref{a0a1a2} into the equations \eqref{genjac:1}-\eqref{genjac:4}, and solving for for $a_3, b_0, b_1, \lambda$ yields a unique solution for $b_0$ and $\lambda$, and a unique solution up to sign for $a_3$ and $b_1$. In particular, we obtain
\begin{equation} \label{a3b0b1} \begin{split} \lambda &  = -\frac{(\psi-1) \psi}{(n \psi - n - m -2) (n \psi - 2 \psi - m - n +2 ) (n \psi + 2 \psi  -m - n )},
\\ a_3 & = \pm \sqrt{-2}\  (\psi -1) \sqrt{\frac{ n \psi  -m - n}{r_1}} ,
\\ b_0 & = \frac{(n \psi - m - n -1) (n \psi + \psi -m - n) r_2 }{2 (\psi -1) ( n \psi - m - n)^2 (n \psi - 2 \psi - m - n +2)},
\\ b_1 & = \pm \frac{ \sqrt{-2}\ (n \psi - m - n -1) (n \psi + \psi -m - n ) r_3}{(\psi -1) (n \psi - 2 \psi - m - n +2) (n \psi -m - n)^{5/2} \sqrt{r_1}}.\end{split} \end{equation}
 In this notation, 
 \begin{equation} \begin{split} 
\\ r_1 & = (n \psi - n - m -2) (n \psi - 2 \psi - m - n +2 ) (n \psi + 2 \psi  -m - n ),
\\  r_2 & =  -6 m + m^2 - 6 n + 2 m n + n^2 + 4 \psi + 6 m \psi + 12 n \psi - 2 m n \psi - 2 n^2 \psi 
\\ & - 6 n \psi^2 + n^2 \psi^2,
\\ r_3 & = -36 m^2 - 4 m^3 + 5 m^4 - 72 m n - 12 m^2 n + 20 m^3 n - 36 n^2 - 12 m n^2 
\\ & + 30 m^2 n^2 - 4 n^3 + 20 m n^3 + 5 n^4 + 88 m \psi + 48 m^2 \psi + 4 m^3 \psi + 88 n \psi 
 \\ & + 168 m n \psi + 24 m^2 n \psi - 20 m^3 n \psi  + 120 n^2 \psi + 36 m n^2 \psi - 60 m^2 n^2 \psi 
 \\ & + 16 n^3 \psi - 60 m n^3 \psi - 20 n^4 \psi - 16 \psi^2 - 88 m \psi^2  - 36 m^2 \psi^2 - 176 n \psi^2 
 \\ & - 168 m n \psi^2 - 12 m^2 n \psi^2 - 168 n^2 \psi^2 - 36 m n^2 \psi^2 + 30 m^2 n^2 \psi^2  - 24 n^3 \psi^2 
 \\ & + 60 m n^3 \psi^2 + 30 n^4 \psi^2 + 88 n \psi^3 + 72 m n \psi^3 + 120 n^2 \psi^3 + 12 m n^2 \psi^3 
 \\ & + 16 n^3 \psi^3  - 20 m n^3 \psi^3 - 20 n^4 \psi^3 - 36 n^2 \psi^4 - 4 n^3 \psi^4 + 5 n^4 \psi^4 .
\end{split} \end{equation}
This proves the following
\begin{lemma} \label{lem:reconstructionwnm} Let $m\geq 2$ and $n\geq 0$. Suppose that $\cW$ is some quotient of $\cW(c,\lambda)$ and that $\cH \otimes V^{\psi -m-1}(\gs\gl_m) \otimes \cW$ admits an extension containing $2m$ primary fields $\{P^{\pm, i} |\ i = 1,\dots, m\}$ of conformal weight $\frac{n+1}{2}$, satisfying \eqref{recon:1},  \eqref{recon:2}, \eqref{W3P1}, and \eqref{W4W5P1}. Then $\cW$ is in fact a quotient of $\cW^{I_{n,m}}(c,\lambda) = \cW(c,\lambda) / I_{n,m} \cdot \cW(c,\lambda)$ where $I_{n,m} \subseteq \mathbb{C}[c,\lambda]$ is the ideal given in Theorem \ref{c(n,m)asquotient}. \end{lemma}

\begin{remark} Note that the formula for $a_3$ and $b_1$ involves square root functions, but this is just because we are using the convention of \cite{LVI} and scaling $W^3$ so that its leading pole is $\frac{c}{3}1$. With a different scaling, we can make all structure constants rational functions of $\psi$, but we keep this convention for convenience.\end{remark}

\begin{remark} The sign ambiguity in formula for $a_3$ and $b_1$ is not essential, and reflects the $\mathbb{Z}_2$-symmetry of $\cW(c,\lambda)$ and its quotients. \end{remark}

Next, we show that Lemma \ref{lem:reconstructionwnm} also holds in the case $m=1$ and $n\geq 0$. First, let $m=1$ and $n\geq 2$, and consider extensions of $\cH \otimes \cW$. Here $\cW$ is a one-parameter quotient of $\cW(c,\lambda)$ where the Virasoro field $L$ has central charge $$c = -\frac{(1 + n) (n \psi -n -2) (n \psi  - \psi -n)}{\psi},$$ and the  generator $J$ of $\cH$ satisfies
$$J(z) J(w) \sim  \frac{n \psi -n-1}{n+1} (z-w)^{-2}.$$
In $\cH \otimes \cW$, the total Virasoro field is $T = L + L^{\cH}$. We postulate that $\cH \otimes \cW$ admits an extension which has two additional odd strong generators $P^{\pm}$ which are primary of conformal weight $\frac{n+1}{2}$ with respect to $T$, and satisfy
\begin{equation} \label{reconm=1:1}  J(z) P^{\pm}(w) \sim \pm P^{\pm}(w)(z-w)^{-1}.
\end{equation}
This forces
\begin{equation} \label{reconm=1:2} \begin{split} & L(z) P^{+}(w) \sim \bigg(   \frac{n+1}{2}  - \frac{n+1}{2 (n \psi -n-1 )}   \bigg) P^{+}(w) (z-w)^{-2} 
\\ & + \bigg(\partial P^{+}  - \frac{n+1}{n \psi -n-1} :JP^{+}: \bigg)(w)(z-w)^{-1}.
 \end{split} \end{equation}
Next, we have the OPEs \eqref{W3P1} and \eqref{W4W5P1} with undetermined coefficients $a_0, a_1, a_2, a_3$ and $b_0,b_1$, where the terms we don't need are omitted. By imposing the same set of Jacobi relations  \eqref{eqn:jacobi1}, \eqref{eqn:jacobi2},  \eqref{eqn:jacobi3}, \eqref{eqn:jacobi4}, \eqref{eqn:jacobi5}, \eqref{eqn:jacobi6}, \eqref{eqn:jacobi7} as above, we find a unique solution $b_0$ and $\lambda$, and a unique solution up to sign for $a_0, a_1, a_2, a_3$ and $b_1$. In particular, Lemma \ref{lem:reconstructionwnm} holds in the case $m=1$ and $n\geq 2$. It is also easy to verify it directly in the cases $m=1$ and $n=0,1$.

\subsection{The exhaustiveness argument} 
In this subsection, we prove that $\tilde{\cC}^{\psi}(n,m) = \cC^{\psi}(n,m)$ as one-parameter vertex algebras. Recall that the specialization $\cC^{\psi_0}(n,m)$ of the one-parameter vertex algebra $\cC^{\psi}(n,m)$ at $\psi= \psi_0$, can be a proper subalgebra of the coset $\text{Com}(V^{\psi_0-m-1}(\gg\gl_m), \cW^{\psi_0}(n,m))$, but this can only occur at $\psi_0 = \frac{m+n}{n}$ when $J$ lies in the coset, or for rational numbers $\psi_0 \leq 1$. By abuse of notation, we shall use the same notation $\cC^{\psi}(n,m)$ if $\psi$ is regarded as a complex number rather than a formal parameter, so that $\cC^{\psi}(n,m)$ always denotes the specialization of the one-parameter algebra at $\psi \in \mathbb{C}$ even if it is a proper subalgebra of the coset. For all $\psi \in \mathbb{C}$, we denote by $\cC_{\psi}(n,m)$ the simple quotient of $\cC^{\psi}(n,m)$. Similarly, for all $\psi \in \mathbb{C}$, we denote by $\tilde{\cC}_{\psi}(n,m)$  the simple quotient of $\tilde{\cC}^{\psi}(n,m)$.

\begin{lemma} For $s\geq 3$, $m\geq 1$, and $n\geq 0$, we have isomorphisms of simple vertex algebras 
\begin{equation} \label{eq:firstcoinci} \tilde{\cC}_{\psi}(n,m) \cong \cW_r(\gs\gl_s),\qquad \psi = \frac{m + n + s}{n} ,\qquad r =-s+ \frac{m + s}{m + n + s}.\end{equation}
\end{lemma} 

\begin{proof} This is immediate from the fact that the truncation curves $V(I_{n,m})$ and $V(I_{s,0})$ intersect at the corresponding point $(c,\lambda)$ given by
\begin{equation} \begin{split} & c = -\frac{( s-1) (n s -m - s) (m + n + s + n s)}{(m + s) (m + n + s)}, 
\\ & \lambda =  \frac{(m + s) (m + n + s)}{(s-2) (2 m + 2 s - n s) (2 m + 2 n + 2 s + n s)}. \end{split}\end{equation}
\end{proof}

\begin{cor} \label{cor:cnmprelim} For $m\geq 1$ and $n\geq 0$, as a one-parameter vertex algebra $\tilde{\cC}^{\psi}(n,m)$ is of type $\cW(2,3,\dots, N)$, for some $N \geq (m+1)(m+n+1) -1$.
\end{cor}

\begin{proof} By Corollary \ref{wprinsingular}, for $\psi$ sufficiently large, $\cW^r(\gs\gl_s)$ has a singular vector in weight $(m+1)(m+n+1)$ and no singular vector in lower weight. Therefore $\cW_r(\gs\gl_s)$ is of type $\cW(2,3,\dots, N)$ for some $N$ between $(m+1)(m+n+1)-1$ and $s$, so the same holds for $\tilde{\cC}_{\psi}(n,m)$. The universal algebra $\tilde{\cC}^{\psi}(n,m)$ specialized at this value of $\psi$ cannot truncate below weight $N$, and therefore the same holds for the one-parameter algebra $\tilde{\cC}^{\psi}(n,m)$.
\end{proof}

\begin{cor} \label{cor:cnmexhaust} For $m\geq 1$ and $n\geq 0$, $\tilde{\cC}^{\psi}(n,m) = \cC^{\psi}(n,m)$ as one-parameter vertex algebras.
\end{cor}

\begin{proof} We have seen that $\cC^{\psi}(n,m)$ is of type $\cW(2,3,\dots, (m+1)(m+n+1) -1)$, and that $\tilde{\cC}^{\psi}(n,m)$ is a subalgebra of $\cC^{\psi}(n,m)$ of type $\cW(2,3,\dots, N)$ for some $N \geq (m+1)(m+n+1) -1$. The only possibility is that $N =  (m+1)(m+n+1) -1$ and $\tilde{\cC}^{\psi}(n,m) = \cC^{\psi}(n,m)$.
\end{proof}
 
 \noindent {\it Proof of Theorem \ref{c(n,m)asquotient}}. This now follows from Lemma \ref{lemma:stronggencnm}, Lemma \ref{lem:reconstructionwnm}, and Corollary \ref{cor:cnmexhaust}, together with the generic simplicity of $\cC^{\psi}(n,m)$. $\qquad \qquad \ \Box$

 \smallskip

An immediate corollary is that the rational $\cW$-algebras of type $A$ at nondegenerate admissible levels are organized into families of uniform truncation, and these families are labeled by the curves $V(I_{n,m})$. More precisely, we have

\begin{cor} \label{cor:uniformtrunctation} Fix $m \geq 1$ and $n\geq 0$. Then for all but finitely many values of $s \geq (m+1)(m+n+1)-1$, $\cW_r(\gs\gl_s)$ for $r =-s+ \frac{m + s}{m + n + s}$, is of type $\cW(2,3,\dots, (m+1)(m+n+1)-1)$. \end{cor} 

\begin{proof} Since $\cC^{\psi}(n,m)$ is of type $\cW(2,3,\dots, (m+1)(m+n+1)-1)$ as a one-parameter vertex algebra, there exists a decoupling relation in weight $ (m+1)(m+n+1)$ of the form \begin{equation} \label{eq:uniformtruncrelation} W^{(m+1)(m+n+1)} = P(L,W^3,\dots, W^{(m+1)(m+n+1)-1}),\end{equation} for some normally ordered polynomial $P$ in $L, W^3,\dots, W^{(m+1)(m+n+1)-1}$ and their derivatives, possibly after localization. Starting from this relation and applying the operator $(W^3)_{(1)}$ repeatedly, one can construct similar decoupling relations $$W^N = P_N(L,W^3,\dots, W^{(m+1)(m+n+1)-1}),\qquad N >(m+1)(m+n+1),$$ without introducing any additional poles. Therefore these decoupling relations exists for all but finitely many values of $\psi$. In particular, for all but finitely many of the values of $\psi$ appearing in \eqref{eq:firstcoinci}, both sides are of the desired type.
\end{proof}

We conjecture that Corollary \ref{cor:uniformtrunctation} in fact holds for all $s \geq (m+1)(m+n+1)-1$, but we cannot prove this without explicitly determining the denominators that appear in \eqref{eq:uniformtruncrelation}. For $m=1$ and $n = 3,4$, this relation was determined explicitly in \cite{ACLI,CLIV}, and our conjecture holds in these cases.

Finally, we can use Theorem \ref{c(n,m)asquotient} to classify all coincidences between the simple quotient $\cC_{\psi}(n,m)$ and principal $\cW$-algebras $\cW_r(\gs\gl_s)$ for $s\geq 3$. When $$\psi = \frac{m + n -1}{n}, \qquad \frac{m + n}{n-1}, \qquad \frac{m + n+1 }{n+1 },$$
we have $c = -2$, so by \cite[Thm. 10.1]{LVI}, $\cC_{\psi}(n,m)$ is isomorphic to the Zamolodchikov $\cW_3$-algebra. Similarly, for 
$$\psi = \frac{m + n -1}{n -1},\qquad \frac{m + n}{n+1}, \qquad \frac{m + n +1}{n},$$ we have $c = 0$ so $\cC_{\psi}(n,m) \cong \mathbb{C}$. For all other values of $\psi$ where the parametrization \eqref{cnmparametrization} is defined, the isomorphisms  $\cC_{\psi}(n,m)\cong \cW_{r}(\gs\gl_s)$ correspond to intersection points on the curves $V(I_{n,m})$ and $V(I_{s,0})$ by \cite[Cor. 10.1]{LVI}. At the points where the parametrization is not defined, $\cC_{\psi}(n,m)$ still exists but is not a quotient of $\cW(c,\lambda)$, and we need a different method to determine if $\cC_{\psi}(n,m) \cong \cW_r(\gs\gl_s)$ for some $r$ and $s$.

\begin{cor} \label{CWclassification} For all $m\geq 1$ and $n\geq 0$, we have the following isomorphisms $\cC_{\psi}(n, m) \cong \cW_{r}(\gs\gl_s)$ for $s\geq 3$:

\begin{enumerate}
\item $\displaystyle \psi =  \frac{m + n + s}{n} ,\qquad r =-s+ \frac{m + s}{m + n + s}$, 

\item $\displaystyle \psi =  \frac{m + n}{n + s},\qquad r = -s + \frac{s-m}{s+n}$, 

\item $\displaystyle \psi =  \frac{m + n - s}{n - s},\qquad r = -s+    \frac{m + n - s}{n - s}$. 
\end{enumerate}

Moreover, aside from the cases $c=0$ and $c=-2$ and the critical levels $\psi = 1$ for $\gs\gl_m$, and $\psi = 0$  for $\cW^{\psi}(n,m)$, these account for all coincidences $\cC_{\psi}(n, m) \cong \cW_{r}(\gs\gl_s)$ for $s\geq 3$, with the following possible exceptions:

\begin{enumerate}
\item $\displaystyle \psi = \frac{m + n +2}{n},\qquad s =  \frac{2 n}{2 + m} \in \mathbb{N}_{\geq 3}$,

\item $\displaystyle \psi =  \frac{m + n -2}{n-2}, \qquad s = \frac{2 m}{n-2} \in \mathbb{N}_{\geq 3}$.

\end{enumerate}

\end{cor}

\begin{proof} 
We first exclude the values \begin{equation} \label{eq:exceptionalpsi} \psi = \frac{m + n +2}{n}, \qquad   \frac{m + n -2}{n-2}, \qquad  \frac{m + n}{n+2},\end{equation} since it follows from the parametrization \eqref{cnmparametrization} that at these points, $\cC^{\psi}(n,m)$ is not obtained as a quotient of $\cW^{I_{n,m}}(c,\lambda)$.

 By \cite[Cor. 10.1]{LVI}, aside from the cases $c = 0,-2$, all remaining isomorphisms $\cC_{\psi}(n,m)\cong \cW_{r}(\gs\gl_s)$ correspond to intersection points on the curves $V(I_{n,m})$ and $V(I_{s,0})$. There are three intersection points $(c, \lambda)$, namely,
\begin{equation} \begin{split}
& \bigg( -\frac{(s-1) (n s -m - s) (m + n + s + n s)}{(m + s) (m + n + s)},\ -\frac{(m + s) (m + n + s)}{(s-2) (n s-2 m - 2 s) (2 m + 2 n + 2 s + n s)}\bigg),
\\ &  
\\ & \bigg(\frac{(s-1) (m - s + m s + n s) (n + s + m s + n s)}{(m - s) (n + s)}, \  \frac{(m - s) (n + s)}{(s-2) (2 m - 2 s + m s + n s) (2 n + 2 s + m s +
    n s)} \bigg),
\\ &  
\\ &  \bigg(  \frac{(s-1) (n - s - m s) (m + n - s + m s)}{(n - s) (m + n - s)},  \   \frac{(n - s) (m + n - s)}{(s-2) (2 n - 2 s - m s) (2 m + 2 n - 2 s + 
   m s)}\bigg),\end{split} \end{equation}
as long as $n,m,s$ are such that these points are defined. It is immediate that the above isomorphisms all hold, and that our list is complete except for possible coincidences at the excluded points \eqref{eq:exceptionalpsi}.

For $ \psi = \frac{m + n +2}{n}$, $\cC_{\psi}(n,m)$ has central charge $ c = -\frac{( 2 n -2 - m) (2 + m + 3 n)}{(2 + m) (2 + m + n)}$. Recall that $\cW_{r}(\gs\gl_s)$ has a singular vector in weight $3$ only for $c = 0$ and $ c = -\frac{(s-1) (3s+2)}{s+2}$. Therefore as long as $ \frac{2 n}{2 + m}$ is not an integer $s \geq 3$, there are no integers for which $\cW_{r}(\gs\gl_s)$ has a singular vector in weight $3$ at the above central charge, and we have no coincidence at this point. However, if $ s = \frac{2 n}{2 + m} \in \mathbb{N}_{\geq 3}$, it is possible that we have a coincidence at this point.

For $ \psi =  \frac{m + n -2}{n-2}$, $\cC_{\psi}(n,m)$ has central charge $ c = \frac{(n-2 - 2 m) (n-2 + 3 m)}{(n-2) (m + n -2)}$. By the same argument, there is no coincidence at this point unless $ s = \frac{2 m}{n-2} \in \mathbb{N}_{\geq 3}$, but in these cases it is possible to have a coincidence. 

For $ \psi =  \frac{m + n}{n+2}$, $\cC_{\psi}(n,m)$ has central charge $ c = \frac{(3 m + 2 n -2) (2 m + 3 n +2)}{( m -2) (2 + n)}$. By the same argument, there is no coincidence at this point unless $m = 1$ and $s = 2(n+1)$. A priori, it is possible that $\cC_{(n+1)/(n+2)}(n,1)$ is isomorphic to $\cW_r(\gs\gl_{2(n+1)})$ for $ r = -2(n+1) + \frac{n+1}{n+2}$, which has central charge $ c = -\frac{(2n+1) (3n+4)}{n+2}$. However, $\cC_{(n+1)/(n+2)}(n,1)$ is known to be isomorphic to the singlet algebra of type $\cW(2,2n+3)$ \cite{ACGY}, whereas $\cW_r(\gs\gl_{2(n+1)})$ is isomorphic to the Virasoro algebra. This follows from the fact that $\cW^r(\gs\gl_{2(n+1)})$ is generated by its weight $2$ and $3$ subspaces \cite[Prop. 5.2]{ALY}, but the weight $3$ field is singular. Hence there are no additional coincidences for $m=1$ and $\psi = \frac{n+1}{n+2}$. \end{proof}

\begin{remark} For the first family above, if we specialize to $\psi = m$ and $s = m n-m-n$, we obtain
$${\cC}_{m}(n,m) \cong \cW_{r}(\gs\gl_{mn-m-n}),\qquad r = -(mn -m-n) + \frac{m-1}{m}.$$
The case $n=2$, i.e., $\cC_m(2, m) \cong \cW_{-{m-2}+\frac{m-1}{m}}(\gs\gl_{m-2})$ is closely related to the conjecture that $ \text{Com}(L_{\psi-m-1}(\gg\gl_m), \cW^{\psi}(n,m))\cong \cW_{-{m-2}+\frac{m-1}{m}}(\gs\gl_{m-2})$ \cite[Conj. 4.3.2]{CY}. This conjecture
 implies that ordinary modules of $\gs\gl_{m+2}$  at level $-2$ have vertex tensor category structure.  \end{remark}

We now consider the case $m=1$, so that 
$$\cC^{\psi}(n,1) \cong \text{Com}(\cH, \cW^{\psi - n -1}(\gs\gl_{n+1}, f_{\text{subreg}})).$$ Specializing Theorem \ref{c(n,m)asquotient} to the case $m=1$ proves \cite[Conj. 9.5]{LVI}, which gave a conjectural description of the truncation curve. Therefore \cite[Conj. 10.2]{LVI}, which classifies coincidences between $\cC_{\psi}(n,1)$ and principal $\cW$-algebras of type $A$, is now a theorem as well. In particular, since $\cC_{\psi}(n,1)$ is isomorphic to $\text{Com}(\cH, \cW_{\psi - n -1}(\gs\gl_{n+1}, f_{\text{subreg}})$, we obtain the following result which was conjectured originally in \cite{B-H}.

\begin{cor} \label{cor:blumenhagen} For all $n\geq 1$, $\text{Com}(\cH, \cW_{\psi - n -1}(\gs\gl_{n+1}, f_{\text{subreg}})) \cong \cW_r(\gs\gl_s)$, where $\psi = \frac{ n + s+1}{n}$ and  $r =-s+ \frac{s+1}{s + n+1}$.
\end{cor}

Finally, we consider the case $n=2$ and $m\geq 2$, so that $\cC^{\psi}(2,m)$ is just the affine coset of the minimal $\cW$-algebra of $\gs\gl_{m+2}$. Specializing Theorem \ref{c(n,m)asquotient} to this case proves the conjectural description of the truncation curve given by \cite[Conj. 9.2]{LVI}. This proves \cite[Conj. 9.1]{LVI}, which is originally due to Kawasetsu \cite{Kaw}, as well as the classification of coincidences between $\cC_{\psi}(2,m)$ and principal $\cW$-algebras of type $A$ given by \cite[Conj. 10.1]{LVI}.

 \section{The structure of $\cD^{\psi}(n,m)$}  \label{section:dnm}
 
The main goal in this section is to realize the coset $\cD^{\psi}(n,m)$ of $\cV^{\psi}(n,m)$ as a simple, one-parameter quotient of $\cW(c,\lambda)$ of the form $\cW_{J_{n,m}}(c,\lambda)$ for some ideal $J_{n,m}$. 
 
 \begin{thm}  \label{d(n,m)asquotient} For $m \geq 1$ and $n\geq 1$, and for $m=0$ and $n\geq 3$, $\cD^{\psi}(n,m) \cong \cW_{J_{n,m}}(c,\lambda)$ where $J_{n,m}$ is described explicitly via the parametrization
\begin{equation} \begin{split} c(\psi) & = -\frac{(n \psi + m - n -1) (n \psi  - \psi + m - n +1) (n \psi + \psi +m - n )}{(\psi -1) \psi},
\\ \ 
\\ \lambda(\psi) & = -\frac{(\psi-1) \psi}{(n \psi + m - n -2) (n \psi - 2 \psi + m - n +2 ) (n \psi + 2 \psi +m - n)}.
\end{split} \end{equation}
Moreover, after a suitable localization, $\cW^{J_{n,m}}(c,\lambda)$ has a singular vector of the form
$$W^{(m+1)(n+1)} - P(L, W^3,\dots, W^{(m+1)(n+1)-1})$$ and no singular vector of lower weight, where $P$ is a normally ordered polynomial in the fields $L, W^3,\dots, W^{(m+1)(n+1)-1}$, and their derivatives. Therefore $\cW_{J_{n,m}}(c,\lambda)$ has minimal strong generating type $\cW(2,3,\dots, (m+1)(n+1)-1)$.
\end{thm}

As in the previous section, there is nothing to prove in the case $m=0$ and $n\geq 3$ since $\cD^{\psi}(n,0) = \cV^{\psi}(n,m) \cong \cW^{\psi-n}(\gs\gl_n)$, so for the rest of this section we assume $m \geq 1$. The strategy is the same as the proof of Theorem \ref{c(n,m)asquotient}. We first show that $\cD^{\psi}(n,m)$ is of type $\cW(2,3,\dots, (m+1)(n+1)-1)$ using its free field limit. Next, we show that $\cD^{\psi}(n,m)$ is at worst an extension of a vertex subalgebra $\tilde{\cD}^{\psi}(n,m)$ which is of type $\cW(2,3,\dots, N)$ for some $N \leq (m+1)(n+1)-1$, and is a one-parameter quotient of $\cW(c,\lambda)$. Therefore $\cV^{\psi}(n,m)$ is an extension of $\cH \otimes V^{-\psi-m+1}(\gs\gl_m) \otimes \tilde{\cD}^k(c,\lambda)$. A similar reconstruction argument then shows that if $\cW$ is any one-parameter quotient of $\cW^I(c,\lambda)$ for some ideal $I$ with the property that $\cH \otimes V^{-\psi-m+1}(\gs\gl_m) \otimes \cW$ admits an extension containing odd fields $\{P^{\pm, i}\}$ transforming as $\mathbb{C}^m \oplus (\mathbb{C}^m)^*$, then $\cW$ must be a quotient of $\cW^{J_{n,m}}(c,\lambda)$, where $J_{n,m}$ is the ideal given in  Theorem \ref{d(n,m)asquotient}. In particular, $\tilde{D}^{\psi}(n,m)$ must be a quotient of $\cW^{J_{n,m}}(c,\lambda)$. Finally, we use a similar exhaustiveness argument to prove that $\tilde{D}^{\psi}(n,m) = \cD^{\psi}(n,m)$.
 
\begin{lemma} \label{lemma:stronggendnm} For $m\geq 1$ and $n\geq 1$, $\cD^{\psi}(n,m)$ is of type $\cW(2,3,\dots, (m+1)(n+1)-1)$ as a one-parameter vertex algebra. Equivalently, this holds for generic values of $\psi$. \end{lemma}

\begin{proof} First, suppose that $m \neq n$, and recall the free field limit $\cV^{\text{free}}(n,m) := \cW^{\text{free}}(\gs\gl_{n|m}, f_{n|m})$. Then $\cV^{\text{free}}(n,m) \cong $ $$  \left\{
\begin{array}{ll}
\cO_{\text{ev}}(m^2,2) \otimes \big( \bigotimes_{i=2}^n \cO_{\text{ev}}(1,2i)\big)  \otimes \cS_{\text{odd}}(m, n+1), & n \ \text{odd},
\smallskip
\\ \cO_{\text{ev}}(m^2,2) \otimes \big( \bigotimes_{i=2}^n \cO_{\text{ev}}(1,2i)\big) \otimes \cO_{\text{odd}}(2m, n+1), & n \ \text{even}. \\
\end{array} 
\right.$$
Here $\cO_{\text{ev}}(m^2,2) = \cH(m^2)$ is just the rank $m^2$ Heisenberg algebra coming from the affine subalgebra, $\cO_{\text{ev}}(1,2i)$ is the algebra generated by $\omega^i$ for $i = 2,\dots, n$, and the fields $\{P^{\pm, i}\}$ generate $\cS_{\text{odd}}(m, n+1)$ or $\cO_{\text{odd}}(2m, n+1)$ when $n$ is odd or even, respectively. By Lemma \ref{lemma:cosetofw},
\begin{equation} \begin{split} \lim_{\psi \rightarrow \infty} \cD^{\psi}(n,m) & \cong \big( \big( \bigotimes_{i=2}^n \cO_{\text{ev}}(1,2i)\big) \otimes \cS_{\text{odd}}(m, n+1)\big)^{\text{GL}_m} 
\\ & \cong \big( \bigotimes_{i=2}^n \cO_{\text{ev}}(1,2i)\big) \otimes \big(\cS_{\text{odd}}(m, n+1)\big)^{\text{GL}_m} ,\qquad n \ \text{odd},\end{split} \end{equation}
\begin{equation} \begin{split} \lim_{\psi \rightarrow \infty} \cD^{\psi}(n,m) & \cong \big( \big( \bigotimes_{i=2}^n \cO_{\text{ev}}(1,2i)\big)  \otimes \cO_{\text{odd}}(2m, n+1)\big)^{\text{GL}_m} 
\\ & \cong \big( \bigotimes_{i=2}^n \cO_{\text{ev}}(1,2i)\big) \otimes \big(\cO_{\text{odd}}(2m, n+1)\big)^{\text{GL}_m} ,\qquad n \ \text{even}.\end{split} \end{equation}

It follows from Theorems \ref{spoddgln} and \ref{oroddgln} that $\big(\cS_{\text{odd}}(m, n+1)\big)^{\text{GL}_m}$ and $\big(\cO_{\text{odd}}(2m, n+1)\big)^{\text{GL}_m}$ are both of type $\cW(n+1,n+2,\dots, (m+1)(n+1)-1)$. Since $ \cO_{\text{ev}}(1,4) \otimes  \cO_{\text{ev}}(1,6) \otimes \cdots \otimes \cO_{\text{ev}}(1,2n)$ is of type $\cW(2,3,\dots, n)$, it follows that $ \lim_{k\rightarrow \infty} \cD^{\psi}(n,m)$ is of type $\cW(2,3,\dots, (m+1)(n+1)-1)$. Therefore $\cD^{\psi}(n,m)$ has the same type as a one-parameter vertex algebra.

Finally, we consider the case $m = n$. Then the free field limit $\cV^{\text{free}}(n,n) := \cW^{\text{free}}(\gp\gs\gl_{n|n}, f_{n|n})$ requires only a slight modification: $\cV^{\text{free}}(n,n) \cong $
$$ \left\{
\begin{array}{ll}
\cO_{\text{ev}}(m^2-1,2) \otimes \big( \bigotimes_{i=2}^n \cO_{\text{ev}}(1,2i)\big)  \otimes \cS_{\text{odd}}(n, n+1), & n \ \text{odd},
\smallskip
\\ \cO_{\text{ev}}(m^2-1,2) \otimes \big( \bigotimes_{i=2}^n \cO_{\text{ev}}(1,2i)\big)  \otimes \cO_{\text{odd}}(2n, n+1), & n \ \text{even}. \\
\end{array} 
\right.$$
Then we have 
\begin{equation} \begin{split} \lim_{\psi \rightarrow \infty} \cD^{\psi}(n,n) & \cong \big( \big( \bigotimes_{i=2}^n \cO_{\text{ev}}(1,2i)\big) \otimes \cS_{\text{odd}}(n, n+1)\big)^{\text{SL}_n \times \text{GL}_1}
\\ & \cong \big( \bigotimes_{i=2}^n \cO_{\text{ev}}(1,2i)\big) \otimes \big(\cS_{\text{odd}}(n, n+1)\big)^{\text{GL}_n} ,\qquad n \ \text{odd},\end{split} \end{equation}
\begin{equation} \begin{split} \lim_{\psi \rightarrow \infty} \cD^{\psi}(n,n) & \cong \big( \big( \bigotimes_{i=2}^n \cO_{\text{ev}}(1,2i)\big) \otimes \cO_{\text{odd}}(2n, n+1)\big)^{\text{SL}_n \times \text{GL}_1} 
\\ & \cong \big( \bigotimes_{i=2}^n \cO_{\text{ev}}(1,2i)\big) \otimes \big(\cO_{\text{odd}}(2n, n+1)\big)^{\text{GL}_n} ,\qquad n \ \text{even}.\end{split} \end{equation}
The rest of the argument is the same as the case $m\neq n$. \end{proof} 

As in the previous section, we modify the weight $3$ field $\omega^3 \in \cD^{\psi}(n,m)$ by subtracting an appropriate a multiple of $\partial L$ to make it $L$-primary. We then rescale it so that its $6$th order pole with itself is $\frac{c}{3} 1$, and we denote this modified field by $W^3$. We now consider the vertex subalgebra $\tilde{D}^{\psi}(n,m) \subseteq \cD^{\psi}(n,m)$ generated by $W^3$.

The proof of the next lemma is the same as the proof of Lemma \ref{lem:tildec}.

\begin{lemma} For $m \geq 1$ and $n\geq 1$, $\tilde{\cD}^{\psi}(n,m)$ is a quotient of $\cW(c,\lambda)$ for some ideal $J$, and is therefore of type $\cW(2,3,\dots, N)$ for some $N \leq  (m+1)(n+1)-1$.
\end{lemma}

Since $\cD^{\psi}(n,m)$ is at worst an extension of  $\tilde{\cD}^{\psi}(n,m)$, we obtain
\begin{cor} For $m \geq 2$ and $n\geq 1$,  and $m \neq n$, $\cV^{\psi}(n,m)$ is an extension of $\cH \otimes V^{-\psi-m+1}(\gs\gl_m) \otimes \tilde{\cD}^{\psi}(n,m)$. Similarly, for $m =1$ and $n\geq 2$, $\cV^{\psi}(n,1)$ is an extension of $\cH \otimes \tilde{\cD}^{\psi}(n,1)$. Finally, for $m =n \geq 1$, $\cV^{\psi}(n,n)$ is an extension of $V^{-\psi-n+1}(\gs\gl_n) \otimes \tilde{\cD}^{\psi}(n,n)$.
\end{cor}

\subsection{The reconstruction argument}
Let $\cW$ be any vertex algebra arising as a one-parameter quotient of $\cW(c,\lambda)$, with the usual strong generating set $\{L, W^i|\ i\geq 3\}$. First, we assume that $m\geq 2$, $n\geq 1$, and $m\neq n$. We shall deal with the cases $m=1$ and $n\geq 2$, and $m = n \geq 1$ separately. 
We set the central charge of $L$ to be 
$$ c =  -\frac{(n \psi + m - n -1) (n \psi  - \psi + m - n +1) (n \psi + \psi +m - n )}{(\psi -1) \psi},$$
and we normalize the generator $J$ of $\cH$ so that 
$$J(z) J(w) \sim \frac{m (n \psi + m - n)}{m - n}  (z-w)^{-2}.$$
In $\cH \otimes V^{-\psi-m+1}(\gs\gl_m) \otimes \cW$, the total Virasoro field is $T = L + L^{\gs\gl_m} + L^{\cH}$. 

We postulate that $\cH \otimes V^{-\psi-m+1}(\gs\gl_m) \otimes \cW$ admits an extension which has additionally $2m$ odd strong generators $\{P^{\pm,i}|\ i = 1,\dots, m\}$ which are primary of conformal weight $\frac{n+1}{2}$ with respect to $T$, and satisfy
\begin{equation} \label{oddrecon:1} 
\begin{split} & J(z) P^{\pm,i}(w) \sim \pm P^{\pm,i}(w)(z-w)^{-1},
\\ & e_{i,j}(z) P^{+,k}(w) \sim \delta_{j,k} P^{+,i}(w)(z-w)^{-1},
\\ & h_i(z) P^{+,j}(w) \sim (\delta_{1,j} - \delta_{i,j}) P^{+,j}(w)(z-w)^{-1}, 
\\ & e_{i,j}(z) P^{-,k}(w) \sim - \delta_{i,k} P^{-,j}(w)(z-w)^{-1},
\\ & h_i(z) P^{-,j}(w) \sim (-\delta_{1,j} + \delta_{i,j}) P^{-,j}(w)(z-w)^{-1}.
\end{split}
\end{equation}
This forces
\begin{equation} \label{oddrecon:2} \begin{split} & L(z) P^{+,1}(w) \sim \bigg(   \frac{n+1}{2} + \frac{m^2-1}{2 m (\psi -1)} + \frac{n - m}{2m (n \psi + m - n)}  \bigg) P^{+,1}(w) (z-w)^{-2} 
\\ & + \bigg(\partial P^{+,1}  + \frac{n - m}{m (n \psi + m - n)}  :JP^{+,1}: + \frac{1}{m(\psi-1)}  \sum_{i=1}^{m-1} :h_i P^{+,1}: 
\\ & + \frac{1}{\psi-1} \sum_{j=2}^m :e_{1,j} P^{+,j}: \bigg)(w)(z-w)^{-1}.
 \end{split} \end{equation}
There are similar expressions for $L(z) P^{+,i}(w)$ for $i>1$, and $L(z) P^{-,j}(w)$, which we omit. Since $\{e_{i,j}, h_k,  L, W^i, P^{\pm, i}\}$ close under OPE, and $W^3$ commutes with $e_{i,j}, h_k$, the most general OPEs of $W^3$ with $P^{+,1}$ is
\begin{equation} \label{oddW3P1} 
\begin{split} W^3(z) P^{+,1}(w) & \sim a_0 P^{+,1}(w)(z-w)^{-3} + \bigg(a_1 \partial P^{+,1} + \dots  \bigg)(w)(z-w)^{-2} 
\\ & + \bigg(a_2 :LP^{+,1}: + a_3  \partial^2 P^{+,1} + \dots \bigg)(w) (z-w)^{-1},
\end{split} \end{equation}
Here the omitted expressions are not needed. As in the previous section, we impose three Jacobi identities of type $(L,W^3, P^{+,1})$ in order to express the constants $a_0, a_1, a_2$ in terms of $a_3$. First, we impose \eqref{eqn:jacobi1}, and using \eqref{oddrecon:1}, \eqref{oddrecon:2}, and \eqref{oddW3P1}, we get
\begin{equation} \label{eq:odda0a1a2first} -3 a_0 + a_1 + a_1 n + \frac{(m^2-1) a_1}{m (\psi -1)} + \frac{a_1 (n - m)}{m (n \psi + m - n)} = 0.\end{equation}
Next, we impose \eqref{eqn:jacobi2}, and we get 
\begin{equation} \begin{split} \label{eq:odda0a1a2second} & -6 a_0 + 2 a_2 + 3 a_3 + \frac{a_2 c}{2} + n (2 a_2  + 3 a_3)  + \frac{(m^2-1) (2 a_2 + 3 a_3)}{m (\psi - 1)}\\ & +  \frac{(2 a_2 + 3a_3) (n - m)}{m (n \psi + m - n)} =0.\end{split} \end{equation}
Finally, we impose \eqref{eqn:jacobi3} and we extract the coefficient of $\partial P^{+,1}$, obtaining
 \begin{equation} \label{eq:odda0a1a2third} -4 a_1 + 3 a_2 + 4 a_3 + 2 a_3 n + \frac{2 (m^2-1) a_3}{m (\psi - 1)} + \frac{2 a_3 (n - m)}{m(n \psi + m - n)}.\end{equation}
Solving \eqref{eq:odda0a1a2first}, \eqref{eq:odda0a1a2second}, and \eqref{eq:odda0a1a2third}, we obtain
\begin{equation} \label{odda0a1a2} \begin{split} & a_0 = \\ & \frac{(n \psi + m - n -2) (n \psi + m - n -1) (n \psi + \psi +m - n) (n \psi + 2 \psi +m - n)}{6 (\psi-1)^2 (n \psi + m - n)^2} a_3,
\\ & a_1 = \frac{(n \psi + m - n -2) (n \psi + 2 \psi + m - n )}{2 (\psi-1) (n \psi + m - n)} a_3,
\\ & a_2 = -\frac{2 \psi}{(\psi-1) (n \psi + m - n)} a_3.\end{split} \end{equation}
 Next, we have
\begin{equation} \label{W4W5P1odd} 
\begin{split} & W^4(z) P^{+,1}(w) \sim b_0 P^{+,1}(w)(z-w)^{-4} + \cdots,
\\ & W^5(z) P^{+,1}(w) \sim b_1 P^{+,1}(w)(z-w)^{-5} + \cdots,\end{split} \end{equation} for some constants $b_0, b_1$. By imposing four Jacobi identities, the constants $a_3,b_0, b_1$ are determined up to a sign, and the parameter $\lambda$ in $\cW(c,\lambda)$ is uniquely determined. First, we impose \eqref{eqn:jacobi4}. This has weight $\frac{n+1}{2}$, and is therefore a scalar multiple of $P^{+,1}$. Using the OPE relations \eqref{ope:first}-\eqref{ope:fourth} together with \eqref{oddrecon:1},  \eqref{oddrecon:2}, \eqref{oddW3P1}, and \eqref{W4W5P1odd}, we compute this scalar to obtain
\begin{equation} \label{genjac:1odd}
1 + 3 a_0 a_1 - b_0 + n + \frac{m^2-1}{m (\psi -1)} + \frac{n - m}{m(n \psi + m - n)}=0.
\end{equation}
Next, we impose \eqref{eqn:jacobi5}. Again, this has weight $ \frac{n+1}{2}$, and is therefore a scalar multiple of $P^{+,1}$. We obtain
\begin{equation} \label{genjac:2odd}
1 + 6 a_0 (a_2 + 2 a_3) - 2 b_0 + n + \frac{m^2-1}{m (\psi  -1)} + 
\frac{n - m}{m(n \psi + m - n)} = 0.
\end{equation}
Next we impose  \eqref{eqn:jacobi6}, which yields
 \begin{equation} \label{genjac:3odd} \begin{split} & \frac{1}{2}  \bigg(-40 a_3 b_0 + 5 a_0 \big((2 + c) \lambda - 16\big) + 4 b_1 \bigg)  
 \\ & +  \frac{8 a_2 (1 + m - m^2 + m n - m \psi - m n \psi)}{m (\psi -1)} 
 \\ & -  a_2  (3 c + 8 b_0) + \frac{8a_2(n - m)}{m (n \psi + m - n)} = 0.\end{split} \end{equation}
Finally, we impose \eqref{eqn:jacobi7}, which yields
\begin{equation} \label{genjac:4odd}
-8 a_1 b_0 + 5 a_0 \big((2 + c) \lambda-16 \big) + 2 b_1 =0.\end{equation}
Substituting the values of $a_0, a_1, a_2$ in terms of $a_3$ given by \eqref{odda0a1a2} into the equations \eqref{genjac:1odd}-\eqref{genjac:4odd}, and solving for for $a_3, b_0, b_1, \lambda$ yields a unique solution for $b_0$ and $\lambda$, and a unique solution up to sign for $a_3$ and $b_1$. In particular, we obtain

\begin{equation} \begin{split} \lambda & =  -\frac{(\psi-1) \psi}{(n \psi + m - n -2) (n \psi - 2 \psi + m - n +2 ) (n \psi + 2 \psi +m - n)},
\\ a_3 & = \pm \sqrt{-2}\  (\psi-1) \sqrt{\frac{n \psi + m - n }{r_1}},
\\ b_0 & =  \frac{(n \psi + m - n -1) (n \psi + \psi + m - n ) r_2}{2 (\psi-1) (n \psi + m - n)^2 ( n \psi - 2 \psi +m -n +2)},
\\ b_1 & = - \pm \frac{ \sqrt{-2}\ (n \psi + m - n -1) (n \psi + \psi +m - n ) r_3}{(\psi-1) (n \psi + m - n)^{5/2} ( n \psi - 2 \psi + m - n +2)  \sqrt{r_1}}.
\end{split} \end{equation}
In this notation, 
 \begin{equation} \begin{split} 
r_1 & = (n \psi + m - n -2) ( n \psi - 2 \psi + m - n +2  ) (n \psi + 2 \psi  + m - n),
\\ r_2 & =  6 m + m^2 - 6 n - 2 m n + n^2 + 4 \psi - 6 m \psi + 12 n \psi + 2 m n \psi - 2 n^2 \psi 
\\ & - 6 n \psi^2 + n^2 \psi^2,
\\ r_3 & = -36 m^2 + 4 m^3 + 5 m^4 + 72 m n - 12 m^2 n - 20 m^3 n - 36 n^2 + 12 m n^2 
\\ & + 30 m^2 n^2 - 4 n^3 - 20 m n^3 + 5 n^4 - 88 m \psi + 48 m^2 \psi - 4 m^3 \psi + 88 n \psi 
\\ & - 168 m n \psi + 24 m^2 n \psi + 20 m^3 n \psi + 120 n^2 \psi - 36 m n^2 \psi - 60 m^2 n^2 \psi 
 \\ & + 16 n^3 \psi + 60 m n^3 \psi - 20 n^4 \psi - 16 \psi^2 + 88 m \psi^2 - 36 m^2 \psi^2 - 176 n \psi^2 
 \\ & + 168 m n \psi^2 - 12 m^2 n \psi^2 - 168 n^2 \psi^2 + 36 m n^2 \psi^2 + 30 m^2 n^2 \psi^2 - 24 n^3 \psi^2
 \\ & - 60 m n^3 \psi^2  + 30 n^4 \psi^2 + 88 n \psi^3 - 72 m n \psi^3 + 120 n^2 \psi^3 - 12 m n^2 \psi^3 
 \\ & + 16 n^3 \psi^3 + 20 m n^3 \psi^3 - 20 n^4 \psi^3 - 36 n^2 \psi^4 - 4 n^3 \psi^4 + 5 n^4 \psi^4.
\end{split} \end{equation}

This proves the following

\begin{lemma} \label{lem:reconstructionvnm} Let $m,n\geq 2$ and $m\neq n$. Suppose that $\cW$ is some quotient of $\cW(c,\lambda)$ and that $\cH \otimes V^{-\psi-m+1}(\gs\gl_m) \otimes \cW$ admits an extension containing $2m$ odd primary fields $\{P^{\pm, i} |\ i = 1,\dots, m\}$ of conformal weight $\frac{n+1}{2}$, satisfying \eqref{oddrecon:1},  \eqref{oddrecon:2}, \eqref{oddW3P1}, and \eqref{W4W5P1odd}. Then $\cW$ is in fact a quotient of $\cW^{J_{n,m}}(c,\lambda) = \cW(c,\lambda) / J_{n,m} \cdot \cW(c,\lambda)$ where $J_{n,m} \subseteq \mathbb{C}[c,\lambda]$ is the ideal given in Theorem \ref{d(n,m)asquotient}. \end{lemma}

As in the previous section, the formula for $a_3$ involves square root functions, but this is just because we scaling $W^3$ so that its leading pole is $\frac{c}{3}1$, as in \cite{LVI}. As before, the sign ambiguity in the formula for $a_3$ and $b_1$ reflects the $\mathbb{Z}_2$-symmetry of $\cW(c,\lambda)$ and its quotients.

Next, we show that Lemma \ref{lem:reconstructionvnm} also holds in the case $m = 1$ and $n\geq 1$. First we let $m=1$ and $n\geq 2$, and consider extensions of $\cH \otimes \cW$. Here $\cW$ is a one-parameter quotient of $\cW(c,\lambda)$ where Virasoro field $L$ has central charge
$$c = -\frac{n (n \psi  - \psi  - n +2) (n \psi + \psi -n +1)}{\psi},$$ and the  generator $J$ of $\cH$ satisfies
$$J(z) J(w) \sim -\frac{n \psi -n +1}{n -1}(z-w)^{-2}.$$ In $\cH \otimes \cW$, the total Virasoro field is $T = L + L^{\cH}$. We postulate that $\cH \otimes \cW$ admits an extension which has additional odd strong generators $P^{\pm}$ which are primary of conformal weight $\frac{n+1}{2}$ with respect to $T$, and satisfy
\begin{equation} \label{oddreconm=1:1}  J(z) P^{\pm}(w) \sim \pm P^{\pm}(w)(z-w)^{-1}.
\end{equation}
This forces
\begin{equation} \label{oddreconm=1:2}  \begin{split} & L(z) P^{+}(w) \sim \bigg(   \frac{n+1}{2}  + \frac{n - 1}{2 (n \psi -n + 1)}  \bigg) P^{+}(w) (z-w)^{-2} 
\\ & + \bigg(\partial P^{+}  + \frac{n - 1}{ (n \psi + 1 - n)}  :JP^{+}:  \bigg)(w)(z-w)^{-1}. \end{split} \end{equation}

Next, we have the OPEs \eqref{oddW3P1} and \eqref{W4W5P1odd} with undetermined coefficients $a_0, a_1, a_2, a_3$ and $b_0,b_1$. By imposing the same set of Jacobi relations \eqref{eqn:jacobi1}, \eqref{eqn:jacobi2},  \eqref{eqn:jacobi3}, \eqref{eqn:jacobi4}, \eqref{eqn:jacobi5}, \eqref{eqn:jacobi6}, \eqref{eqn:jacobi7} as above, we find a unique solution $b_0$ and $\lambda$, and a unique solution up to sign for $a_0, a_1, a_2, a_3$ and $b_1$. In particular, Lemma \ref{lem:reconstructionvnm} holds in the case $m=1$ and $n\geq 2$.

Finally, we consider the case $m,n\geq 2$ and $m = n$. We now consider extensions of $V^{-\psi-n+1}(\gs\gl_n) \otimes \cW$, where $\cW$ is a one-parameter quotient of $\cW(c,\lambda)$ with central charge $$c = -\frac{( n+1) (n \psi -1) ( n \psi - \psi  +1)}{ \psi-1}.$$ In $V^{-\psi-n+1}(\gs\gl_n) \otimes \cW$, the total Virasoro field is $T = L + L^{\gs\gl_n}$. We postulate that $V^{-\psi-n+1}(\gs\gl_n) \otimes \cW$ admits an extension which has $2n$ additional odd strong generators $P^{\pm,i}$ which are primary of conformal weight $\frac{n+1}{2}$ with respect to $T$, and satisfy
\begin{equation} \label{oddreconm=n:1} 
\begin{split} & e_{i,j}(z) P^{+,k}(w) \sim \delta_{j,k} P^{+,i}(w)(z-w)^{-1},
\\ &  h_i(z) P^{+,j}(w) \sim (\delta_{1,j} - \delta_{i,j}) P^{+,j}(w)(z-w)^{-1}, 
\\ & e_{i,j}(z) P^{-,k}(w) \sim -\delta_{i,k} P^{-,j}(w)(z-w)^{-1},
\\ & h_i(z) P^{-,j}(w) \sim (-\delta_{1,j} + \delta_{i,j}) P^{-,j}(w)(z-w)^{-1}.
\end{split}
\end{equation}
This forces
\begin{equation} \label{oddreconm=n:2} \begin{split} & L(z) P^{+,1}(w) \sim \bigg( \frac{n+1}{2} + \frac{n^2-1}{2 n (\psi -1)} \bigg) P^{+,1}(w) (z-w)^{-2} 
\\ & + \bigg(\partial P^{+,1}   + \frac{1}{n(\psi -1)}  \sum_{i=1}^{m-1} :h_i P^{+,1}:  + \frac{1}{\psi -1} \sum_{j=2}^m :e_{1,j} P^{+,j}: \bigg)(w)(z-w)^{-1}.
 \end{split} \end{equation}
There are similar expressions for $L(z) P^{\pm,i}(w)$ which we omit. As usual, we have the OPEs \eqref{oddW3P1} and \eqref{W4W5P1odd} with undetermined coefficients $a_0, a_1, a_2, a_3$ and $b_0,b_1$. By imposing the same set of Jacobi relations as above, we find a unique solution $b_0$ and $\lambda$, and a unique solution up to sign for $a_0, a_1, a_2, a_3$ and $b_1$. In particular, Lemma \ref{lem:reconstructionvnm} holds in the case $m=n$ and $n\geq 2$. It is also easy to verify it directly in the case $m=n = 1$.

\subsection{The exhaustiveness argument}

In this subsection, we prove that $\tilde{\cD}^{\psi}(n,m) = \cD^{\psi}(n,m)$ as one-parameter vertex algebras. Recall that the specialization $\cD^{\psi_0}(n,m)$ of $\cD^{\psi}(n,m)$ at $\psi= \psi_0$, can be a proper subalgebra of the coset $\text{Com}(V^{-\psi_0-m+1}(\gg\gl_m), \cV^{\psi_0}(n,m))$ in the case $n\neq m$, or of the orbifold $\text{Com}(V^{-\psi_0-n+1}(\gs\gl_n), \cV^{\psi_0}(n,n))^{\text{GL}_1}$ in the case $n=m$, but this can only occur for rational numbers $\psi_0 \leq 1$, or when $\psi_0 = \frac{n-m}{n}$ in the case $n\neq m$, since $J$ then lies in the coset. As before, we use the same notation $\cD^{\psi}(n,m)$ if $\psi$ is regarded as a complex number rather than a formal parameter, so that $\cD^{\psi}(n,m)$ always denotes the specialization of the one-parameter algebra at $\psi \in \mathbb{C}$ even if it is proper subalgebra of the coset. For all $\psi \in \mathbb{C}$, we denote by $\cD_{\psi}(n,m)$ the simple quotient of $\cD^{\psi}(n,m)$. Similarly, for all $\psi \in \mathbb{C}$, we denote by $\tilde{\cD}_{\psi}(n,m)$  the simple quotient of $\tilde{\cD}^{\psi}(n,m)$.

\begin{lemma} For $s\geq 3$, $m\geq 1$, and $n\geq 1$, we have isomorphisms of simple vertex algebras 
$$\tilde{\cD}_{\psi}(n,m) \cong \cW_r(\gs\gl_s),\qquad \psi =  \frac{n - m}{n + s} ,\qquad r =-s+ \frac{m + s}{n + s}.$$
\end{lemma} 

\begin{proof} This is immediate from the fact that the truncation curves $V(J_{n,m})$ and $V(I_{s,0})$ intersect at the corresponding point $(c,\lambda)$ given by
\begin{equation} \begin{split} & c = -\frac{(s-1) (m s - n s-n - s) (m + s + m s - n s)}{(m + s) (n + s)},
\\ & \lambda =  -\frac{(m + s) (n + s)}{(s-2) (m s - n s-2 n - 2 s) (2 m + 2 s + m s - n s)}. \end{split} \end{equation}
\end{proof}

By the same argument as the proof of Corollary \ref{cor:cnmprelim}, we obtain
\begin{cor} For $m\geq 1$ and $n\geq 1$, as a one-parameter vertex algebra, $\tilde{\cD}^{\psi}(n,m)$ is of type $\cW(2,3,\dots, N)$, for some $N \geq (m+1)(n+1) -1$.
\end{cor}

Since $\cD^{\psi}(n,m)$ is of type $\cW(2,3,\dots, (m+1)(n+1) -1)$, and $\tilde{\cD}^{\psi}(n,m)$ is a subalgebra of $\cD^k(n,m)$ of type $\cW(2,3,\dots, N)$ for some $N \geq (m+1)(n+1) -1$, we must have $N =  (m+1)(n+1) -1$, and we immediately obtain
\begin{cor} \label{cor:dnmexhaust} For $m\geq 1$ and $n\geq 1$, $\tilde{\cD}^{\psi}(n,m) = \cD^{\psi}(n,m)$ as one-parameter vertex algebras.
\end{cor}

\noindent {\it Proof of Theorem \ref{d(n,m)asquotient}}. This now follows from Lemma \ref{lemma:stronggendnm}, Lemma \ref{lem:reconstructionvnm}, and Corollary \ref{cor:dnmexhaust}, together with the generic simplicity of $\cD^{\psi}(n,m)$.  $\qquad \qquad \Box$

\section{Proof of Main Theorem}

Finally, we are ready to prove the main result of this paper, Theorem \ref{intro:mainthm}. 

\smallskip

\noindent {\it Proof of Theorem \ref{intro:mainthm}}. In all cases where $\cC^{\psi}(n,m)$ and $\cD^{\psi}(n,m)$ arise as quotients of $\cW(c,\lambda)$, the statement follows immediately from the parametrizations of $I_{n,m}$ and $J_{n,m}$ given by Theorems \ref{c(n,m)asquotient} and \ref{d(n,m)asquotient}, together with \cite[Cor. 10.2]{LVI} which says that the simple one-parameter quotients of $\cW(c,\lambda)$ are in bijection with the set of truncation curves.

In the case $\cD^\psi(2, 0)  \cong \cC^{\psi^{-1}}(2, 0) \cong \cD^{\psi'}(0, 2)$, the first isomorphism is just Feigin-Frenkel duality for the Virasoro algebra, and the second follows from \cite[Thm. 8.7]{ACLII} for $\gg = \gs\gl_2$.

Finally, the cases $\cD^\psi(1, 0)  \cong \cC^{\psi^{-1}}(1,0) \cong \cD^{\psi'}(0, 1)$ and $\cD^\psi(0, 0)  \cong \cC^{\psi^{-1}}(0,0) \cong \cD^{\psi'}(0, 0)$ hold trivially because all these vertex algebras are just $\mathbb{C}$. $\qquad \qquad \qquad \qquad \qquad \qquad \qquad \qquad \qquad \qquad \qquad \qquad \qquad \qquad\ \  \Box$

\section{Uniqueness and reconstruction} \label{section:uniqueness}

In this section, we prove a strong uniqueness theorem for the $\cW$-algebras $\cW^{\psi}(n,m)$ and $\cW$-superalgebras $\cV^{\psi}(n,m)$. As a corollary, in the case $m=1$, we exhibit $\cW_{\psi}(\gs\gl_{n+1}, f_{\text{subreg}})$ as a simple current extension of $V_L \otimes  \cW_r(\gs\gl_s)$, in the case $ \psi =  \frac{n + s+1}{n}$ and $ r =-s+ \frac{s+1}{s + n +1}$, where $s\geq 3$ and $s+1$, $s+n+1$ coprime. Here $V_L$ is the lattice vertex algebra for $L = \sqrt{s(n+1)} \ \mathbb{Z}$. This gives a new and independent proof of Arakawa and van Ekeren's recent result that $\cW_k(\gs\gl_{n+1}, f_{\text{subreg}})$ is rational for these values of $k = \psi-n-1$ \cite{AvEII}. In the case $m>1$, we conjecture that for $ \psi= \frac{m + n + s}{n}$, $\cH \otimes L_{\psi-m-1}(\gs\gl_m)$ embeds in $\cW_{\psi}(n,m)$, which by Corollary \ref{CWclassification} and \cite[Lemma 2.1]{ACKL} would imply that $\text{Com}(\cH \otimes L_{\psi-m-1}(\gs\gl_m),\cW_{\psi}(n,m)) \cong \cW_r(\gs\gl_s)$ for $ \psi =  \frac{m + n + s}{n}$ and $ r =-s+ \frac{m + s}{m + n + s}$. We indicate how our uniqueness theorem will be used to prove this conjecture in future work.

\begin{thm} \label{thm:uniquenesswnm} For $m\geq 1$ and $n\geq 0$, the full OPE algebra of $\cW^{\psi}(n,m)$ is determined completely from the structure of $\cW_{I_{n,m}}(c,\lambda)$, the normalization of $J$, the action of $\gg\gl_m$ on the generators $\{P^{\pm, i}\}$, and the nondegeneracy condition $$P^{+,i}_{(n)} P^{-,j} = \delta_{i,j} 1.$$ 

In particular, for $m\geq 2$, if the generator $J$ of $\cH$ is normalized as in Lemma \ref{lem:wnmj}, and $\cA^{\psi}(n,m)$ is a one-parameter vertex algebra which extends $\cH \otimes V^{\psi-m-1}(\gs\gl_m) \otimes \cW_{I_{n,m}}(c,\lambda)$ by even fields $\{P^{\pm, i}|\ i = 1,\dots, m\}$ of conformal weight $\frac{n+1}{2}$ which are primary with respect to $\cH \otimes V^{\psi-m-1}(\gs\gl_m)$ as well as the total Virasoro field $T = L + L^{\gs\gl_m} + L^{\cH}$, then $\cA^{\psi}(n,m)\cong \cW^{\psi}(n,m)$. 

Similarly, if $m = 1$ and $J$ is normalized as in Lemma \ref{lem:wnmj}, and $\cA^{\psi}(n,1)$ is a one-parameter vertex algebra which extends $\cH \otimes \cW_{I_{n,1}}(c,\lambda)$ by even fields $\{P^{\pm}\}$ of conformal weight $\frac{n+1}{2}$ which are primary with respect to the action of $\cH$ as well as the total Virasoro field $T = L + L^{\cH}$, then $\cA^{\psi}(n,1) \cong \cW^{\psi}(n,1)$. 
\end{thm}

\begin{proof} Suppose first that $m > 1$. Let $\cA^{\psi}(n,m)$ be a vertex algebra extension of $\cH\otimes V^{\psi-m-1}(\gs\gl_m) \otimes \cW_{I_{n,m}}(c,\lambda)$ by fields $\{P^{\pm, i}|\ i = 1,\dots, m\}$ satisfying the above properties. Recall that the generator $W^3\in \cW_{I_{n,m}}(c,\lambda)$ is normalized so that $W^3_{(5)} W^3 = \frac{c}{3} 1$, and $W^i = W^3_{(1)} W^{i-1}$ for $i = 4,\dots, n$. Our assumptions imply that \eqref{recon:1},  \eqref{recon:2}, \eqref{W3P1}, and \eqref{W4W5P1} must hold, as well as the formulas \eqref{a0a1a2} and \eqref{a3b0b1}.

First, we claim that the OPE $W^3(z) P^{+,1}(w)$ is completely determined by known OPEs up to the sign of $a_3$; in particular, all terms appearing in this OPE are linear in $a_3$, and are given by $a_3 f(\psi)$ for some algebraic function $f(\psi)$. By weight considerations, we only need to compute $W^3_{(0)} P^{+,1}$ and $W^3_{(1)} P^{+,1}$, since $W^3_{(2)} P^{+,1} = a_0 P^{+,1}$ and $a_0$ has been computed above. We must have
$$W^3_{(1)} P^{+,1} = a_1 \partial P^{+,1} + \sum_{j=2}^m \lambda_j :e_{1,j} P^{+,j}: + \sum_{i=1}^{m-1} \mu_i :h_i P^{+,1}:,$$ for constants $\lambda_j, \mu_i$. Using the fact that 
\begin{equation} \begin{split} & J_{(1)} (W^3_{(1)} P^{+,1}) = 0,\qquad (e_{j,1})_{(1)}  (W^3_{(1)} P^{+,1}) = 0,\qquad (h_i)_{(1)}  (W^3_{(1)} P^{+,1}) = 0,\\ & J_{(1)} \partial P^{+,1} = P^{+,1},\qquad (e_{j,1})_{(1)} \partial P^{+,1} = P^{+,j},\qquad (h_i)_{(1)}   \partial P^{+,1} = P^{+,1},\end{split} \end{equation} the constants $\lambda_j, \mu_i$ are uniquely determined. Finally, $W^3_{(0)} P^{+,1}$ is uniquely determined from the identity
$$L_{(2)} (W^3_{(1)} P^{+,1}) - W^3_{(1)} (L_{(2)} P^{+,1}) - \sum_{i=0}^2 (L_{(i)} W^3)_{(3-i)} P^{+,1} = 0.$$

Since $\{P^{+,i}|\ i = 1,\dots, m\}$ is irreducible as an $\gs\gl_m$-module, by acting by elements of $V^{\psi-m-1}(\gs\gl_m)$ on $W^3_{(r)}P^{+,1}$ for $r = 0,1,2$, the OPEs $W^3(z) P^{+,i}(w)$ are also determined uniquely up to the sign of $a_3$. Next, we claim that for $j = 4,\dots, n$, the OPEs $W^j(z) P^{\pm,j}(w)$ are uniquely determined up to the sign of $a_3$. This follows by induction on $j$ from the identity, 
\begin{equation} \label{jacobiw3wjp} W^3_{(r)} (W^j_{(s)} P^{\pm,i}) - W^j_{(s)} (W^3_{(r)} P^{\pm,i}) - \sum_{i=0}^r \binom{r}{i} (W^3_{(i)} W^j)_{(r+s-i)} P^{\pm,i},\end{equation} together with the fact that $(W^3)_{(1)} W^{r-1} = W^r$ for $r \geq 4$.

Next, we need to consider the OPEs $W^j(z) P^{-,i}(w)$. We can carry out same procedure as in Section \ref{section:cnm} using $P^{-,1}$ instead of $P^{+,1}$. Starting with the OPEs 
\begin{equation} \label{W3W4W5Q1} 
\begin{split} W^3(z) P^{-,1}(w) & \sim \tilde{a}_0 P^{-,1}(w)(z-w)^{-3} + \bigg(\tilde{a}_1 \partial P^{-,1} + \dots  \bigg)(w)(z-w)^{-2} 
\\ & + \bigg(\tilde{a}_2 :LP^{-,1}: + \tilde{a}_3  \partial^2 P^{-,1} + \dots \bigg)(w) (z-w)^{-1},
\\ W^4(z) P^{-,1}(w) & \sim \tilde{b}_0 P^{-,1}(w)(z-w)^{-4} + \cdots,
\\  W^5(z) P^{-,1}(w) & \sim \tilde{b}_1 P^{-,1}(w)(z-w)^{-5} + \cdots,\end{split} \end{equation} solving the same set of Jacobi identities  \eqref{eqn:jacobi1}, \eqref{eqn:jacobi2},  \eqref{eqn:jacobi3}, \eqref{eqn:jacobi4}, \eqref{eqn:jacobi5}, \eqref{eqn:jacobi6}, \eqref{eqn:jacobi7} as before, we obtain the same value for $\lambda$ as well as $\tilde{a}_i$ and $\tilde{b}_i$, with the same sign ambiguity in $\tilde{a}_3$ and $\tilde{b}_1$. The choice of sign for $a_3$ and $\tilde{a}_3$ are not independent; it turns out that we must have $\tilde{a}_3 = -a_3$. This is a consequence of the nondegeneracy given by Lemma \ref{lem:nondegwnm}. We have the identity
\begin{equation} \label{jacobiw3pp1} W^3_{(0)} ((P^{+,1})_{(n)} P^{-,1}) - (P^{+,1})_{(n)} (W^3_{(0)} P^{-,1}) -   (W^3_{(0)} P^{+,1})_{(n)} P^{-,1} = 0.\end{equation} 
 The first term vanishes because $(P^{+,1})_{(n)} P^{-,1}$ is a constant. As for the remaining terms, recall that all terms appearing in $W^3_{(0)} P^{+,1}$ are linear in the scalar $a_3$, which is the coefficient of $:L P^{+,1}:$. Similarly, all terms appearing in $W^3_{(0)} P^{-,1}$ are linear in $\tilde{a}_3$, which is the coefficient of $:L P^{-,1}:$. From second term above, we obtain a multiple of $L$ coming from $-(P^{+,1})_{(n)} (\tilde{a}_3:L P^{-,1}:)$, and the only such term is $-\tilde{a}_3 :L ((P^{+,1})_{(n)} P^{-,1}):\ = -\tilde{a}_3 L$, since $(P^{+,1})_{(n)} P^{-,1} = 1$. Similarly, from the third term, the only contribution comes from $-(a_3:L P^{+,1}:)_{(n)} P^{-,1}$, and yields $-a_3 L$. This forces $\tilde{a}_3 = -a_3$. 
 
By the above argument, all OPEs $W^j(z) P^{-,i}(w)$ are then determined uniquely up to the choice of sign of $a_3$. Next, we claim that all OPEs $P^{\pm,i}(z) P^{\mp,j}(w)$ are completely determined up to the sign of $a_3$. This follows inductively from the identity
\begin{equation} \label{jacobiw3pp2} \begin{split} & W^3_{(1)} ((P^{+,i})_{(r+1)} P^{-,j}) - (W^3_{(1)} P^{+,i})_{(r+1)} P^{-,j})  - (P^{+,i})_{(r+1)} (W^3_{(1)} P^{-,j}) 
\\ & - (W^3_{(0)} P^{+,i})_{(r+2)} P^{-,j}=0, \end{split} \end{equation} together with the fact that all OPEs $W^j(z) P^{\pm, i}(w)$ are determined up to the sign of $a_3$. However, changing the sign of $a_3$ corresponds to rescaling the field $W^3$ by $-1$ and does not change the isomorphism type of $\cA^{\psi}(n,m)$.

This argument shows that $\cA^{\psi}(n,m)$ and $\cW^{\psi}(n,m)$ have the same strong generators and OPE algebras. Finally, since $\cW^{\psi}(n,m)$ is freely generated and simple as a one-parameter vertex algebra (equivalently, this holds for generic values of $\psi)$, its universal enveloping vertex algebra in the sense of \cite{DSKI} is already simple, and hence is the unique object in the category of vertex algebras with this OPE algebra. It follows that $\cA^{\psi}(n,m) \cong \cW^{\psi}(n,m)$ as one-parameter vertex algebras.

In the case $m=1$, the proof is similar but easier since there is no action of $\gs\gl_m$. First, the OPE $W^3(z) P^+(w)$ is completely determined up to the sign of $a_3$. Using \eqref{jacobiw3wjp}, the OPEs $W^j(z) P^+(w)$ for $j = 4,\dots, n$, are also determined up to this sign. Next, the OPE $W^3(z) P^-(w)$ is uniquely determined up to the sign of $\tilde{a}_3$, and \eqref{jacobiw3pp1} implies that $\tilde{a}_3 = -a_3$. The OPE $P^+(z) P^-(w)$ is then determined from \eqref{jacobiw3pp2}, up to this sign. Finally, the choice of sign for $a_3$ does not affect the isomorphism type of $\cA^{\psi}(n,1)$. \end{proof}

\begin{remark} In the above theorem, suppose that the formal parameter $\psi$ is specialized to some complex number $\psi_0$. Since the OPE algebra of $\cA^{\psi_0}(n,m)$ is the same as the OPE algebra of $\cW^{\psi_0}(n,m)$, the simple quotients $\cA_{\psi_0}(n,m)$ and $\cW_{\psi_0}(n,m)$ must also coincide. \end{remark}

Under some mild arithmetic conditions on $n$ and $k$, we shall now use this result to exhibit $\cW_k(\gs\gl_{n+1}, f_{\text{subreg}})$ as a simple current extension of $V_L \otimes  \cW_r(\gs\gl_s)$ for $L=\sqrt{s(n+1)}\ \mathbb{Z}$, in the case $ k = -(n+1) + \frac{n + s+1}{n}$ and $ r =-s+ \frac{s+1}{s+n+1}$.

Let $n, r$ be in $\mathbb{Z}_{>1}$ such that $n+1$ and $n+r$ are coprime (in particular, $nr$ is even). Following the notation in \cite{CLIV}, we define
$$
\cY(n, r) := \cW_{\ell}(\mathfrak{sl}_n), \qquad \ell =-n + \frac{n+r}{n+1}.
$$ Let $L=\sqrt{nr}\ \mathbb{Z}$ and $V_L$ the lattice vertex algebra of $L$. Recall that the modules for $\cY(n, r)$ are parameterized by modules of $L_{r}\left(\mathfrak{sl}_n\right)$, i.e. by integrable positive weights of $\widehat{\mathfrak{sl}}_n$ at level $r$. Then \cite[Main Theorem 4]{CII} gives the fusion rules; see also \cite{FKW, AvE} for these fusion rules assuming a certain coprime condition. In particular, $$L_{r\omega_s} \boxtimes L_{r\omega_t} = L_{r \omega_{\overline{r+s}}} ,\qquad \overline{r+s} = \left\{
\begin{array}{ll} r+s & r+s < n,
\\ r+s-n & r+s \geq n ,  \\
\end{array} 
\right.
$$
and we identify $\omega_0$ with zero. In \cite{CLIV}, it was shown that
\begin{equation}
A(n, r) \cong  \bigoplus_{s=0}^{n-1} V_{L+\frac{rs}{\sqrt{rn}}} \otimes \mathbb L_{r\omega_s}
\end{equation}
is a simple vertex algebra extending $V_L \otimes \cY(n, r)$. If $r$ is even, this is a $\mathbb{Z}$-graded vertex algebra, while for odd $r$ it is only $\frac{1}{2}\mathbb{Z}$-graded. 
The subspace of lowest conformal weight in each of the $V_{L+\frac{rs}{\sqrt{rn}}} \otimes\mathbb L_{r\omega_s}$ is one-dimensional, and we denote the corresponding vertex operators by $X_s$. The top level of $X_1$ and $X_{n-1}$ has conformal weight $\frac{r}{2}$ and in general the one of $X_s$ is the minimum of $\{ \frac{rs}{2},\frac{(n-s)r}{2}\}$. It follows that 
\begin{equation}\label{reg}
X_1(z)X_1(w) \sim 0, \qquad X_{n-1}(z)X_{n-1}(w) \sim 0.
\end{equation} 
By \cite[Prop. 4.1]{CKL} the OPE of $X_s$ and $X_{n-s}$ has a nonzero multiple of the identity as leading term. Without loss of generality, we may rescale $X_1$ and $X_{n-1}$ so that
\begin{equation} \label{norm1} 
X_1(z)X_{n-1}(w) \sim  \prod_{i=1}^{n-1}(i(k+n-1)-1) (z-w)^{-r} + \dots.
\end{equation}
Let $J$ be the Heisenberg field of $V_L$ and we normalize it such that
\begin{equation} \label{norm2} 
J(z)J(w) \sim \bigg(\frac{(n-1)k}{n} +n-2\bigg) (z-w)^{-2}.\end{equation}
Then we have
\begin{equation} \label{norm3} 
J(z)X_1(w) \sim X_1(w)(z-w)^{-1}, \qquad J(z)X_{n-1}(w) \sim - X_{n-1}(w) (z-w)^{-1}.
\end{equation}
\begin{prop}
$A(n, r)$ is generated by $J, X_1, X_{n-1}$ together with generators of $\cY(n, r)$.
\end{prop}
\begin{proof}
As a simple current extension $A(n, r)$ is generated by the fields in $V_{L+\frac{rs}{\sqrt{rn}}} \otimes \mathbb L_{r\omega_s}$ for $s = 0, 1, n-1$ (this is for example a special case of \cite[Main Theorem 1]{CKMII}. Due to Corollary \ref{cor:nodescendants} it is enough to take the top level of $V_{L+\frac{r}{\sqrt{rn}}} \otimes \mathbb L_{r\omega_1}$ and $V_{L+\frac{r(n-1)}{\sqrt{rn}}} \otimes \mathbb L_{r\omega_{n-1}}$. Thus $A(n, r)$ is generated by $X_1, X_{n-1}$ together with generators of $V_L \otimes \cY(n, r)$. Denote the Fock module of weight $\mu$ of the Heisenberg vertex algebra by $F_\mu$. 
This can then be further improved, since $A(n, r)$ is also an infinite order simple current extension
\[
A(n, r) \cong  \bigoplus_{t \in \mathbb Z }\bigoplus_{s=0}^{n-1} F_{\sqrt{nr}t+\frac{rs}{\sqrt{rn}}} \otimes \mathbb L_{r\omega_s}
\]
and so again by \cite[Main Theorem 1]{CKMII} this vertex algebra is generated by $J, X_1, X_{n-1}$ together with  generators of $\cY(n, r)$.
\end{proof}

\begin{thm} \label{thm:reconstructionm=1}
Let $n, r$ be as above and let $k = -r + \frac{n+r}{r-1}$. Then
$$A(n, r) \cong \cW_k(\mathfrak{sl}_r, f_{\text{subreg}}).$$
In particular, we recover the theorem of Arakawa and van Ekeren \cite{AvEII} that $\cW_k(\mathfrak{sl}_r, f_{\text{subreg}})$ is lisse and rational.
\end{thm}
\begin{proof} It follows from \eqref{reg}, \eqref{norm1} and \eqref{norm2}, together with Theorem \ref{thm:uniquenesswnm} and the previous Proposition that $A(n, r)$ has the same strong generating type and OPE algebra as $\cW_k(\mathfrak{sl}_r, f_{\text{subreg}})$, and since $\cW_k(\mathfrak{sl}_r, f_{\text{subreg}})$ is the unique simple graded object of this kind, we get a homomorphism $  A(n, r) \rightarrow \cW_k(\mathfrak{sl}_r, f_{\text{subreg}})$.
This must be an isomorphism since $A(n, r)$ is simple.
\end{proof}

\begin{remark}
Similar results are expected to hold for $m>1$. In that case one however has to deal with vertex algebra extensions that are not of simple current type. Thanks to \cite{CKMII} this situation can be handled provided one can show that subcategories of principal $\cW$-algebras of type $A$ are braid-reversed equivalent to corresponding categories of affine vertex algebras at admissible level, but this is exactly \cite[Thm. 7.1]{CII}.
The latter have been understood in \cite{CHY, CII}. Study of fusion categories of type $A$ is work in progress and those results will allow us to reconstruct $\cW_{\psi}(n, m)$ and $\cV_{\psi}(n, m)$ at those levels where the simple quotient of the coset is a rational principal $\cW$-algebra of type $A$. 
\end{remark}
Let $m\geq 4$.
Set $\lambda := \sqrt{\frac{(m+2)}{2m(m-2)}}$ and consider the extension of $\cY(m-2, 2) =\cW_{2-m +\frac{m-1}{m}}(\gs\gl_{m-2})$ times $L_{-1}(\gs\gl_m)$ times a Heisenberg vertex algebra given by 
\[
B(m, 2) := \bigoplus_{s \in \mathbb Z} \mathbb L_{\omega_{\bar s}} \otimes  W_{-1}(\lambda_s) \otimes F_{s\lambda}.
\]
Here $\bar s = s \mod m-2$ and the notation of $L_{-1}(\gs\gl_m)$-modules is taken from \cite[Section 5.2]{CY} where it was shown that these modules form a vertex tensor category of simple currents, see also \cite{AP}, that is $ W_{-1}(\lambda_s) \boxtimes  W_{-1}(\lambda_t) \cong  W_{-1}(\lambda_{s+t})$. Moreover the top level of $ W_{-1}(\lambda_s)$  is $\rho_{s\omega_1}$ of $s$ is non-negative and  $\rho_{-s\omega_{m-1}}$ otherwise. 
Let $J$ be the Heisenberg field and we normalize it to have norm $\lambda^{-2}$. Let $Y_s$ be the field corresponding to the top level of $L_{\omega_{\bar s}} \otimes  W_{-1}(\lambda_s) \otimes F_{s\lambda}$. It follows that 
\[
J(z)X_1(w) \sim \frac{X_1(w)}{(z-w)}, \qquad J(z)X_{-1}(w) \sim -\frac{X_{-1}(w)}{(z-w)}.
\]
Moreover the top level of $X_{\pm 1}$ is computed to be $\frac{3}{2}$ and conformal weight of the top level of $X_{\pm 2}$ ensures that the operator product of $X_1$ with $X_{-1}$ is regular. Thus the same proof as the one of Theorem \ref{thm:reconstructionm=1} applies, and we obtain
\begin{thm} \label{thm:reconstructionm=2}
For $m\geq 4$,  $B(m, 2) \cong \cW_m(m, 2)$. In particular \cite[Conj. 4. 3. 2]{CY} for $r=0$ is true, and hence the category of ordinary modules of $\gs\gl_{m+2}$ at level minus two is a vertex tensor category by \cite[Cor. 4.3.3]{CY}.
\end{thm}
Let $\psi = \frac{n+1}{n}$ and consider $C_\psi(n-1, 1)$ which has $c=-2$ and so it is just the $\cW_{-3+\frac{3}{2}}(\gs\gl_3)$ algebra. 
Consider $\cS(1)^{\mathbb Z/n\mathbb Z}$ which is easily checked to be strongly generated by  $:\beta^{n}:, :\gamma^{n}:, \frac{:\beta\gamma:}{n}$ together with the two strong generators of  $\cW_{-3+\frac{3}{2}}(\gs\gl_3)$. It follows that $\cS(1)^{\mathbb Z/n\mathbb Z}$ has the same strong generating type and OPE algebra as $\cW_{\psi -n}(\mathfrak{sl}_{n}, f_{\text{subreg}})$ and hence we can again conclude that 
\begin{thm} \label{thm:reconstructionm=3}
For $\psi = \frac{n+1}{n}$, $\cW_{\psi -n}(\mathfrak{sl}_{n}, f_{\text{subreg}})\cong \cS(1)^{\mathbb Z/n\mathbb Z}$.
\end{thm}

Finally, we have a uniqueness theorem for hook-type $\cW$-superalgebras.
\begin{thm}\label{thm:uniquenesssvnm} For $m\geq 1$ and $n\geq 1$, the full OPE algebra of $\cV^{\psi}(n,m)$ is determined completely from the structure of $\cW_{J_{n,m}}(c,\lambda)$, the normalization of $J$, the action of $\gg\gl_m$ on the generators $\{P^{\pm, i}\}$, and the nondegeneracy condition $$P^{+,i}_{(n)} P^{-,j} = \delta_{i,j} 1.$$ 

In particular, for $m\geq 2$ and $m\neq n$, if the generator $J$ of $\cH$ is normalized as in Lemma \ref{lem:vnmj}, and $\cA^{\psi}(n,m)$ is a one-parameter vertex algebra which extends $\cH \otimes V^{-\psi-m+1}(\gs\gl_m) \otimes \cW_{J_{n,m}}(c,\lambda)$ by odd fields $\{P^{\pm, i}|\ i = 1,\dots, m\}$ of conformal weight $\frac{n+1}{2}$ which are primary with respect to $\cH \otimes V^{-\psi-m+1}(\gs\gl_m)$ as well as the total Virasoro field $T = L + L^{\gs\gl_m} + L^{\cH}$, then $\cA^{\psi}(n,m) \cong \cW^{\psi}(n,m)$. 

Similarly, if $m = 1$ and $n>1$, $J$ is normalized as in Lemma \ref{lem:vnmj}, and $\cA^{\psi}(n,1)$ is a one-parameter vertex algebra which extends $\cH \otimes \cW_{J_{n,1}}(c,\lambda)$ by odd fields $\{P^{\pm}\}$ of conformal weight $\frac{n+1}{2}$ which are primary with respect to $\cH$ as well as the total Virasoro field $T = L + L^{\cH}$, then $\cA^{\psi}(n,1)\cong \cW^{\psi}(n,1)$. 

Finally, if $m =n$ and $n\geq 1$, and $\cA^{\psi}(n,n)$ is a one-parameter vertex algebra which extends $V^{-\psi-n+1}(\gs\gl_n)\otimes \cW_{J_{n,n}}(c,\lambda)$ by odd fields $\{P^{\pm, i}|\ i = 1,\dots, n\}$ of conformal weight $\frac{n+1}{2}$ which are primary with respect to both $V^{-\psi-n+1}(\gs\gl_n)$ and the total Virasoro field $T = L + L^{\gs\gl_m}$, then $\cA^{\psi}(n,n) \cong \cW^{\psi}(n,n)$. 
\end{thm}

The proof is omitted since it is the same as the proof of Theorem \ref{thm:uniquenesswnm}. It is now easy to prove the analogues of Theorems \ref{thm:reconstructionm=1}, \ref{thm:reconstructionm=2} and \ref{thm:reconstructionm=3} for principal $\cW$-superalgebras of $\gs\gl_{n|1}$, since this $\cW$-superalgebra coincides with $\cV^{\psi}(n,1)$ and can be realized as the Heisenberg coset of a subregular $\cW$-algebra times a pair of free fermions \cite{CGN}.

\section{Another perspective on triality}\label{sec:kernel}

Let
\[
A[\gs\gl_N, \psi] := \bigoplus_{\lambda \in P^+} V^k(\lambda) \otimes V^\ell(\lambda) \otimes V_{\sqrt{N}\mathbb Z + \frac{s(\lambda)}{\sqrt{N}}}
\]
with $\psi= k+N, \psi'=\ell+N$ and 
$\frac{1}{\psi} +\frac{1}{\psi'} =1$.
The map $s : P^+ \rightarrow \mathbb Z/ N\mathbb Z$ is defined by $s(\lambda) = t \quad \text{if} \ \lambda = \omega_t \mod Q$, where $\omega_t$ is the $t$-th fundamental weight of $\gs\gl_N$ and we identify $\omega_0$ with $0$. 
The $V^k(\lambda)$ are generalized Verma modules at level $k$ whose top level is the integrable $\gs\gl_N$-module $\rho_\lambda$ of highest-weight $\lambda$. 
Let $f$ be a nilpotent element with corresponding complex $C_f$, i.e. the homology $H_f(V^k(\gg) \otimes C_f)$ is the $\cW$-algebra $\cW^k(\gg, f)$. We then denote the $\cW^k(\gg, f)$-module $H_f(M \otimes C_f)$ simply by  $H_f(M)$ for $M$ a $V^k(\gg)$-module. One then sets
\[
A[\gs\gl_N, f, \psi] := \bigoplus_{\lambda \in P^+} V^k(\lambda) \otimes H_f(V^\ell(\lambda)) \otimes V_{\sqrt{N}\mathbb Z + \frac{s(\lambda)}{\sqrt{N}}}
\]
and conjectures that
\begin{conj}\label{conj;simple} \cite{CG} With the above notation and
 for generic $k$ and any nilpotent element $f$, the object $A[\gs\gl_N, f, \psi]$ can be given the structure of a simple vertex superalgebra, such that the top level of $V^k(\lambda) \otimes H_f(V^\ell(\lambda)) \otimes V_{\sqrt{N}\mathbb Z + \frac{s(\lambda)}{\sqrt{N}}}$ is odd for $\lambda = \omega_1, \omega_{N-1}$.
\end{conj}
For $f$ the principal nilpotent, this is just $A[\gs\gl_N, f, \psi] \cong V^{k-1}(\gs\gl_N) \otimes \cF(2N)$ by the coset construction of principal $\cW$-algebras of \cite{ACLII}. Here $\cF(2N)$ is the vertex superalgebra of $2N$ free fermions.

Set $N=n+m$ and consider the nilpotent element $f=f_{n, m}$  corresponding to the partition $N=n+1+ \dots +1$ so that $\cW^\ell(\gs\gl_N, f)= \cW^{\ell+N}(n, m)$
is a hook-type $\cW$-algebra with $V^{\ell+n-1}(\gg\gl_m)$ as subalgebra. 
The top level corresponding to the standard representation of $\gs\gl_N$ in $A[\gs\gl_N, f, \psi] $ has conformal weight $\frac{N}{2} -\frac{n-1}{2}= \frac{m+1}{2}$, and it is expected to be odd. We want to take a coset that contains these elements. For this let $J$ be as in Lemma \ref{lem:vnmj} and let $\gamma$ be the generator of $\sqrt{N}Z=\gamma\mathbb Z$, i.e. $\gamma^2=N$.  Denote the corresponding Heisenberg field by $\gamma$ as well and set $H=J-\gamma$, and $\cH$ the Heisenberg vertex algebra generated by $H$. This ensures that the commutant with $V^{\ell+n-1}(\gs\gl_m) \otimes \cH$ contains the fields of conformal weight $\frac{m+1}{2}$ in the standard representation of $\gs\gl_N$, and its conjugate. Moreover, if the generator $H$ of $\cH$ is normalized as in Lemma \ref{lem:wnmj}, then these fields have Heisenberg weight $\pm 1$.

\begin{thm}\label{thm: Sduality} With the above notation and
 for generic $k$, if Conjecture \ref{conj;simple} is true for $f= f_{n, m}$, then \
$\text{Com}\left(V^{\ell+n-1}(\gs\gl_m) \otimes \cH, A[\gs\gl_N, f_{n, m}, \psi]\right) \cong \cW^{-k-m+1}(\gs\gl_{m|N}, f_{m|N})$.
\end{thm}
\begin{proof} 
Set $\cB^\psi(n, m) := \text{Com}\left(V^{\ell+n-1}(\gs\gl_m) \otimes \cH, A[\gs\gl_N, f_{n, m}, \psi]\right)$. At generic level the category $KL_{k}(\gs\gl_N)$ is semisimple and so 
\[
\cB^\psi(n, m) = \bigoplus_{\lambda \in S} V^k(\lambda) \otimes D_\lambda
\]
for certain modules $D_\lambda$ of $\cC^{\psi'}(n, m)$. 
Here the sum is over a set $S$ of weights of $\gg\gl_N$.
By Theorem \ref{intro:mainthm}, we also have 
\[
\cW^{-k-m+1}(\gs\gl_{m|N}, f_{m|N}) = \bigoplus_{\lambda \in S' } V^k(\lambda) \otimes C_\lambda
\]
for certain modules $C_\lambda$ of $\cC^{\psi'}(n, m)$ and $S'$ is also a set of weights of $\gg\gl_N$. In fact the $C_\lambda$ are simple for generic level by Theorem \ref{thm:completereducibility}.
The main step is to prove that the theorem holds on the level of graded characters
$\text{ch}\left[\cB^\psi(n, m)  \right] = \text{ch}\left[ \cW^{-k-m+1}(\gs\gl_{m|N}, f_{m|N}) \right]$.
This is a  meromorphic Jacobi form argument that is deferred to the appendix, see Theorem \ref{thm:chardecomp}. 
The coset $\cB^\psi(n, m)$ is simple since at generic level the category $KL_{\ell+n-1}(\gs\gl_m)$ is semisimple and so \cite[Prop. 5.4]{CGN} applies if Conjecture \ref{conj;simple} is true for $f= f_{n, m}$.
We thus have two simple vertex superalgebras that have the same graded character. Especially (and as noted before) $\cB^\psi(n, m)$ has odd fields of conformal weight $\frac{m+1}{2}$ in the standard representation of $\gs\gl_N$, and its conjugate. Recall that if the generator $H$ of $\cH$ is normalized as in Lemma \ref{lem:wnmj}, then these fields have Heisenberg weight $\pm 1$. Theorem \ref{thm:uniquenesssvnm} applies to the vertex superalgebra generated by 
these fields together with the subalgebra $V^k(\gg\gl_N) \otimes \cC^{\psi'}(n, m)$, and hence it must be isomorphic to $\cW^{-k-m+1}(\gs\gl_{m|N}, f_{m|N})$.
Finally, since graded characters of $\cW^{-k-m+1}(\gs\gl_{m|N}, f_{m|N})$ and $\cB^\psi(n, m)$ coincide, this subalgebra must already be the complete coset $\cB^\psi(n, m)$.
\end{proof}
We recall relative semi-infinite Lie algebra cohomology \cite{FGZ}, and for this we use Section 2.5 of \cite{CFL}.
Let $\gg$ be a simple Lie algebra with basis $B$ and dual basis $B'$. Consider free fermions $\mathcal F(\gg)$ in two copies of the adjoint representation of $\gg$ with generators $\{ b^x, c^{x'} \,  | \, x\in B, \, x'\in B' \}$ and operator products
$b^x(z) c^{y'}(w) \sim \delta_{x, y} (z-w)^{-1}$.
Consider  $V^{-2h^\vee}(\gg)$ and let $x(z)$ be the field corresponding to $x\in \gg$. The zero mode $d:=d_0$ of the field
\[
d(z) := \sum_{x\in \mathcal B} :x(z) c^{x'}(z): - \frac{1}{2} \sum_{x, y \in \mathcal B} :(:b^{[x, y]}(z) c^{x'}(z):)c^{y'}(z):
\] 
 squares to zero. Let $\widetilde{\mathcal F}(\gg)$ be the subalgebra of $\mathcal F(\gg)$ generated by the $b^x$ and $\partial c^{x'}$. 
Let $M$ be a module for $V^{-2h^\vee}(\gg)$.
The relative complex is
\[
C^{\text{rel}}(\gg, d) =  \left(M\otimes \widetilde{\mathcal F}(\gg)\right)^\gg
\]
and it is preserved by $d$  \cite[Prop. 1.4.]{FGZ}. The corresponding cohomology is denoted by $H^{\text{rel},\bullet }_{\infty}(\gg, M)$. 
We need the following property that follows from \cite{FGZ}, as explained in Section 2.5 of \cite{CFL}:
\begin{equation}\label{eq:relcoho}
H^{\text{rel}, 0 }_{\infty}(\gg, V^k(\lambda) \otimes V^{-2h^\vee-k}(\mu)) = \begin{cases} \mathbb C & \ \text{if} \ \mu=-\omega_0(\lambda) \\ 0 & \ \text{otherwise} .\end{cases}
\end{equation}
Here $\omega_0$ is the unique Weyl group element that interchanges the fundamental Weyl chamber with its negative.

 Let $\psi \in \mathbb C$ be generic, and fix $n\geq m \in \mathbb Z_{\geq 0}$. Set $\ell = -\frac{m}{n}(n-(n-m)\psi)$ and $k = \frac{m}{m-n}(n\psi^{-1} +m-n)$. 
 Consider $\cW^\psi(n-m, m)$ and $\cV^{\psi^{-1}}(n, m)$ so that by our main theorem their cosets are isomorphic. We aim to relate these two algebras using the relative semi-infinite Lie algebra cohomology. 
Note that we normalize the Heisenberg fields of $\cW^\psi(n-m, m)$ and $\cV^{\psi^{-1}}(n, m)$ in such a way that they have norm $\ell$ and $k$. Consider 
$A[\gs\gl_m, 1-\psi]  \otimes \pi^{k-\ell}$ with $\pi^{k-\ell}$ a rank one Heisenberg vertex algebra generated by $X(z)$ and normalized such that it has level $k-\ell$. 
Let $Y$ be the Heisenberg field of the $V_{\sqrt{m}\mathbb Z}$ subalgebra of $A[\gs\gl_m, 1-\psi]$ and we normalize it to have level $m$. Define $J^{-\ell}, J^k$ by $J^{-\ell}-J^k =X$ and $\frac{J^k}{k} -\frac{J^{-\ell}}{\ell} =Y$, so that $J^k$ has level $k$ and $J^{-\ell}$ has level $-\ell$. 
$\cW^\psi(n-m, m)$ has an action of $V^{\psi -m - 1}(\gs\gl_m) \otimes \pi^\ell$ and $A[\gs\gl_m, 1-\psi]  \otimes \pi^{k-\ell}$ has an action of $V^{-\psi -m + 1}(\gs\gl_m) \otimes \pi^{-\ell}$, so that the diagonal action has level $-2h^\vee$ and we can take relative semi-infinite Lie algebra cohomology. We conjecture 
\begin{conj}\label{conj:cohom} With the above set-up, 
$H^{\text{rel}, 0 }_{\infty}(\gs\gl_m, \cW^\psi(n-m, m) \otimes A[\gs\gl_m, 1-\psi]  \otimes \pi^{k-\ell})$ is a simple vertex superalgebra. 
\end{conj} 
\begin{thm}\label{thm:relcoho}
Let $k$ be generic and assume that Conjecture \ref{conj:cohom} is true. Then 
$\cV^{\psi^{-1}}(n, m)\cong H^{\text{rel}, 0 }_{\infty}(\gs\gl_m, \cW^\psi(n-m, m) \otimes A[\gs\gl_m, 1-\psi]  \otimes \pi^{k-\ell})$.
\end{thm}
\begin{proof}
Recall that 
\[
\cW^\psi(n-m, m) \cong \bigoplus_{\lambda \in R^+ } V^{\psi - m -1}(\lambda) \otimes C^\psi(\lambda)
\]
and 
\[
\cV^{\psi^{-1}}(n, m) \cong \bigoplus_{\lambda \in R^+ } V^{-\psi^{-1} - m +1}(\lambda) \otimes D^{\psi^{-1}}(\lambda)
\]
for certain nonzero simple $\cC^\psi(n-m, n)$-modules $ C^\psi(\lambda)$ and $ D^{\psi^{-1}}(\lambda)$. The set $R^+$ is determined in Remark \ref{rem:int-set} and is $R^+ = \{ (\lambda, n)  | \lambda \in P^+ , r \in \mathbb Z, \lambda = \omega_{i(r)} \mod A_{m-1}, i(r) \in [0, m-1], i(r) = r \mod m  \}$. Here $P^+$ denotes the set of dominant weights of $\gs\gl_m$ as usual. 
On the other hand by \eqref{eq:relcoho} we immediately get that 
\begin{equation*} \begin{split} & 
H^{\text{rel}, 0 }_{\infty}(\gs\gl_m, \cW^\psi(n-m, m) \otimes A[\gs\gl_m, 1-\psi]  \otimes \pi^{k-\ell}) 
\\ & =  \bigoplus_{\lambda \in R^+ } V^{-\psi^{-1} - m +1}(\lambda) \otimes C^\psi(-\omega_0(\lambda)).\end{split} \end{equation*}
One computes that the top level of $ V^{-\psi^{-1} - m +1}(\lambda) \otimes C^\psi(-\omega_0(\lambda))$ is $\frac{n+1}{2}$ for $\lambda = \omega_1, \omega_{m-1}$ and further conformal weight computations ensure that  Theorem \ref{thm:uniquenesssvnm} applies to the vertex superalgebra generated by 
these fields together with the subalgebra $V^{-\psi^{-1} - m +1} \otimes \cC^\psi(n-m, n)$ and is thus isomorphic to $\cV^{\psi^{-1}}(n, m)$. This already must be the complete algebra $H^{\text{rel}, 0 }_{\infty}(\gs\gl_m, \cW^\psi(n-m, m) \otimes A[\gs\gl_m, 1-\psi]  \otimes \pi^{k-\ell})$ since both are direct sums over the same set $R^+$ with summands being $V^{\psi - m -1}(\lambda)$ times a simple $\cC^\psi(n-m, n)$-module. 
\end{proof}

\appendix

\section{Decomposing characters} \label{section:Sduality}

Consider the objects
\[
A[\gs\gl_N, f, k] := \bigoplus_{\lambda \in P^+} V^k(\lambda) \otimes H_{f}(V^\ell(\lambda)) \otimes V_{\sqrt{N}\mathbb Z + \frac{s(\lambda)}{\sqrt{N}}}
\]
with $\ell$ and $s : P^+ \rightarrow \mathbb Z/ N\mathbb Z$ defined by 
\[
\frac{1}{k+N}+\frac{1}{\ell+N} =1, \qquad  s(\lambda) = t \quad \text{if} \ \lambda = \omega_t \mod Q
\]
where $\omega_t$ is the $t$-th fundamental weight of $\gs\gl_N$ and we identify $\omega_0$ with $0$. Let $J$ be as in Lemma \ref{lem:vnmj} and let $\gamma$ be the generator of $\sqrt{N}Z=\gamma\mathbb Z$, i.e. $\gamma^2=N$.  Denote the corresponding Heisenberg field by $\gamma$ as well, and set $H=J-\gamma$ and $\cH$ the Heisenberg vertex algebra generated by $H$. 

\begin{thm}\label{thm:chardecomp}
Conjecture \ref{conj:S-dualityintro} (2) holds on the level of characters, that is
\[
\text{ch}\left[\text{Com}\left(V^{\ell+n-1}(\gs\gl_m) \otimes \cH, A[\gs\gl_N, f_{n, m}, k]\right)  \right] = \text{ch}\left[ \cW^{-k-m+1}(\gs\gl_{m|N}, f_{m|N}) \right].
\]
\end{thm}
\begin{proof}
Set $f = f_{n,m}$ and $C=\text{ch}\left(\text{Com}\left(V^{\ell+n-1}(\gs\gl_m) \otimes \cH, A[\gs\gl_N, f, k]\right)  \right)$. 
The proof has several steps.
\begin{enumerate}
\item Introduce convenient notations.
\item Give an explicit expression for the character of $H_f(V^k(\lambda))$.
\item Show that the characters of $A[\gs\gl_N, 0, k]$ and hence $A[\gs\gl_N, f, k]$ are quotients of certain Jacobi forms.
\item Use the denominator identity of $\widehat{\gs\gl}_{2|1}$ to decompose these meromorphic Jacobi forms. $C$ is a certain coefficient of this decomposition.
\item Use the denominator identity of $\gs\gl_{N|m}$ to write $C$ as an infinite product and identify it with $ \text{ch}\left[ \cW^{-k-m+1}(\gs\gl_{m|N}, f_{m|N}) \right]$.
\end{enumerate}

\noindent {\em Step 1: Notations}

The combinatorics of the proof are slightly different depending on  $N, n, m$ being even or odd and can be combined into a uniform proof by setting
\[
a_M := \begin{cases}  0 & M \ \text{odd}, \\ \frac{1}{2} & M \ \text{even}, \end{cases}
\]
for any integer $M$.
We denote the root lattice of $\gs\gl_N$ by $Q$ and embed it in $\mathbb Z^N$ in the standard way 
$$Q= \mathbb Z(\delta_1-\delta_2) \oplus \dots \oplus \mathbb Z(\delta_{N-1}-\delta_N) \subseteq \mathbb Z^N = \mathbb Z \delta_1 \oplus \dots \oplus \mathbb Z\delta_N,$$
 where the $\delta_i$ form an orthonormal basis of $\mathbb Z^N$. We choose as a set of positive roots the set
\begin{equation}
\begin{split}
\Delta_+ &= \left\{ \delta_i - \delta_j |  1\leq i < j \leq  m\right\} \cup \left\{\delta_a -\delta_b | m+1 \leq  a < b \leq  N \right\} \ \cup\\
&\quad   \left\{ \delta_i - \delta_a | 1 \leq i \leq m,  m+1+ \left\lceil \frac{n}{2}\right\rceil\leq a \leq N \right\}\ \cup\\&\quad \left\{\delta_a -\delta_i | 1 \leq i \leq m, m+1 \leq  a \leq  m+ \left\lceil\frac{n}{2}\right\rceil\right\},
\end{split}
\end{equation}
with the first two sets the positive roots $\Delta_+^m$ and $\Delta_+^n$ of the $\gs\gl_m$ and $\gs\gl_n$ subalgebras. Note that the Weyl vector $\rho_N$ of $\gs\gl_N$ then decomposes as
\begin{equation}\nonumber
\begin{split}
\rho_N &= \rho_m+\rho_n + \sigma, \qquad\\ \sigma &= \frac{m}{2}\left(\delta_{m+1} + \dots + \delta_{m+\lceil \frac{n}{2} \rceil }- \delta_{m+ \lceil \frac{n}{2} \rceil + 1} - \dots - \delta_N\right) 
\\ & \quad + \left(a_n -\frac{1}{2}\right)\left( \delta_1 + \dots + \delta_m \right),
\end{split}
\end{equation}
with $\rho_m$ and $\rho_n$ the Weyl vectors of the $\gs\gl_m$ and $\gs\gl_n$ subalgebras. Set
\[
\widetilde \Delta_+ = \left\{ \delta_{m + \left\lceil\frac{n}{2}\right\rceil} - \delta_i | 1\leq i \leq m\right\}.
\]
These are the positive roots that are needed for the contribution of the dimension $\frac{n}{2}$ fields in the character formula.

Let $\sqrt{N}\mathbb Z = \gamma\mathbb Z$ with $\gamma^2=N$. 
Let $\mathfrak h^\sharp$ be the subalgebra of the Cartan subalgebra corresponding to the $\Delta_+^m \cup\{\delta\}$, i.e. it is orthogonal to the Cartan subalgebra of the $\gs\gl_n$-subalgebra.

We choose the basis $\mathcal B$ of $\mathfrak h$ to be 
\begin{equation}
\begin{split}
\mathcal B &=\left\{e_{i, i+1} | 1\leq  i \leq  m-1\right\} \cup \left\{ e \right\}  \cup \left\{e_{i, i+1} | m+1\leq  i \leq  N-1\right\}, \\ e_{i, j} & :=e_i-e_j,\qquad 
\delta_i(e_j) = \begin{cases}  1 & i=j \\ 0 & i \neq  j \end{cases}, 
\\  e & :=  \frac{n(e_1 + \dots + e_m) - m(e_{m+1}+\dots + e_N )}{N}.
\end{split}
\end{equation}

\noindent {\em Step 2: the Euler-Poincar\'e character}

Set $h= \sum_{i=1}^N u_ie_i$ and define $x_i=e^{u_i}= e^{\delta_i(h)}$ and $x=e^{\delta(h)}$ with $\delta$ dual to $e$, i.e. $\delta(e)=1$ and $\delta(e_{i, i+1})=0$ for $i\neq m$. Then
$\rho_N(e) = \sigma(e) = m \left(a_n -\frac{1}{2} \right)$. The element 
$x$ of the $\gs\gl_2$-triple for the reduction is 
\[
x = \sum_{s=1}^{n-1} (n+1-2s)e_{m+s}.
\]
We now specialize to $u_i=0$ if $m+1\leq i \leq N$. Note that for $1\leq i \leq m$,
$(\delta_{m + \left\lceil\frac{n}{2}\right\rceil} - \delta_i)(e_j)= -\delta_{i, j}$, $(\delta_{m + \left\lceil\frac{n}{2}\right\rceil} - \delta_i)(e)=-1$.
The Euler-Poincar\'e character \eqref{eq:EP} is then
\begin{equation}
\begin{split}
& \text{ch}[H_f(V^k(\lambda))](q, h) = q^{\Delta}
 q^{\frac{(\lambda+\rho)^2}{2(k+h^\vee)}}  {\sum_{\omega \in W} \epsilon(\omega) e^{\omega(\lambda +\rho)(h-\tau x)}} \frac{1}{\Psi(q, h)} \\ &
\Psi(q, h) = q^{\frac{\text{dim} \gg}{24}} e^{\rho(h)} \prod_{r=1}^\infty (1-q^r)^{N-1} \prod_{\alpha \in \Delta_+^m} (1- e^{-\alpha(h)} q^{r-1}) (1- e^{\alpha(h)} q^{r})\\
&   \prod_{\beta\in \widetilde\Delta_+} (1- e^{-\beta(h)}q^{r+a_n-1})(1-e^{\beta(h)} q^{r-a_n}).
\end{split}
\end{equation}
Here $\Delta$ is some fixed constant whose precise value is not important.
The domain is $|q| <1$ and $|e^{\pm \beta(h)}q^{(n+1)/2}| < 1, |e^{\pm \alpha(h)}q| < 1$ for $\beta \in \tilde\Delta_+, \alpha \in \Delta_+^m$. 
 Note that $e^{\rho(h)} = e^{\rho_m(h)}$ if $n$ is even and   $e^{\rho(h)} = e^{\rho_m(h)}e^{-um/2}$ if $n$ is odd.
Note that $x_1 \dots x_m=1$. 
Then
\begin{equation}
\begin{split}
&\Psi(q, h) = 
\\ & q^{\frac{\text{dim} \gg}{24}} e^{\rho_m(h)} x^{-\frac{m}{2}+ma}\prod_{r=1}^\infty (1-q^r)^{N-1} \prod_{1 \leq i < j \leq m} (1- x_i^{-1}x_j q^{r-1}) (1- x_ix_j^{-1} q^{r})\\
&\quad (-x)^{md'} q^{\Delta_d}\prod_{1 \leq i \leq m} (1- xx_iq^{r+a_n+d'-1})(1- x^{-1}x_i^{-1} q^{r-a_n-d'})\\
&= q^{\frac{\text{dim} \gg}{24}} e^{\rho_m(h)} \prod_{r=1}^\infty (1-q^r)^{N-1} \prod_{1 \leq i < j \leq m} (1- x_i^{-1}x_j q^{r-1}) (1- x_ix_j^{-1} q^{r})\\
&\quad (-1)^{md'} x^{m(d'-\frac{1}{2}+a_n)} q^{\Delta_{d'}}\prod_{1 \leq i \leq m} (1- xx_iq^{r+a_n+d'-1})(1- x^{-1}x_i^{-1} q^{r-a_n-d'})\\
&= q^{\frac{\text{dim} \gg}{24}} e^{\rho_m(h)} \prod_{r=1}^\infty (1-q^r)^{N-1} \prod_{1 \leq i < j \leq m} (1- x_i^{-1}x_j q^{r-1}) (1- x_ix_j^{-1} q^{r})\\
&\quad (-1)^{md'} x^{md} q^{\Delta_{d'}}\prod_{1 \leq i \leq m} (1- xx_iq^{r+d-1/2})(1- x^{-1}x_i^{-1} q^{r-d-1/2}).
\end{split}
\end{equation}
Here we did a shift of the form $xx_i\mapsto xx_iq^{d'}$ which gave the prefactor and $\Delta_{d'}$ is some number depending on $d'$ and $a_n$, whose precise value is not important. We will fix the integer $d'$ later and we set $d= d'-\frac{1}{2}+a_n$ for convenience.

\noindent  {\em Step 3: The character of $A[\gs\gl_N, 0, k]$}

We are interested in the character of $A[\gs\gl_N, 0, k]$ which we abbreviate by $A$, i.e.
\[
A := \sum_{\lambda \in P^+} \text{ch}[V^k(\lambda)](q, h_1)\text{ch}[V^\ell (\lambda)](q, h_2) \frac{\theta_{\sqrt{N}\mathbb Z+\frac{s(\lambda)}{\sqrt{N}}}(q; x)} {\eta(q)}.
\]
Recall the character of $V^k(\lambda)$ given in \eqref{eq:chVk}.
Recall that $\ell$ satisfies 
$
\frac{1}{k+h^\vee} + \frac{1}{\ell+h^\vee} = 1,
$
so that 
\begin{equation}\nonumber
\begin{split}
\text{ch}[V^k(\lambda)](q, h_1)\text{ch}[V^\ell (\lambda)](q, h_2) &=  q^{\frac{(\lambda+\rho)^2}{2(k+h^\vee)}}q^{\frac{(\lambda+\rho)^2}{2(\ell+h^\vee)}}  \frac{N_\lambda(h_1)N_\lambda(h_2)}{\Pi(q, h_1 )\Pi(q, h_2 )}
\\ &  = q^{\frac{(\lambda+\rho)^2}{2}} \frac{N_\lambda(h_1)N_\lambda(h_2)}{\Pi(q, h_1 )\Pi(q, h_2 )}. 
\end{split}
\end{equation}
Set $P^+_t := P^+ \cap (Q+\omega_t)$. Then 
\begin{equation}\nonumber
\begin{split}
& \sum_{\lambda \in P^+_t}  q^{\frac{(\lambda+\rho)^2}{2}}N_\lambda(h_1)N_\lambda(h_2)
\\ &= \sum_{\lambda \in P^+_t}\sum_{\omega_1, \omega_2 \in W}  q^{\frac{(\lambda+\rho)^2}{2}} \epsilon(\omega_1)\epsilon(\omega_2) e^{\omega_1(\lambda+\rho)(h_1)}e^{\omega_2(\lambda+\rho)(h_2)}\\
&= \sum_{\lambda \in P^+_t}\sum_{\omega_1, \omega_2 \in W}  q^{\frac{(\lambda+\rho)^2}{2}} \epsilon(\omega_1)\epsilon(\omega_2 \circ \omega_1) e^{\omega_1(\lambda+\rho)(h_1)}e^{\omega_2(\omega_1(\lambda+\rho))(h_2)} \\
&= \sum_{\omega_2 \in W} \epsilon(\omega_2)  \sum_{\lambda \in P^+_t}\sum_{\omega_1 \in W}  q^{\frac{(\lambda+\rho)^2}{2}} e^{\omega_1(\lambda+\rho)(h_1)}e^{\omega_2(\omega_1(\lambda+\rho))(h_2)} \\
&= \sum_{\omega \in W} \epsilon(\omega)  \sum_{\lambda \in Q + \omega_t  +\rho}  q^{\frac{\lambda^2}{2}} e^{\lambda(h_1)}e^{\omega(\lambda)(h_2)}. \\
\end{split}
\end{equation}
In the third equality we used that  $ \epsilon(\omega_1)\epsilon(\omega_2 \circ \omega_1) = \epsilon(\omega_2)$, and in the last one that the Weyl group acts transitively on Weyl chambers, together with the fact that $\sum_{\omega \in W} \epsilon(\omega) e^{\omega(\lambda)} =0$ for any $\lambda$ that is orthogonal to at least one simple root. 

For $h_1, h_2$ in $\mathfrak h$, we set  $s_i = \omega_1 - \delta_1 +\delta_i$ and 
 $y_i =e^{s_i(h_1)}, z_i= e^{s_i(h_2)}$, so that the character of the standard representation of $\gs\gl_N$ is just 
 \[
 \chi_{\omega_1}(h_1) = y_1 + \dots  + y_N.
 \]
 The relation to our previous Jacobi variables is
 \[
 y_i = \begin{cases} x_i x^{n/N} & \qquad 1 \leq i \leq m, \\ x_ix^{-m/N} & \quad  m+1 \leq i \leq  N, \end{cases}
 \]
 and we write $(yz_\omega)_\nu$ for the set $\{y_1 z_{\omega(1)}q^{\nu \epsilon_1}, \dots, q^{\nu \epsilon_N}y_Nz_{\omega(N)}\}$ for any $\nu \in P$.
  Note that $\rho \in Q$ for $N$ odd and $\rho \in Q + \omega_{N/2}$ for $N$ even. We thus set $\nu_N =0$ if $N$ is odd and $\nu_N =  \omega_{N/2}$ for $N$ even. Let $\nu \in \nu_N +Q$. It follows that we get theta functions
\begin{equation}
\begin{split}
  \sum_{\lambda \in Q + \omega_t  +\rho}  q^{\frac{\lambda^2}{2}} e^{\lambda(h_1)}e^{\omega(\lambda)(h_2)} &= \theta_{Q + \omega_t +\rho }(q, (yz_\omega)_0)
 \\ & =  q^{\frac{\nu^2}{2}} e^{\nu(h_1)}e^{\omega(\nu)(h_2)} \theta_{Q + \omega_t}(q, (yz_{\omega})_{\nu}),
\end{split}
\end{equation}
where we used the usual translation property of Jacobi theta functions. Recall that  $a_N = 0$ if $N$ is odd and $a_N=\frac{1}{2}$ if $N$ is even.  
Let $g'$ be a half integer to be fixed later and set $g=g'+a_N$. Let $u$ in $(\gamma \mathbb Z)^*$ and $w=e^{\gamma(u)/N}$. 
Since $$\mathbb Z^N = \bigcup_{t=0}^{N-1} (Q +\omega_t) \oplus \left( \gamma\mathbb Z + \frac{t\gamma}{N}\right)= \bigcup_{t=0}^{N-1} (Q +\omega_t+\rho) \oplus \left( \gamma\mathbb Z + ga_N\gamma + \frac{t\gamma}{N}\right),
$$ and using that
\[
\theta_{\gamma\mathbb Z +\frac{t\gamma}{{N}}}(q, w) = q^{a_N^2\frac{N}{2}} w^{Na_N} \theta_{\gamma\mathbb Z +ga_N\gamma+ \frac{t\gamma}{{N}}}(q, wq^{a_N}), 
\]
the numerator of $A$, that is $\text{Num} :=A\Pi(q, h_1)\Pi(q, h_2)\eta(q)$ is of the form 
\begin{equation}
\begin{split}
\text{Num}  &= 
 \sum_{t=0}^{N-1} \theta_{Q + \omega_t +\rho}(q, (yz_{\omega})) \theta_{\gamma\mathbb Z +\frac{t\gamma}{{N}}}(q, w)\\
&= q^{a_N^2\frac{N}{2}} w^{Na_N}  \sum_{t=0}^{N-1} \theta_{Q + \omega_t +\rho}(q, (yz_{\omega})) \theta_{\gamma\mathbb Z +a_Ng\gamma+ \frac{t\gamma}{{N}}}(q, wq^{a_N})\\
&= q^{a_N^2\frac{N}{2}} w^{Na_N} \theta_{\mathbb Z^N}(q, (yz_\omega)_{\nu},  wq^{g}) \\
 &=  q^{a_N^2\frac{N}{2}} w^{Na_N} 
 \\ & \quad \prod_{i=1}^N \prod_{r=1}^\infty (1+wy_iz_{\omega(i)}q^{r-1/2 +g}) (1+w^{-1}y_i^{-1}z_{\omega(i)}^{-1} q^{r-1/2-g})(1-q^r) \\
 &= q^{a_N^2\frac{N}{2}} w^{Na_N}  \prod_{i=1}^N  \vartheta_{3}(wy_iz_{\omega(i)}q^{g} ; q) \\
 &\sim  w^{Ng-na_n} \prod_{a=1}^n(y_{m+a}z_{\omega(m+a)})^{-a_n} e^{\rho_{n}(h_1)} e^{\omega(\rho_n)(h_2)} 
 \\ & \quad  \prod_{i=1}^N  \vartheta_{3}(wy_iq^{\delta_i \rho_n}z_{\omega(i)}q^{g-a_n \delta_{i>m}} ; q), 
\end{split}
\end{equation}
with the standard theta function
\begin{equation}
\begin{split}
\vartheta_3( z;q) &=  \prod_{r=1}^\infty (1+zq^{r-1/2}) ( 1- q^r) (1+z^{-1} q^{r-1/2}).
\end{split}
\end{equation}
In the last line we used the usual transformation behavior of theta functions under translation of the Jacobi variable. 
Here we defined $\delta_{i>m}$ to be equal to one if $i>m$ and zero otherwise. We also defined the symbol $\sim$ for $q$-series, meaning that
\[
f(q) \sim g(q) \quad \leftrightarrow \qquad  \exists \ \Delta \in \mathbb C: \ f(q) = q^\Delta g(q).
\]

We specialize $h_1 =  -x\tau +h $ as before, that is $\delta_i(h) = 0$ for $m+1\leq 1 \leq N$.
The specialized numerator is then 
\begin{equation}
\begin{split}
\text{Num} \sim &  \ w^{Ng-na_n} \prod_{a=1}^n z_{\omega(m+a)}^{-a_n}y_i^{-a_n}  e^{\omega(\rho_n)(h_2)}   \prod_{i=m+1}^N  \vartheta_{3}(wz_{\omega(i)}q^{g-a_n} ; q) \\ & \prod_{i=1}^m  \vartheta_{3}(wy_iz_{\omega(i)}q^{g} ; q). 
\end{split}
\end{equation}

We introduce
\[
A = \frac{m}{N}\frac{\gamma}{N}+\delta, \qquad  B=n\frac{\gamma}{N}-\frac{\delta}{N}.
\]
These are orthogonal on each other, and we are interested in the constant coefficient with respect to $A$. Write $a=e^{\alpha A(h+u)}, b=e^{\beta B(h+u)}$ and $c:= b^{n+m/N^2}$. We fix $\alpha = \frac{N^2}{m^2+mnN^2}$ and $\beta = \frac{N^2}{mn+n^2N^2}$,
so that
\[
wx^{n/N} = a b^{n-n/N^2}=acb^{-1/N}, \quad wx^{-m/N} = b^{n+m/N^2}= c, \quad x = a b^{-1/N}.
\]
 Then 
\[
wy_i = \begin{cases} wx_ix^{n/N}  = x_i a c b^{-1/N} & 1\leq i \leq  m, \\ 
wx_ix^{-m/N} =x_i c =c & m+1 \leq i \leq  N. \end{cases}
\]

\noindent {\em Step 4: Meromorphic Jacobi form decomposition}

Let $1\leq i \leq  m$.
We are interested in the decomposition of
\[
M_i := \prod_{r=1}^\infty \frac{ (1+wy_iz_{\omega(i)}q^{r-1/2+g}) (1+w^{-1}y_i^{-1}z_{\omega(i)}^{-1} q^{r-1/2-g})(1-q^r)}{(1- xx_iq^{r+d-1/2})(1- x^{-1}x_i^{-1} q^{r-d-1/2})}.
\]

We need the identity (A.2) of \cite{CR}
\[
\prod_{r=1}^\infty \frac{(1-uvq^{r-1})(1-q^r)^2(1-u^{-1}v^{-1}q^{r})}{(1+uq^{r-1}))(1+vq^{r-1}))(1+u^{-1}q^{r}))(1+v^{-1}q^{r})} =  \sum_{s \in \mathbb Z} \frac{(-1)^s u^s }{1+vq^s}
\]
which holds for $|q| < |u| <1$. Set 
$ u = -xx_i q^{d+\frac{1}{2}} = -ab^{-1/N}x_iq^{d+\frac{1}{2}}, v = cz_{\omega(i)}q^{g-d}$.  
We will fix $d=\frac{n}{2}-2$ in a moment. Recall that the domain of the character is $|(xx_i)^{\pm 1} q^{\frac{n+1}{2}}| <1$. In order for the decomposition now to be valid we restrict this domain to the subdomain $|q| < |xx_i q^{\frac{n+1}{2}-2}| < 1$. 
The decomposition of the meromorphic Jacobi form follows:
\begin{equation}
\begin{split}
M_i 
&=  \prod\limits_{r=1}^\infty (1+cz_{\omega(i)}q^{r +g-d-1}  )(1-q^r)^{-1} (1+c^{-1}z_{\omega(i)}^{-1}q^{r+d-g}   ) 
\\ & \quad \sum_{s\in \mathbb Z}\frac{a^sx_i^s b^{-s/N} q^{(d+1/2)s}}{1+cz_{\omega(i)}q^{s+g-d}  } \\
&\sim  \eta^{-2}\vartheta_3(cz_{\omega(i)}q^{g-d-1/2}; q )\sum_{s\in \mathbb Z} \frac{a^sx_i^s b^{-s/N} q^{(d+1/2)s}}{1+cz_{\omega(i)}q^{s+g-d}  } \\
&\sim (cz_{\omega(i)})^{d+\frac{1}{2}-a_n}  \eta^{-2} \vartheta_3(cz_{\omega(i)}q^{g-a_n}; q )\sum_{s\in \mathbb Z} \frac{a^sx_i^s b^{-s/N} q^{(d+1/2)s}}{1+cz_{\omega(i)}q^{s+g-d}  }.
\end{split}
\end{equation}
Let $\widetilde{\text{Num}} :=
\Pi(h_2, q)\Pi_m(x, q)\eta(q)^{n+1}\text{ch}[H(A)]$ be the numerator of the Euler-Poincar\'e character.
\begin{equation}\nonumber
\begin{split}
&  \widetilde{\text{Num}} 
= (-1)^{md'} x^{-md} w^{Ng} c^{-na_n} \Sigma \\
& \Sigma \sim \sum_{\omega \in W} \epsilon(\omega) e^{\omega(\rho_n)(h_2)}
\prod_{a=1}^n z_{\omega(m+a)}^{-a_n}     \prod_{i=m+1}^N  \vartheta_{3}(wz_{\omega(i)}q^{g-a_n} ; q) \prod_{i=1}^m  M_i \\
&\sim \frac{c^{m(d+\frac{1}{2}-a_n)}}{\eta^{2m}}\sum_{\omega \in W}  \epsilon(\omega) e^{\omega(\rho_n)(h_2)} \prod_{i=1}^m (z_{\omega(i)})^{d+\frac{1}{2}} 
\prod_{a=1}^N z_{\omega(a)}^{-a_n}     \prod_{i=1}^N  \vartheta_{3}(cz_{\omega(i)}q^{g-a_n} ; q)  \\
&\quad\quad \prod_{i=1}^m  \sum_{s_i\in \mathbb Z} \frac{a^{s_i}x_i^{s_i} b^{-{s_i}/N} q^{(d+1/2)s_{i}}}{1+cz_{\omega(i)}q^{s_i+g-d}  } \\
&\sim \frac{c^{m(d+\frac{1}{2}-a_n)}}{\eta^{2m}}\sum_{\omega \in W}  \epsilon(\omega)e^{\omega(\rho_n)(h_2)} \prod_{i=1}^m  (z_{\omega(i)})^{d+\frac{1}{2}}
    \prod_{i=1}^N  \vartheta_{3}(cz_{\omega(i)}q^{g-a_n} ; q)  
 \\ & \quad\quad \prod_{i=1}^m   \sum_{s_i\in \mathbb Z} \frac{a^{s_i}x_i^{s_i} b^{-{s_i}/N} q^{(d+1/2)s_{i}}}{1+cz_{\omega(i)}q^{s_i+g-d}  }\\
&\sim \frac{c^{m(d+\frac{1}{2}-a_n)}}{\eta^{2m}}  \prod_{i=1}^N  \vartheta_{3}(cz_{i}q^{g-a_n} ; q)\sum_{\omega \in W}  \epsilon(\omega)e^{\omega(\rho_n)(h_2)} \prod_{i=1}^m  (z_{\omega(i)})^{d+\frac{1}{2}}
 \\ & \quad\quad  \prod_{i=1}^m   \sum_{s_i\in \mathbb Z} \frac{a^{s_i}x_i^{s_i} b^{-{s_i}/N} q^{(d+1/2)s_{i}}}{1+cz_{\omega(i)}q^{s_i+g-d}  } ,\\
\end{split}
\end{equation}
where we used
\begin{equation}\nonumber
\begin{split}
\prod_{a=1}^N z_{\omega(a)}^{-a_n} &= 1, \  \text{and} \ 
 \prod_{i=1}^N  \vartheta_{3}(cz_{\omega(i)}q^{g-a_n} ; q) =  \prod_{i=1}^N  \vartheta_{3}(cz_{i}q^{g-a_n} ; q).
\end{split}
\end{equation}
Set $g=d+a_m$.
Then the prefactor becomes $x^{-md} w^{Ng} = c^{dN} a^{ma_m} b^{nNa_m}$.
The multiplicity of the affine $\gg\gl_m$ is the coefficient corresponding to $a^0e^{\rho_m(h)} = a^0 \prod_{i=1}^m x_i^{\frac{m+1-2i}{2}}$. Recall that $x_1\dots x_m=1$ and so we need to consider the summand with $s_i = \frac{m+1-2i}{2}-a_m$,  
that is 
\begin{equation}\nonumber
\begin{split}
\eta(m)\Pi(q, h_2)C &\sim (-1)^{md'} b^{ma_m/N} c^{dN}  b^{nNa_m}c^{m(d+\frac{1}{2}-a_n)} c^{-na_n} 
\\ &\quad\prod_{i=1}^N  \frac{\vartheta_{3}(cz_{i}q^{d+a_m-a_n} ; q)}{\eta(q)}  \sum_{\omega \in W}  \epsilon(\omega)e^{\omega(\rho_n)(h_2)} \prod_{i=1}^m 
       \frac{ (z_{\omega(i)})^{\frac{2d+1}{2}}  }{1+cz_{\omega(i)}q^{s_i+a_m}  } 
 \\&\sim (-1)^{md'}  c^{N(d+a_m-a_n)}c^{m(2d+\frac{1}{2})} 
 \\ &\quad  \prod_{i=1}^N  \frac{\vartheta_{3}(cz_{i}q^{d+a_m-a_n} ; q)}{\eta(q)} \sum_{\omega \in W}  \epsilon(\omega)e^{\omega(\rho_n)(h_2)} \prod_{i=1}^m 
       \frac{ (z_{\omega(i)})^{\frac{2d+1}{2}}  }{1+cz_{\omega(i)}q^{s_i+a_m}  }
 \\&\sim (-1)^{md'}   c^{N(a_m-1/2)} 
 \\ &\quad \prod_{i=1}^N  \frac{\vartheta_{3}(cz_{i}q^{a_m-\frac{1}{2}}; q) }{\eta(q)}
 \sum_{\omega \in W}  \epsilon(\omega)e^{\omega(\rho_n)(h_2)} \prod_{i=1}^m 
       \frac{ (cz_{\omega(i)})^{\frac{2d+1}{2} }  }{1+cz_{\omega(i)}q^{s_i+a_m}  }
\\ &\sim (-1)^{md'}    c^{N(a_m-1/2)} 
 \\ &\quad \prod_{i=1}^N  \frac{\vartheta_{3}(cz_{i}q^{a_m-\frac{1}{2}}; q)}{\eta(q)} 
 \sum_{\omega \in W}  \epsilon(\omega)e^{\omega(\rho_n)(h_2)} \prod_{i=1}^m 
       \frac{ (cz_{\omega(i)})^{\frac{2d+1}{2}}   }{1+cz_{\omega(i)}q^{s_i+a_m}  }. \\
\end{split}
\end{equation}
We now set $2d=n-4$.

\noindent  {\em Step 5: Denominator identity of $\mathfrak{sl}(N|m)$}

Recall the denominator identity of the finite dimensional Lie superalgebra $\mathfrak{sl}(N|m)$, \cite[Thm. 2.1]{KWVI}.
 Consider the lattice $\mathbb Z \epsilon_1 \oplus \dots \oplus \mathbb Z \epsilon_N \oplus \mathbb Z \mu_1 \oplus \dots \oplus \mathbb Z \mu_m$ with $\epsilon_i\epsilon_j = \delta_{i, j}$, $\mu_i \epsilon_j=0$ and $\mu_i\mu_j=-\delta_{i, j}$. Then we choose the sets 
\begin{equation}
\begin{split} 
\Delta_0^+ &= \{ \epsilon_i - \epsilon_j | 1\leq  i < j \leq  N\}\  \cup \ \{ \mu_i - \mu_j | 1\leq  i < j \leq  m\} \\
 \Delta_1^+ &= \{ \epsilon_i - \mu_j | 1 \leq i \leq j \leq m \} \ \cup \ \{ \mu_i - \epsilon_j | 1 \leq i  \leq  m,  i < j \leq N \}.
\end{split}
\end{equation}
 These can be identified with the set of even and odd positive roots of $\mathfrak{sl}(N|1)$. Let $\rho_0 = \frac{1}{2} \sum_{\alpha \in \Delta_+^0} \alpha$ and $\rho_1 = \frac{1}{2} \sum_{\alpha \in \Delta_+^1} \alpha$ be the even and odd part of the Weyl vector $\rho = \rho_0 - \rho_1$. Then 
\begin{equation}
\begin{split}
\rho_0 &= \frac{1}{2}\sum_{i=1}^N (N+1-2i)\epsilon_i  + \frac{1}{2}\sum_{i=1}^m (m+1-2i)\mu_i \\
\rho_1 &=  \frac{1}{2}\sum_{i=1}^m (m+2-2i)\epsilon_1 - \frac{m}{2}\sum_{i=m+1}^n \epsilon_i +\frac{1}{2}\sum_{i=1}^m (N-2i)\mu_i \\
\rho&= \frac{n-1}{2}\sum_{i=1}^m (\epsilon_i-\mu_i )+ \frac{1}{2}\sum_{i=1}^n(n+1-2i)\epsilon_{m+i}. 
\end{split}
\end{equation}
 
 Then the denominator identity reads
\begin{equation}
 \frac{e^{\rho_0 }\prod\limits_{\alpha \in \Delta_0^+}   (1 - e^{-\alpha}) }{e^{\rho_1 }\prod\limits_{\alpha \in \Delta_1^+}   (1 +  e^{-\alpha}) } = \sum_{\sigma \in S_N} \text{sgn}(\sigma)  \frac{e^{\sigma(\rho)}}{ \prod\limits_{i=1}^m 1 + e^{\mu_i -\sigma(\epsilon_i)}}.
\end{equation}
We set $e^{\epsilon_i}=z_i$ and $e^{\mu_i} = c^{-1}q^{(m+1-2i)/2}$ so that the identity becomes
\begin{equation}
\begin{split}
& \sum_{\omega \in W}  \epsilon(\omega)e^{\omega(\rho_n)(h_2)} \prod_{i=1}^m 
       \frac{ (cz_{\omega(i)})^{\frac{n-1}{2}}   }{1+c^{-1}z_{\omega(i)}q^{-s_i-a_m}  }
 \\ & =  \frac{\pi_N(z)\pi_m(q)}{c^{mN/2}\prod\limits_{i=1}^N\prod\limits_{j=1}^m{(1+z_icq^{\frac{m+1-2j}{2}}})} \\ 
       &= c^{-N(a_m-1/2)} \frac{\pi_N(z)\pi_m(q)}{\prod\limits_{i=1}^N\prod\limits_{1 \leq j \leq \left\lceil \frac{m}{2} \right \rceil}{(1+z_icq^{a_m + j -1}})\prod\limits_{1 \leq j' \leq \left\lfloor \frac{m}{2} \right \rfloor}{(1+z_icq^{a_m + j'}})} ,\\
       \pi_N(z) &= \prod_{1 \leq i < j \leq  N} ((z_jz_i^{-1})^{\frac{1}{2}} - (z_iz_j^{-1})^{\frac{1}{2}}) = (-1)^{m\left \lceil \frac{n}{2}\right \rceil} e^{\rho_N(h_2)} \prod_{\alpha\in \Delta_+} (1-e^{-\alpha}),\\
       \pi_m(q)&= \prod_{1 \leq i < j \leq  m} (e^{(\mu_j-\mu_i)/2} - e^{(\mu_i-\mu_j)/2} ) \sim \prod_{1 \leq i < j \leq  m} (1  -q^{i-j}).
\end{split}
\end{equation}
Recall that $d' = d+\frac{1}{2}-a_n = \frac{n+1}{2} -a_n$, so that $d' = 0 \mod 2$ if $n = 0, 3 \mod 4$ and $d'=1\mod 2$ if $n=1, 2 \mod 4$, i.e. $ (-1)^{md'}=(-1)^{m\left \lceil \frac{n}{2}\right \rceil}$.
Putting all together we get the desired answer
\begin{equation}\nonumber
\begin{split}
C &\sim \frac{ (-1)^{md'}}{\eta(q)^m\Pi(q, h_2)}    c^{N(a_m-1/2)} 
\\ &\qquad \prod_{i=1}^N  \frac{\vartheta_{3}(cz_{i}q^{a_m-\frac{1}{2}}; q)}{\eta(q)} 
 \sum_{\omega \in W}  \epsilon(\omega)e^{\omega(\rho_n)(h_2)} \prod_{i=1}^m 
       \frac{ (cz_{\omega(i)})^{\frac{2d+1}{2}}   }{1+cz_{\omega(i)}q^{s_i+a_m}  } \\
       &\sim \frac{(-1)^{md'}}{\eta(q)^m\Pi(q, h_2)}    c^{N(a_m-1/2)} \prod_{i=1}^N  \frac{\vartheta_{3}(cz_{i}q^{a_m-\frac{1}{2}}; q)}{\eta(q)} \\
&\qquad       c^{-N(a_m-1/2)} \frac{(-1)^{m\left \lceil \frac{n}{2}\right \rceil} e^{\rho_N(h_2)} \prod\limits_{\alpha\in \Delta_+} (1-e^{-\alpha})}{\prod\limits_{i=1}^N\prod\limits_{1 \leq j \leq \left\lceil \frac{m}{2} \right \rceil}{(1+z_icq^{a_m + j -1}})\prod\limits_{1 \leq j' \leq \left\lfloor \frac{m}{2} \right \rfloor}{(1+z_icq^{a_m + j'}})}\\
      &\sim   \frac{e^{\rho_N(h_2)} \prod\limits_{\alpha\in \Delta_+} (1-e^{-\alpha}) }{\Pi(q, h_2)}\ \  \frac{\prod\limits_{1 \leq i < j \leq  m} (1  -q^{i-j})}{\eta(q)^m} 
      \\ & \qquad  \frac{ \prod\limits_{i=1}^N  \vartheta_{3}(cz_{i}q^{a_m-\frac{1}{2}}; q)\eta(q)^{-1} 
     } {\prod\limits_{i=1}^N\prod\limits_{1 \leq j \leq \left\lceil \frac{m}{2} \right \rceil}{(1+z_icq^{a_m + j -1}})\prod\limits_{1 \leq j' \leq \left\lfloor \frac{m}{2} \right \rfloor}{(1+z_icq^{a_m + j'}})}\\
     &\qquad \sim \text{ch}\left[ \cW^{-k-m+1}(\gs\gl_{m|N}, f_{m|N}) \right].
\end{split}
\end{equation}

\end{proof}

%%  The bibliography

\end{document}